\documentclass[12pt,english]{amsart}
\usepackage{amsmath,amssymb}
\usepackage{bm,color,hyperref}
\usepackage{graphicx} % Required for inserting images
\usepackage{tikz}
    \usetikzlibrary{cd}
\usepackage{standalone}
\usepackage{enumitem} %Required to change enumeration labels

\newtheorem{theorem}{Theorem}[section]
\newtheorem*{theorem*}{Theorem}
\newtheorem{corollary}[theorem]{Corollary}
\newtheorem{lemma}{Lemma}[section]
\newtheorem*{lemma*}{Lemma}
\newtheorem{proposition}{Proposition}[section]
\newtheorem*{proposition*}{Proposition}

\theoremstyle{definition}
\newtheorem{remark}{Remark}[section]
\newtheorem{definition}{Definition}
\newtheorem*{definition*}{Definition}

\newcommand{\HALT}{\end{document}}

\newcommand{\I}{\mathrm{i}}

\newcommand{\conv}{\operatorname{conv}}
\newcommand{\ext}{\operatorname{ext}}

\let\emptyset\varnothing    %%with a nicer one.

\title{Sofic systems and extreme points of self-affine limit sets}

\begin{author}[S. Silvestri]{Stefano Silvestri}
\address{ %
Rome\\
Italy}
\email{silvestri.dev@gmail.com}
\end{author}
\begin{author}[R. A. P\'erez]{Rodrigo A. P\'erez}
\address{ %
IU Indianapolis, Dept. of Math. Sciences\\
402 N Blackford Street\\
Indianapolis, Indiana 46202\\
 United States }
\email{rodperez@iu.edu}
\end{author}

\begin{document}
\begin{abstract}
    Given any iterated function system (IFS) on the plane whose affine transformations share the same contracting matrix $T$, we characterize the set of extreme points of the attractor $A$ (along the convex hull $\conv(A)$) in terms of a directed graph on the edges of $\conv(A)$. This graph corresponds to a sofic system on the coding space associated to the IFS. When $T$ is proper orthogonal or purely real, we give conditions for $\conv(A)$ to be a finite polygon and, under additional symmetry of the transformations, we give conditions for $A$ itself to coincide with $\conv(A)$. 
\end{abstract}
\maketitle

\section{Introduction}
\label{sect:intro}
Consider a finite collection of $n \geq 2$ plane transformations 
\[F_j(\vec{x})=T\vec{x}+\vec{v}_j \quad (0 \leq j \leq n-1),\]
where $T$ a non-singular $2 \times 2$ real matrix, and all $\vec{v}_j \in \mathbb{R}^2$. Since the matrix $T$ is common to all maps, the collection $\{F_j\}$ is \emph{homogeneous}. We will consider the separate cases of $T$ being proper orthogonal (equivalent to complex multiplication), diagonalizable, and skew.

We require that each $F_j$ is a contraction which means that the singular values of $T$ are all less than $1$.
Then, $\{F_j\}$ constitutes an \emph{Iterated Function System} (IFS for short; see general definition in Section~\ref{sect:background}). The classical theorem of Hutchinson states that there exists a unique non-empty compact set $A$, known as the \emph{limit} or \emph{attractor set}, that satisfies
\[A = F_0(A) \cup F_1(A) \cup \ldots \cup F_{n-1}(A).\]

We are concerned in describing the \emph{extreme points} of $A$, that is, those points in $A$ that lie on the boundary of its convex hull. To give a brief overview of our results and to better understand how they fit in the literature, consider the following simple example: Figure \ref{fig:IntroExample} depicts the famous Twin Dragon fractal, whose IFS consists only of the two maps $F_0$ and $F_1$ that rotate the plane $45^{\circ}$ and scale it by $1/2$, and then translate left or right by one unit. The convex hull is an octagon and, accordingly, the extreme points encircle $A$ in a chain of eight straight Cantor sets. Under the two maps $F_0$ and $F_1$, these Cantors map onto each other in a somewhat tangled fashion. In particular, three iterations of $F_1$ map the top edge of the convex hull to the bottom-left edge. This shorter segment then covers two halves of the bottom Cantor set after one more iteration of either $F_0$ or $F_1$. A symmetric chain of iterations describes the top Cantor set as two copies of the bottom. In this situation, neither top nor bottom nor any other Cantor sets is described in terms of itself, but rather in terms of each other. Such behavior constitutes an example of a \emph{Directed Graph Iterated Function System} (DGIFS for short).
\begin{figure}
    \centering
    \includegraphics[width=0.5\linewidth]{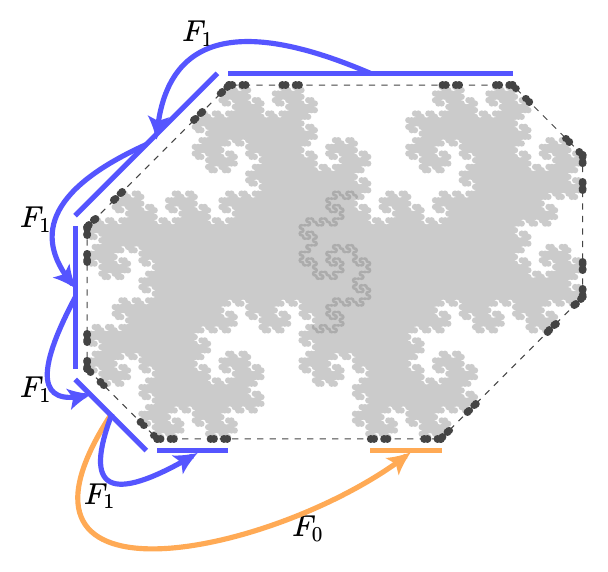}
    \caption{On every face of the convex hull, the extreme points of the Twin Dragon form a Cantor set. Some images of these under $F_0$ or $F_1$ \emph{remain} on the boundary, and the picture demonstrates a chain of such maps that renders the bottom Cantor set as two iterated copies of the top. A symmetric chain renders the top Cantor set as two iterated copies of the bottom, but those arrows are not displayed for the sake of visual economy.}
    \label{fig:IntroExample}
\end{figure}

Describing extreme points in the attractor of IFS is not a new objective. Vass \cite{V18} offers compelling motivation for studying convex hulls of attractors of IFS, and provides an algorithm to construct them: an initial special finite collection of points in the attractor is iterated and, at each stage, tested for keeping or discarding. His method applies to homogeneous and non-homogeneous maps, but only with proper orthogonal matrices; thus, there is partial overlap with our work in Section~\ref{sect:complex}, but whereas Vass' construction is procedural, ours produces a concrete combinatorial description. On the same class of IFS, Tetenov and Davydkin \cite{TD02, DT05} can guardantee that for non-homogeneous IFS with proper orthogonal matrices, the convex hull of the attractor is a finite sided polygon if the unitary parts of the matrices form a finite group, i.e. if the unitary parts are rotation by a rational multiple of $2\pi$. However, they do not provide an exact number of sides of the convex hull nor an explicit description of which points of the attractor are extreme. In a different direction, a non-algorithmic characterization of extreme points has already been found \cite{HS16,HS17,CW19,R25} in some of the cases encompassed by our more general theory.

Our contribution is to characterize the set of extreme points for \emph{any} homogeneous IFS as the unique set of attractors of a DGIFS. For well-known and well-studied IFS the DGIFS is explicitly given; for arbitrary homogeoneous IFS the system of equations follow a specific pattern in terms of the transformations $\{F_j\}$ but it has to be recovered case by case. Our approach in the proofs is geometrical and independent in method from those in the cited articles, though in some cases there is overlap in the underlying ideas.

We open the article with a brief review of background information in Section \ref{sect:background}. It sets terminology, and introduces a new object essential to the proofs. Then, Section \ref{sect:complex} deals with the self-similar IFS, i.e. the case when the eigenvalues of $T$, the linear part of each affine map, are complex. Section \ref{sect:real} and \ref{sect:jordan} focus on the self-affine IFS, i.e. when the eigenvalues of $T$ are real. Specifically, section \ref{sect:real} considers the case of eigenvalues with two linearly independent eigenvectors; and Section \ref{sect:jordan} covers equal eigenvalues with only one eigendirection. Each of the main sections is accompanied by illustrative examples. Finally, Section \ref{sect:future} concludes with a discussion of how the results may extend to the non-homogeneous setting, pointing to future directions.

\section{Background}
\label{sect:background}
\subsection{Affine transformations}
For an affine transformation of the plane $T\vec{x}+\vec{v}$, the matrix $T$, can be characterized using the Jordan canonical form: let $P$ be the matrix whose columns are the eigenvectors and let $J$ be the Jordan matrix, then \[P^{-1}TP=J.\]
If the eigenvalues of $T$ are
\begin{enumerate}
    \item complex $a\pm b\I$, then $J=\begin{pmatrix}a&-b\\b&a\end{pmatrix}$ or $J=\begin{pmatrix}a&b\\-b&a\end{pmatrix}$
    \item real $\mu\neq\eta$, then $J=\begin{pmatrix}\mu&0\\0&\eta\end{pmatrix}$
    \item equal $\mu=\eta$, then $J=\begin{pmatrix}\eta&1\\0&\eta\end{pmatrix}$
\end{enumerate}
Therefore we can focus on affine transformations whose matrix is already in Jordan form. 
We will call the first of these the \emph{complex case}, the second the \emph{real case}, and the last the \emph{Jordan block case} following \cite{HS16,HS17} .

\subsection{Iterated Function Systems}
Let $X=\mathbb{R}^2$ or $X=\mathbb{C}$ and consider the complete metric space $(X,d)$. A map $f: X\to X$ is called a \emph{contraction} with factor $r<1$ if $d(f(x),f(y))\leq r d(x,y)$ for any $x,y\in X$.

A famous theorem of Hutchinson \cite{H81} states that given a collection of contractions $\{f_j\}_{j=0}^{n-1}$ on $(X,d)$, then there exists a unique non-empty compact set, called a \emph{limit} or \emph{attractor set}, that satisfies \[A=\bigcup_{0\leq j< n}f_j(A).\]
The set of functions $\{f_j\}_{j=0}^{n-1}$ is called an \emph{Iterated Function System} (IFS).

More can be said about the limit set: $A$ is the closure of the set of fixed points of every finite composition of the maps $f_j$. In particular, Hutchinson proved the following: denote by $\Sigma=\{0,1,\cdots,n-1\}^\mathbb{N}$ the space of sequences on $n$-symbols and define the left shift $\tau_j:\Sigma\to\Sigma$ with $\tau_j(w_0w_1w_2\cdots)=jw_0w_1w_2\cdots$, then
\begin{theorem}
\label{thm:CodeSpaceProjection}
    Given the IFS of contractions $\{f_j\}_{j=0}^{n-1}$, then there is a continuous map $\pi:\Sigma\to X$ with $f_j\circ\pi=\pi\circ\tau_j$ for $j=0,\ldots,n-1$ and $\pi(\Sigma)$ is the attractor $A$ of the IFS.
\end{theorem}
As a consequence, for any infinite sequence $\omega=w_0w_1\cdots$ with $w_j\in\{0,1,\cdots,n-1\}$, the projection $\pi(\omega)$ is a point in the limit set which corresponds to the composition $f_\omega(x)=\lim_{k\to\infty}(f_{w_0}\circ \cdots \circ f_{w_k})(x)$ where the choice of $x$ does not matter.

A \emph{Directed Graph Iterated Function System} (DGIFS), sometimes called Graph Directed IFS or Graph IFS or Recurrent IFS, is a generalization of an IFS. They were formulated in full generality in \cite{MW88}, but earlier versions with more restrictive assumptions appeared in \cite{D82, G87, B89a, B89b}.

A \emph{directed graph} $G:=(V,E)$, consists of a finite set $V$ of vertices and a finite set $E$ of directed edges with loops and multiple edges allowed. Let $E_{uv}\subset E$ be the set of edges from the vertex $u$ to the vertex $v$.

A DGIFS on $X$ is a finite collection of distinct contracting functions $\{f_e:X_u\to X_v~|~e\in E,~u,v\in V\}$, where $X_u$ is a copy of $X$ associated to the vertex $u$. It is further required that the out degree of every vertex $u\in V$ to be greater or equal to $1$. For such a DGIFS there exists \cite[Theorem 4.3.5]{Edgar1990} a unique list of nonempty compact sets $\{B_u\subset X_u\}_{u\in V}$ such that $$B_u = \bigcup_{v\in V}\bigcup_{e\in E_{uv}} f_e(B_v),\quad \forall u\in V.$$ Each compact set $B_u$ is called a DG-attractor and $\{B_u\}_{u\in V}$ is called the list (or vector) of attractors of the DGIFS. We remark that each $B_u$ lives in its own copy $X_u$ of $X$, however in this article we picture them as sets in the same space $X$.

The standard IFS corresponds to the special case of a DGIFS with only a single vertex with as many self-loops as there are functions defining the IFS.

Observe that the defining equation of the limit set of an IFS $A=\bigcup_{k=0}^{n-1} f_k(A)$ can also be rewritten as a system of $n$ equations: if we denote by $A_i=f_i(A)$ then 
\[A_i=f_i(A_0)\cup f_i(A_1)\cup\cdots\cup f_i(A_{n-1}),\qquad i=0,\ldots, n-1.\]
This is an example of a DGIFS whose associated graph has $n$ vertices. Bandt in \cite{B89a} considered the above system of equations by deleting some terms on the right hand side of some or all equations. The resulting DG-attractors are then subsets of $A$. Bandt has remarked that such a construction is equivalent to a subshift of finite type on the space of sequences in $n$-symbols, $\Sigma=\{0,1,\cdots,n-1\}^\mathbb{N}$. Recall that Theorem \ref{thm:CodeSpaceProjection} provides a way to represents points in the attractor with infinite sequences from the code space $\Sigma$. Thus suppose $S\subset\Sigma$ is a set characterized by a finite rule that determines the sequences belonging to it. If such set is closed and invariant under $\tau_i$ for some $i$ then it is called a \emph{subshift of finite type}, abbreviated as SFT.

In \cite{B89b}, instead, the focus was in the more general case of having a DGIFS with $m$ equations where $m$ could be larger or smaller than the number of maps defining the IFS. The resulting DG-attractors are again subsets of $A$, but now such construction is equivalent to a sofic system.

A \emph{sofic system} is a generalization of a SFT. In particular, a sofic system is characterized by a finite number of distinct sequences $\nu$ that can follow any other finite sequence $\omega=w_0w_1\cdots w_j$.

\subsection{Convex hull and extreme points}
A set $S\subseteq\mathbb{R}^2$ is \emph{affine} if for any two elements of $S$ the line through them is contained in $S$. A set $C\subseteq\mathbb{R}^2$ is \emph{convex} if for any two elements $\vec{x},\vec{y}\in C$ the line segment $[\vec{x},\vec{y}]=\{(1-t)\vec{x}+t\vec{y}~:~t\in[0,1]\}$ is contained in $C$. The \emph{convex hull} of an arbitrary set $S$, denoted by $\conv(S)$, is the smallest convex set containing $S$. 

A point $\vec{p}\in\mathbb{R}^2$ is an \emph{extreme point} of $S$ if $\vec{p}\in\conv(S)$  and there do not exist distinct $\vec{x},\vec{y}\in\conv(S)$ and $tt\in(0,1)$ with $\vec{p}=t\vec{x}+(1-t)\vec{y}$. Write $\ext(S)$ for the set of extreme points of $S$. 
Following the usage in~\cite{V18,CW19}, we additionally allow ``extreme point'' to include points $S\cap\partial\conv(S)$; we reserve the term \emph{corner point} for a genuine vertex.

The behavior of convex sets and extreme points is characterized by the following:
\begin{lemma}[\cite{B92}]
\label{lem:extreme}
Let $A,C_1,\ldots,C_n$ be susbsets of $\mathbb{R}^2$ with $A$ compact and let $F$ be an affine transformation. Then
    \begin{itemize}
        \item 
        $\bigcup_{i=1}^n \conv(C_i)\subseteq \conv\left(\bigcup_{i=1}^n C_i\right)$
        \item $\ext\left(\bigcup_{i=1}^n C_i\right)\subseteq \bigcup_{i=1}^n \ext(C_i)$
        \item $\conv(F(A))=F(\conv(A))$
        \item $\ext(F(A))\subseteq F(\ext(A))$
    \end{itemize}
\end{lemma}
\begin{theorem}[\cite{B92}]
\label{thm:extreme}
    Suppose $A\subseteq\mathbb{R}^2$ satisfies $A=\bigcup_{i=1}^nF_i(A)$ where $F_1,\ldots,F_n$ are affine transformations. Then
    \begin{itemize}
        \item $\bigcup_{i=1}^nF_i(\conv(A))\subseteq \conv(A)$
        \item $\ext(A)\subseteq \bigcup_{i=1}^nF_i(\ext(A))$
    \end{itemize}
\end{theorem}
A consequence of the above theorem is that every extreme point of the attractor $A$ is necessarily the image of some other extreme point for some affine map $F_i$.

\subsection{Support function and bounding lines}
Given $d\in \mathbb{R}$ and a non-zero vector $\vec{b}\in \mathbb{R}^2$ then the line in the plane whose normal direction is $\vec{b}$ can be defined as \[ \mathcal{L}=\left\{\vec{x}\in\mathbb{R}^2~:~\langle \vec{b},\vec{x}\rangle=\vec{b}^{\mathrm T}\vec{x}=d\right\}.\] 

The \emph{support function} of a convex set $K\subset\mathbb{R}^2$ is defined by \[h_K(\vec{x})=\sup\{\langle \vec{x},\vec{y}\rangle~|~\vec{y}\in K\}.\] Such function is a convex function, since $h_K(\vec{x})$ is the pointwise supremum of convex functions: for fixed $\vec{x}$ the dot product $\langle \vec{x},\vec{y}\rangle$ is a convex function of $\vec{y}$. Moreover, it is positive homogeneous of degree $1$ since $h_K(k\vec{x})=kh_K(\vec{x})$ for any scalar $k\geq0$. Thanks to this property, it is possible to characterize $h_K$ by its values on the unit sphere $\mathbb{S}^{1}\subset\mathbb{R}^2$. 

The support function is naturally related to the geometric notion of ``tangency''. A \emph{supporting half-plane} to a closed convex set $K$ is a closed half-plane which contains $K$ and has a point of $K$ in its boundary. A \emph{supporting line} to a closed convex set $K$ is a line which is the boundary of a supporting half-plane. Thus, a supporting line to $K$ is associated with a dot product which achieves its maximum on $K$: for $\theta\in\mathbb{S}^1$ the supporting line to $K$ with normal $\theta$ is defined
\[\mathcal{L}(\theta):=\left\{\vec{y}\in\mathbb{R}^2~:~\langle \theta,\vec{y}\rangle=h_K(\theta)\right\}.\]
Geometrically, assuming $\vec{0}\in K$, $h_K(\theta)$ represents the distance from the origin to the supporting line $\mathcal{L}(\theta)$ to $K$ in the direction of $\theta\in\mathbb{S}^{1}$. 

Let us identify the unit circle $\mathbb{S}^1$ with the interval $[0,2\pi]$. Let $A$ be the limit set of the IFS generated by $n\geq2$ affine transformations $F_k(\vec{x})=L_k\vec{x}+\vec{v}_k$, where $L_k$ are $2\times2$ non-singular real matrix for every $0\leq k<n$. From the above discussion there exists at least one supporting line $\ell$ to $A$ at each extreme point, since $\ext(A)\subset A\subset\overline{\conv(A)}$. 

Since $A$ satisfies $A=\bigcup_{k=0}^{n-1}F_k(A)$, we claim, in fact, that there is at least one supporting line to $A$, that intersects both $F_k(A)$ and $F_j(A)$ for some $0\leq k\neq j<n$.
Given that the limit set $A$ is a union of compact sets, the support function to $A$ in a given direction is equivalent to the maximum between the support function to each set in the union: for $\theta\in [0,2\pi]$ \[h_A(\theta)=\sup_{\vec{x}\in A}\langle\theta,\vec{x}\rangle=\max_{0\leq k<n}h_{F_k(A)}(\theta)=\max_{0\leq k<n}\sup\left\{\langle \theta,\vec{x}\rangle~:~\vec{x}\in F_k(A)\right\}.\] Observe that by definition we can write
\begin{align*}
h_{F_k(A)}(\theta) &=& \sup\left\{\langle \theta,\vec{x}\rangle~:~\vec{x}\in F_k(A)\right\}= \sup\left\{\langle \theta,F_k(\vec{x})\rangle~:~\vec{x}\in A\right\}\\
&=& \sup\left\{\langle \theta,L_k\vec{x}\rangle~:~\vec{x}\in A\right\}+\langle \theta,\vec{v}_k\rangle\\
&=& \sup\left\{\langle L_k^{\mathrm T}\theta,\vec{x}\rangle~:~\vec{x}\in A\right\}+\langle \theta,\vec{v}_k\rangle= h_A(L_k^{\mathrm T} \theta) + \langle \theta, \vec{v}_k \rangle.
\end{align*}

\begin{definition}
    Let $A$ be the attractor of an IFS generated by $n$ affine contractions $\{F_k(\vec{x}) = L_k \vec{x} + \vec{v}_k\}_{k=0}^{n-1}$ on $\mathbb{R}^2$. For $\theta \in [0,2\pi]$, define the \emph{dominance region} of the $k$th map as
    \[D_k := \left\{\theta \in [0,2\pi]~:~h_{F_k(A)}(\theta) = h_A(\theta)\right\}.\]
    A supporting line to $A$ with outward normal $\theta_* \in D_k \cap D_j$ for some $k \neq j$ is called a \emph{bounding line}.
\end{definition}

\begin{proposition}
\label{prop:boundinglines}
    Let $A$ be the attractor of an IFS generated by $n \geq 2$ affine contractions $\{F_k(\vec{x}) = L_k \vec{x} + \vec{v}_k\}_{k=0}^{n-1}$ on $\mathbb{R}^2$. Then the following hold.
    \begin{enumerate}
        \item[(i)] There exist at least two bounding lines.     
        \item[(ii)] For distinct matrices $L_k$, a bounding line to $A$ in the direction $\theta_* \in D_k \cap D_j$ is the solution to the equation
        \[h_A(L_k^{\mathrm T} \theta_*)+\langle \theta_*, \vec{v}_k \rangle=h_A(L_j^{\mathrm T} \theta_*)+\langle\theta_*,\vec{v}_j\rangle.\]
        \item[(iii)] If $L_k = L$ for all $k$, the dominance regions reduce to
        \[D_k = \left\{\theta \in [0,2\pi]~:~\langle\theta,\vec{v}_k\rangle\geq\langle\theta,\vec{v}_j\rangle,~\forall0\leq k\neq j<n\right\}.\]
        In this case, $D_k\cap D_j\neq\emptyset$ if and only if the segment $\vec{v}_j+t(\vec{v}_k-\vec{v}_j)$ for $t\in[0,1]$ is an edge of $\conv(\{\vec{v}_k\}_{k=0}^{n-1})$ and $\theta_*\in D_k\cap D_j$ is perpendicular to $(\vec{v}_j-\vec{v}_k)$.
    \end{enumerate}
\end{proposition}
\begin{proof}
    Since $A = \bigcup_{k=0}^{n-1} F_k(A)$ and $h_A(\theta) = \max_k h_{F_k(A)}(\theta)$, every $\theta \in [0,2\pi]$ belongs to at least one $D_k$, so the union of $D_k$ covers $\mathbb{S}^1$. Each $D_k$ is closed as it is the preimage of $\{0\}$ under the continuous non-negative function $h_A(\theta)-h_{F_k(A)}(\theta)$.
    
    We now show that there is no $k$ for which $D_k=\mathbb{S}^1$. By way of contradiction, suppose $D_k=\mathbb{S}^1$ for some $0\leq k<n$, so that $h_A(\theta)=h_A(L_k^{\mathrm T}\theta)+\langle\theta,\vec{v}_k\rangle$ holds for all $\theta\in[0,2\pi]$. Applying this identity $m$ times to the right hand side of the equation gives
    \[h_A(\theta)=h_A\left((L_k^{\mathrm T})^m\theta\right)+\sum_{j=0}^{m-1}\left\langle (L_k^{\mathrm T})^j\theta,\vec{v}_k\right\rangle.\]
    Since $F_k$ is a contraction, then $\|(L_k^{\mathrm T})^m\|\to 0$ as $m\to\infty$. Because the support function is positive homogeneous and $A$ is compact, $\lvert h_A(t)\rvert\leq\operatorname{diam}(A)\cdot\lvert t\rvert$ for $t\in\mathbb{R}^2$ and, thus, the first term in the equation above vanishes in the limit. The series converges absolutely since $\langle(L_k^{\mathrm T})^j\theta,\vec{v}_k\rangle=\langle\theta,L_k^j\vec{v}_k\rangle$ and $\sum_{j=0}^{\infty} \|(L_k)^j\|\leq(I-\|L_k\|)^{-1}<\infty$, so taking the limit yields 
    \[h_A(\theta)=\sum_{j=0}^{\infty}\langle\theta,L_k^j\vec{v}_k\rangle=\left\langle \theta,\sum_{j=0}^{\infty}L_k^j\vec{v}_k\right\rangle=\left\langle\theta,(I - L_k)^{-1}\vec{v}_k\right\rangle.\]
	A support function of the form $h_A(\theta)=\langle\theta,\vec{p}\rangle$ for a fixed $\vec{p}\in\mathbb{R}^2$ implies $\overline{\conv(A)}$ is equal to the singleton set $\{\vec{p}\}$, which is impossible. Hence $D_k\subsetneq[0,2\pi]$ for all $k$.

    Since $\{D_k\}_{k=0}^{n-1}$ is a finite closed cover of $\mathbb{S}^1$, which is connected, then the sets $D_k$ cannot be pairwise disjoint. Therefore, there exist pair of distinct integers $0\leq k,j<n$ with $D_k\cap D_j\neq\emptyset$. Any $\theta_*\in D_k\cap D_j$ satisfies
    \[h_{F_k(A)}(\theta_*)=h_{F_j(A)}(\theta_*),\]
    so the supporting line to $A$ with normal $\theta_*$ simultaneously supports both $F_k(A)$ and $F_j(A)$, and $A$ lies in the corresponding closed left half-plane. This gives a bounding line.

    To show that there is at least one other bounding line, consider the set of directions $\theta$ where for a specific $0\leq k<n$ the value of $h_{F_k(A)}(\theta)$ is strictly greater than that of $h_{F_j(A)}(\theta)$ for all $j\neq k$:
    \[U:=\bigcup_{0\leq k<n}\left\{\theta\in\mathbb{S}^1~:~h_{F_k(A)}(\theta)>h_{F_j(A)}(\theta)~\forall j\neq k\right\}=\bigcup_{0\leq k<n}D_k^\circ,\]
    where $D_k^\circ$ denotes the interior of the set. The set $U$ is open since it is the union of preimages of open intervals under the continuous functions $h_{F_k(A)}(t)-h_{F_j(A)}(t)$. By way of contradiction, suppose there was only one bounding line, then there exists only one pair of distinct integers $0\leq k,j<n$ for which $D_k\cap D_j\neq\emptyset$ and, in particular, $D_k\cap D_j=\{\theta_*\}$. Then \[U=\mathbb{S}^1\setminus(D_k\cap D_j)=\mathbb{S}^1\setminus\{\theta_*\}\] is an open connected set. It follows that $\{D_m\cap U\}_{m=0}^{n-1}$ is a cover of $U$ by pairwise disjoint closed sets. Therefore, by connectedness of $U$ for exactly one $0\leq m<n$ the set $D_m\cap U$ is non-empty. Evidently, $m=k$ or $m=j$ but then $D_m=\mathbb{S}^1$ which is a contradiction.

    Suppose now that the matrices $L_k$ are all equal to $L$, then the support function to each $F_k(A)$ becomes \[h_{F_k(A)}(\theta)=h_A(L^{\mathrm T}\theta)+\langle\theta,\vec{v}_k\rangle.\] Thus, the condition $h_{F_k(A)}(\theta)\geq h_{F_j(A)}(\theta)$ for all $\theta\in D_k$ simplifies to \[\langle\theta,\vec{v}_k\rangle\geq\langle\theta,\vec{v}_j\rangle\implies\langle\theta,\vec{v}_k-\vec{v}_j\rangle\geq0\] and the result follows. 
\end{proof}

For the remaining of this article we will restrict ourselves to IFS where part $(iii)$ of the above poposition is satisfied, so it is natural to have the following definitions.

\begin{definition}
    Let $\{\vec{v}_k\}_{k=0}^{n-1}$, for $n\geq2$, be the translation vectors of an homogeneous IFS generated by affine transformations. For any $0\leq k\neq j<n$ the set \[\ell_{k,j}:=\left\{\vec{v}_k+t(\vec{v}_k-\vec{v}_j)~|~t\in\mathbb{R}\right\}\] is called a \emph{core line} if a segment of it is an edge of the convex hull of the set of translation vectors. The corresponding bounding lines to the limit set $A$ are denoted by $\widehat{\ell}_{k,j}$.
\end{definition}

\begin{definition}
    Let $A$ be the attractor of an homogeneous IFS generated by $n$ affine transformations. Define the \emph{bounding set} of $A$, which is denoted by $\widehat{\conv}(A)$, to be the intersection of all the closed left half-planes of the bounding lines $\widehat{\ell}_{k,j}$.
\end{definition}
By definition it follows that $\conv(A)\subseteq \widehat{\conv}(A)$ and since a finite intersection of half-planes is convex, then so is $\widehat{\conv}(A)$.
\begin{figure}
    \centering
    \includegraphics[scale=0.4, angle=0]{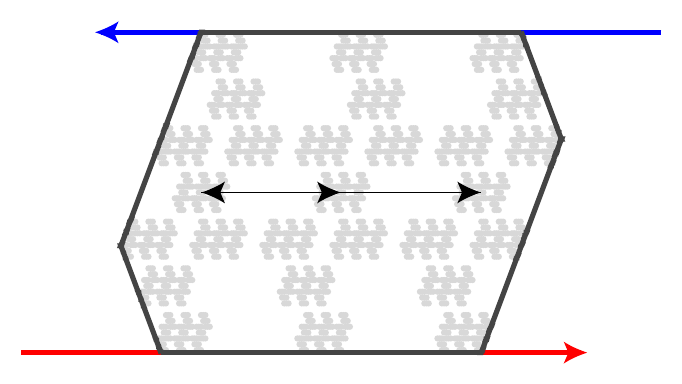}
    \includegraphics[scale=0.4,angle=90]{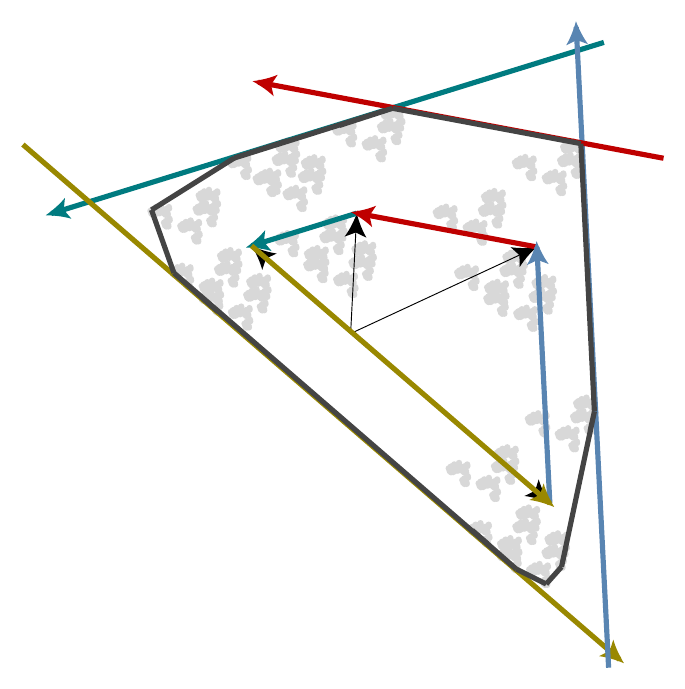}%
    \includegraphics[scale=0.4]{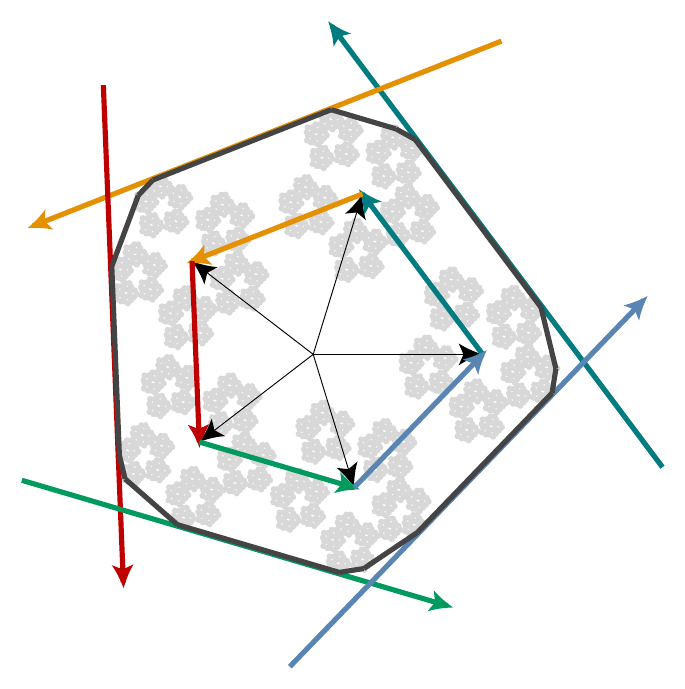}
    \caption{Examples of bounding set and convex hull of the attractor for IFS generated by three, four, or five affine transformations. The arrows that share a common origin are the translation vectors of the affine transformations. The arrows connecting pairs of translation vectors are segments of the core lines.}
    \label{fig:randomhull}
\end{figure}
\begin{proposition}
\label{prop:extremeconvexcore}
    Let $\{F_k\}_{k=0}^{n-1}$ be an homogeneous IFS of affine transformations and let $A$ be the associated limit set.
    
    Then for any edge $\gamma$ of $\conv(A)$ there exists at least one, possibly finite or empty, sequence $w_0w_1w_2\cdots$ with $w_k\in\{0,1,\cdots,n-1\}$ such that \[(F_{w_0}\circ F_{w_1}\circ F_{w_2}\circ\ldots)(\widehat{\gamma})\supseteq \gamma\]
    where $\widehat{\gamma}$ is an edge of $\widehat{\conv}(A)$.
\end{proposition}
\begin{proof}
    By definition, any edge $\widehat{\gamma}$ of $\widehat{\conv}(A)$ is a segment of some bounding line (or possibly the entire bounding line in the case when there are only two parallel ones). Furthermore, if $\widehat{\ell}_{k,j}$ is a bounding line, then by definition we can find a segment $[\vec{x},\vec{y}]$ of it which is contained or equal to a particular edge $\gamma$ of $\conv(A)$.
    
    From Theorem \ref{thm:extreme} we know that extreme points of $A$ are images of other extreme points. Since any edge $\gamma$ of $\conv(A)$ contains extreme points, then there must exists an integer $0\leq k< n$ for which $F_k(\gamma)\subseteq\gamma'$ for some other edge $\gamma'$ of $\conv(A)$. 
\end{proof}

As we shall see later, the above proposition is the core idea in our search of a description of the extreme points. We will see that the sequences $w_0w_1\cdots$ in Proposition \ref{prop:extremeconvexcore} are obtained by a sofic system on $\{0,1,\cdots,n-1\}$.

\section{The complex case}
\label{sect:complex}
The first family of  homogeneous IFS that we will consider has affine transformations with contracting matrix of the form $T=\begin{pmatrix}a&-b\\b&a\end{pmatrix}$ for real $a,b$. Applying $T$ to any vector in $\mathbb{R}^2$ is equivalent to multiply a point in $\mathbb{C}$ by $\lambda=a+b\I$, hence we will write $F_k(\vec{x})=T\vec{x}+\vec{v}_k$ as $f_k(z)=\lambda z+v_k$.

We will first prove that the $\conv(A)$ has finitely many sides if and only if $\arg(\lambda)$ is a rational multiple of $2\pi$ and in that case there is a DGIFS describing $\ext(A)$. Subsequently, we will focus on two families of IFS which have been studied before and differ by the type of translations: either $\{v_k\}_{k=0}^{N-1}$ are vertices of a regular $n$-gon or they lie along a line through the origin. We refer to the first family as a \emph{polygonal IFS} and the second family as \emph{collinear IFS}.

One particular consequence of what we will show in the next sections is that when $\arg(\lambda)$ is a rational multiple of $2\pi$, the itineraries of corner points, i.e. the vertices of the convex hull, are always periodic.

We remind the reader that this setting is the only overlap with the work of Vass \cite{V18} and \cite{DT05}.

\subsection{Arbitrary translations}
\label{subsect:complexarbitrary}
Fix an integer $N\geq2$, $\lambda \in \mathbb{D}\setminus\{0\}$, and define $\Gamma_{N}=\{f_k\}_{
kj=0}^{N-1}$ to be the IFS consisting of the contractive similarities \[f_k(x)=\lambda z+v_k,\] where $v_k\in\mathbb{C}$. Let $\alpha_k$ be such that $\arg(v_k)=2\pi\alpha_k$. Denote by $C$ the set of $v_k$ that are corners of $\conv(\{v_k\}_{k=0}^{N-1})$ and, after a reindexing if necessary, assume these are $v_0,v_1,\ldots,v_{N_c-1}$ for some $N_c\leq N$. We also assume that these vertices are ordered counterclockwise according to their angle, that is $0\leq\alpha_0<\alpha_1<\cdots<\alpha_{N_c-1}<1$. The outward normal direction of each core line $\ell_{k+1,k}$ then is given by $\arg(\pm\I(v_{k+1}-v_k))/(2\pi)=\arg(v_{k+1}-v_k)/(2\pi)\pm1/4$, where $k$ is considered modulo $N_c$. By definition, $\theta_k$ are also the outward normal directions of the bounding lines. Proposition \ref{prop:extremeconvexcore} tells us that there are iterates of the sides of the bounding set of $A$ that are mapped into the boundary of $\conv(A)$. Therefore, by setting $\phi$ to be such that $\arg(\lambda)=2\pi \phi$, define the set \[\Theta:=\bigcup_{j=0}^{N_c-1}\{k\phi+\theta_j\bmod 1~:~k\geq0\}.\] In other words, $\Theta$ is the set of all possible outward normal directions obtained by the repeatedly rotated bounding lines. For ease of exposition, we say that a line is at angle $\theta$ if the normal vector to it has argument $\theta$.
\begin{lemma}
\label{lem:Thetaset}
    If $\phi=p/q$ with coprime integers $p$ and $q$, then $\Theta$ has at most $qN_c$ elements. If $\phi\not\in\mathbb{Q}$, then $\Theta$ is a dense subset of $[0,1]$.
\end{lemma}
\begin{proof}
    By definition $\Theta$ is the union of the orbits of $\theta_j$ under rotation by $\phi$. If $\phi$ is irrational then a standard exercise shows that for each $j$ the set $\{k\phi+\theta_j\}_{k\geq0}$ is a dense subset of $[0,1]$. If $\phi=p/q$, for coprime integers $p$ and $q$, then for each $j$ the orbit of $\theta_j$ is periodic with period $q$. If the orbit of each $\theta_j$ do not intersect, then $\lvert\Theta\rvert=qN_c$. If there exists $k\geq1$ such that $\theta_j+k\phi=\theta_{j'}$ for two distinct $j$ and $j'$, then the orbit of $\theta_j$ coincides with that of $\theta_{j'}$ and, consequently, $\lvert\Theta\rvert\leq q(N_c-1)$.
\end{proof}
\begin{lemma}
\label{lem:ThetasetConvA}
    The angles in the set $\Theta$ are the outward normal directions of the sides of $\conv(A)$.
\end{lemma}
\begin{proof}
    The claim follows from Proposition \ref{prop:extremeconvexcore} and that $f_{s}(\widehat{\ell}_{j+1,j})$ is at angle $k\phi+\theta_j$ for some sequence $s=s_1s_2\cdots s_k$ with $s_j\in\{0,1,\cdots, N-1\}$.
\end{proof}

\begin{theorem}
\label{thm:maincplxarbitrary}
    Let $\phi$ be such that $\arg(\lambda)=2\pi\phi$ and consider $\Theta$ as defined above.
    \begin{itemize}
        \item If $\phi=p/q$ with coprime integers $p$ and $q$, then $\Theta$ has at most $qN_c$ elements. Moreover, 
        \begin{enumerate}
            \item the convex hull of $A$ is a finite sided polygon whose sides are at the angles of $\Theta$;
            \item the set of extreme points of $A$ is described by a DGIFS consisting of at most $qN_c$ equations;
            \item if $A$ is convex, then $\lvert\lambda\rvert\geq2^{-1/q}$.
        \end{enumerate}
        \item If $\phi\not\in\mathbb{Q}$ then 
        \begin{enumerate}
            \item $\Theta$ is a dense subset of $[0,1]$.
            \item The attractor $A$ is not convex.
        \end{enumerate}
    \end{itemize}
\end{theorem}
\begin{remark}
    If $\phi\not\in\mathbb{Q}$, then $\ext(A)$ can be approximated by applying the above theorem while considering a $\psi\in\mathbb{Q}$ such that $\lvert\psi-\phi\rvert<\varepsilon$, for a given $\varepsilon>0$. This is due to the continuity of the function mapping any $\lambda$ to its corresponding attractor $A$.
\end{remark}
\begin{proof}
    The previous two lemmas prove the claims in (1) for each case of $\phi$. 
    
    Let $\phi=p/q$ and consider the partition of the unit circle obtained with the $\theta_j$: denote by $I_j$ the smallest arc between two consecutive $\theta_j$ so that $I_j:=[\theta_{j-1},\theta_{j}]$. Recall from Proposition \ref{prop:boundinglines} that each $\theta_j$ belongs to the dominance region $D_j$ of the map $f_j$ and to $D_k$ for at least one more $k$. It follows that the interval $I_j$ corresponds to the dominance region $D_{j}$. For each $0\leq j<N_c$ let $B_{j,0}$ be the set of extreme points of $A$ that lie on the bounding line $\widehat{\ell}_{j+1,j}$. By Theorem \ref{thm:extreme} and since $\theta_j$ is periodic with period $q$ under rotation by $\phi$, 
    \[
    B_{j,0}=\bigcup_kf_{k}(B_{j,q-1})\text{ and }B_{j,m}=f_{n}(B_{j,m-1})
    \]
    where $k$ runs over the indices of those maps $f_k$ whose dominance region contains $\theta_j$, and $n$ and $m$ satisfy $m\phi+\theta_j\in I_{n}^\circ=(\theta_{n-1},\theta_{n})$.

    Suppose now that $A$ is convex and $\phi=p/q$. By convexity, the bounding lines $\widehat{\ell}_{k+1,k}$ intersect $\partial\conv(A)=\partial A$ on segments of positive length: the sets $B_{j,i}\subset A$ for each $j$ and $i$ are segments. Using the equations we have just found, for each index $j$ we have
    \[\lvert B_{j,0} \rvert=\sum_{k}\lvert f_k(B_{j,q-1})\rvert\geq2\lvert\lambda\rvert\lvert B_{j,q-1}\rvert\text{ and }\lvert B_{j,q-1}\rvert=\lvert\lambda\rvert^{q-1}\lvert B_{j,0}\rvert\]
    which implies that we must have $2\lvert\lambda\rvert^{q}\geq1$ for the assumptions to hold. Note the first inequality, which gives a factor of $2$, is necessary because we do not know a priori the distribution of the $v_k$ along the bounding line $\widehat{\ell}_{j+1,j}$, so our best bet is to find the condition for which the two outermost copy of $B_{j,q-1}$ on the bounding line intersect.
    
    Finally, let us assume $\phi$ is irrational and that $A$ is convex. As mentioned before, the sets $B_{j,0}$ must be segments of positive length belonging to the bounding lines $\widehat\ell_{j+1,j}$. Moreover, we know that $B_{j,0}$ is the union of copies $f_k(B_{x})$ where $k$ runs over the indices of those maps $f_k$ whose dominance region contains $\theta_j$, and $B_x$ are those edges of $A$ with normal direction $t_x$ such that $\phi+t_x=\theta_j$. From Lemma \ref{lem:Thetaset} and \ref{lem:ThetasetConvA} we know that such $t_x$ do not exist.
\end{proof}

\begin{remark}
\label{rem:sofic}
In the above proof, we defined each $B_{j,i}$ to be the set of points on an edge of $\conv(A)$. The equations relating those sets are equivalent to a set of rules on $\Sigma_N:=\{0,1,\cdots,N-1\}^\mathbb{N}$ that identify which itineraries belong to a particular $B_{j,i}$. In other words, we have found a sofic system on $\Sigma_N$ that determines the extreme points of $A$.     
\end{remark}

After the following example, we will consider cases in which the translation vectors $\{v_k\}_{k=0}^N$ have a particular structure. We will give the exact system of equations that define the DGIFS and improve the condition on $\lambda$ for which the attractor is convex.

\subsection{Examples}
\label{subsect:complexgeneralexamples}
Here we will illustrate the proof of Theorem \ref{thm:maincplxarbitrary} by considering the IFS $\Gamma_4=\{f_j\}_{j=0}^3$ given by the following similarities:
\begin{align*}
    f_0(z)=\lambda z+2&\quad&f_1(z)=\lambda z+\frac{\sqrt{3}}{2}\I\\
    f_2(z)=\lambda z-\frac{3}{8}(1-\sqrt{3}\I)&\quad&f_3(z)=\lambda z+\frac{2\sqrt{3}}{4+\sqrt{3}}(1+\I)
\end{align*}
where $\lambda=0.5\exp(2\pi \I/4)=0.5\I$. The translations $v_k$ were chosen so that the outward normal direction of the sides of their convex hull were not all rational: the convex hull of $\{v_k\}_{k=0}^3$ is the region in $\mathbb{R}^2$ bounded by the lines 
\[
    y=-\frac{\sqrt{3}}{4}x+\frac{\sqrt{3}}{2},\quad y=\frac{1}{\sqrt{3}}x+\frac{\sqrt{3}}{2},\quad y=-\frac{3\sqrt{3}}{19}x+\frac{6\sqrt{3}}{19}
\] from which we calculate the outward normal direction of the sides to be 
\begin{align*}
\theta_0=\frac{1}{2\pi}\arctan\left(\frac{4}{\sqrt{3}}\right)\approx0.1849632654&\quad&\theta_1=\frac{1}{3}\\\theta_2=\frac{1}{2\pi}\arctan\left(\frac{19}{3\sqrt{3}}\right)+\frac{1}{2}\approx0.707512932496&&
\end{align*}
Let us now look at the orbit of each $\theta_j$ under rotation by $\phi=1/4$ and record in which interval $I_j=[\theta_{j-1},\theta_j]$ the iterate of $\theta_j$ lands:

\begin{align*}    
    1\phi+\theta_0\in I_2 && 1\phi+\theta_1\in I_2 && 1\phi+\theta_2\in I_0 \\
    2\phi+\theta_0\in I_2 && 2\phi+\theta_1\in I_0 && 2\phi+\theta_2\in I_1 \\
    3\phi+\theta_0\in I_0 && 3\phi+\theta_1\in I_0 && 3\phi+\theta_2\in I_2 \\
    4\phi+\theta_0\in I_0\cap I_1 && 4\phi+\theta_1\in I_1\cap I_2 && 4\phi+\theta_2\in I_2\cap I_0 
\end{align*}
Using these information we construct the directed graph in Figure \ref{fig:Random}. Note that $B_{0,0}$ has an outgoing arrow with label $3$ because $v_3$ is collinear to $v_0$ and $v_1$.
\begin{figure}
    \centering
    \includegraphics[width=0.55\linewidth]{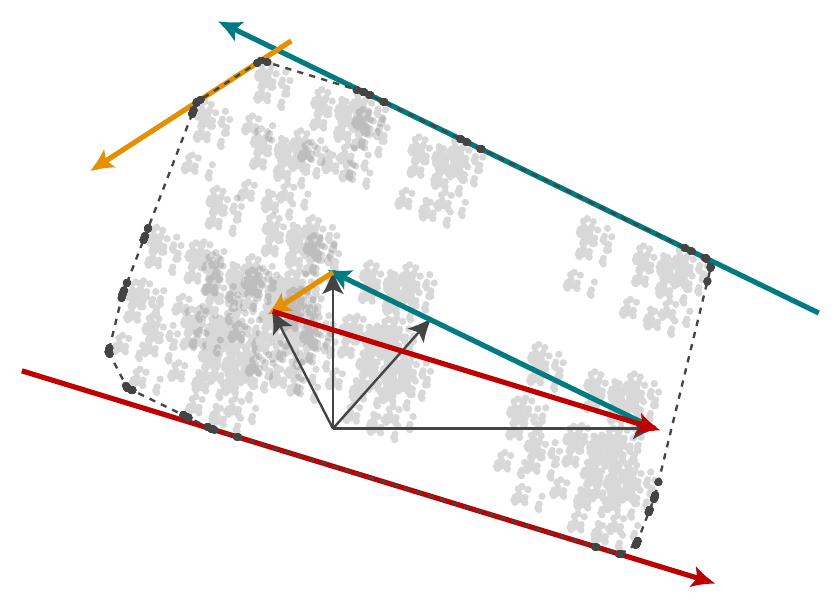}
    % \hspace{0.7cm}
    \begin{tikzpicture}
    
    \node [draw, circle] (0+60) at (0+60:4) (s0) {\small $B_{0,0}$};
    \node [draw, circle] (30+60) at (30+60:4) (s1) {\small $B_{0,1}$};
    \node [draw, circle] (60+60) at (60+60:4) (s2) {\small $B_{0,2}$};
    \node [draw, circle] (90+60) at (90+60:4) (s3) {\small $B_{0,3}$};
    \node [draw, circle] (120+60) at (120+60:4) (s4) {\small $B_{1,0}$};
    \node [draw, circle] (150+60) at (150+60:4) (s5) {\small $B_{1,1}$};
    \node [draw, circle] (180+60) at (180+60:4) (s6) {\small $B_{1,2}$};
    \node [draw, circle] (210+60) at (210+60:4) (s7) {\small $B_{1,3}$};
    \node [draw, circle] (240+60) at (240+60:4) (s8) {\small $B_{2,0}$};
    \node [draw, circle] (270+60) at (270+60:4) (s9) {\small $B_{2,1}$};
    \node [draw, circle] (300+60) at (300+60:4) (s10) {\small $B_{2,2}$};
    \node [draw, circle] (330+60) at (330+60:4) (s11) {\small $B_{2,3}$};
    
    \path[->] (s0) edge [bend left=20,below] node {$3$} (s3);
    \path[->] (s0) edge [bend left=40,below] node {$1$} (s3);
    \path[->] (s0) edge [bend left=60,below] node {$0$} (s3);
    \draw[->,thick] (s1)--(s0) node[midway, above] {\small $2$};
    \draw[->,thick] (s2)--(s1) node[midway, above] {\small $2$};
    \draw[->,thick] (s3)--(s2) node[midway, above] {\small $0$};
    
    \path[->] (s4) edge [bend left=30,right] node {$2$} (s7);
    \path[->] (s4) edge [bend left=60,right] node {$1$} (s7);
    \draw[->,thick] (s5)--(s4) node[midway, left] {\small $2$};
    \draw[->,thick] (s6)--(s5) node[midway, left] {\small $0$};
    \draw[->,thick] (s7)--(s6) node[midway, below left] {\small $0$};
    
    \path[->] (s8) edge [bend left=30,left] node {$0$} (s11);
    \path[->] (s8) edge [bend left=60,left] node {$2$} (s11);
    \draw[->,thick] (s9)--(s8) node[midway, right] {\small $0$};
    \draw[->,thick] (s10)--(s9) node[midway, right] {\small $1$};
    \draw[->,thick] (s11)--(s10) node[midway, right] {\small $2$};
    
\end{tikzpicture}
    \caption{The extreme points, the convex hull, and the bounding lines of the attractor of $\Gamma_4$. The arrows with a common origin indicate the translation vectors. On the right the sofic system for $\ext(A)$.}
    \label{fig:Random}
\end{figure}
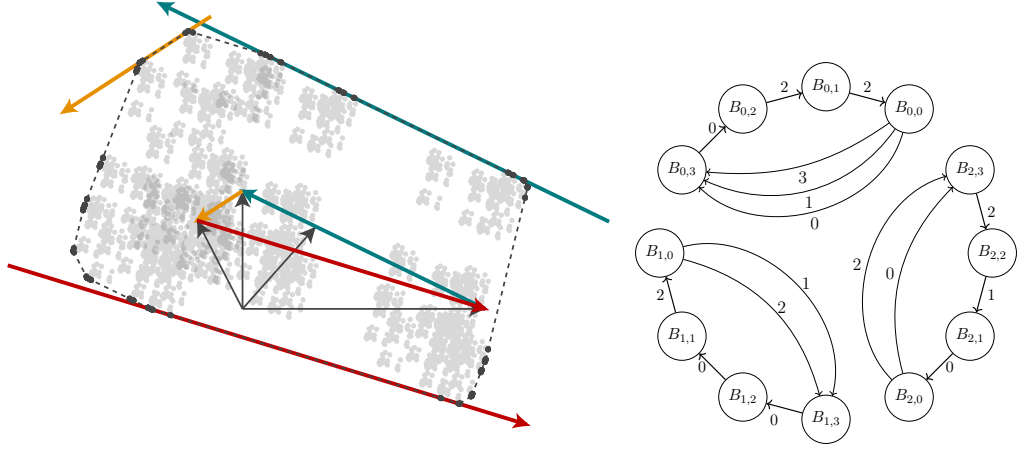

Given that $v_0,v_3$ and $v_1$ are not equally spaced along the core line $\ell_{1,0}$, to guarantee connectedness of $B_{0,0}$ we can only ask that $f_0(B_{0,3})$ to intersect $f_1(B_{0,3})$. That is, we want \[\lvert f_0(B_{0,3})\rvert+\lvert f_1(B_{0,3})\rvert=2\lvert\lambda\rvert\lvert B_{0,3}\rvert=2\lvert\lambda\rvert^{4}\lvert B_{0,0}\rvert\geq\lvert B_{0,0}\rvert\implies 2\lvert\lambda\rvert^4\geq1.\] 

\subsection{Polygonal IFS}
\label{subsect:complexpolygonal}
Fix $n\geq2$, a non-zero $w\in\mathbb{C}$, set $\lambda\in\mathbb{D}\setminus\{0\}$, and define $\Phi_{n}=\{f_j\}_{j=0}^{n-1}$ to be the IFS consisting of the contractive similarities \[f_j(z)=\lambda z+\xi_n^jw,\]
with $\xi_n=\exp(2\pi \I/n)$. Denote by $A$ the attractor of $\Phi_{n}$. 

Calegari and Walker characterized the extreme points of $A$ when $w=1$. By defining a specific (real) linear function on the unit circle, they are able to find for a given direction the extreme points. Their underline idea in the proof of their theorem is what inspired us to generalize it. We will apply Theorem \ref{thm:maincplxarbitrary} to prove an equivalent theorem to that in \cite{CW19} in terms of directed graph iterated function system. We later found out that the real linear function used by Calegari and Walker was already introduced by Hare and Sidorov \cite{HS16} to obtain equivalent results.

\begin{theorem}
\label{thm:maincplxpolygon}
    Let $\phi$ and $\alpha$ be such that $\arg(\lambda)=2\pi\phi$ and $\arg(\xi_n)=2\pi\alpha$.
    \begin{itemize}
        \item If $p/q=\phi\in\mathbb{Q}$ with coprime integers $p$ and $q$, let $b$ be such that $bn={\textrm lcm}(n,q)$, and set $d=\gcd(n,q)$. Define \[\Theta=\left\{k\phi+\left(m+\frac{1}{2}\right)\alpha\right\}_{k=0,m=0}^{k=b-1,m=n-1}\]
        Then:
        \begin{enumerate}
            \item the convex hull of $A$ is a polygon with $nb$ sides at the angles of $\Theta$.
            \item The set of extreme points of $A$ is the unique vector of non-empty compact sets $\{B_j\}_{j=0}^{nb-1}$ which solve the following system of equations: for each $0\leq m\leq n-1$ 
            \[\begin{cases}
                B_{mb}=f_{m}(B_{mb-np/d})\cup f_{m+1}(B_{mb-np/d}) &\\
                B_k=f_{m+1}(B_{k-np/d}) &
            \end{cases}\]
            where $mb< k< (m+1)b$ (the indices of the sets are to be considered modulo $nb$, while those of the maps modulo $n$).
            \item If the attractor $A$ is convex, then $\lvert\lambda\rvert\geq2^{-1/b}$ and the system of equations defined above describes the points on the boundary of $A$.
        \end{enumerate}
        \item If $\phi\not\in\mathbb{Q}$ then 
        \begin{enumerate}
            \item the convex hull does not have finitely many sides.
            \item The attractor $A$ is not convex.
        \end{enumerate}
    \end{itemize}
\end{theorem}
In contrast to the results of \cite{CW19}, for a rational $\phi$ the description of the extreme points is given explicitly without the need of many computations. In their theorem, when $\phi\in\mathbb{Q}$, they mention that any single edge of $\conv(A)$ is itself the attractor of an IFS of affine maps. The same result can be recovered from the DGIFS in our theorem by simple substitution. 

The proof of the theorem follows the same argument as that of Theorem \ref{thm:maincplxarbitrary}, we only need to show the details. We begin with a couple of lemmas. Some of the claims in them can be found in \cite{CW19} with possibly different terminology.
\begin{lemma}
\label{lem:rationalphi}
     Assume $\phi=p/q$ with coprime integers $p$ and $q$. Let $b$ such that $bn={\textrm lcm}(n,q)$ and set $d=\gcd(n,q)$. Then
     \begin{enumerate}
         \item for all integers $k$ the integer $j=np/d$ is such that \[(k+b)\phi+\theta_m=k\phi+\theta_{m+j}\] and $b$ is the smallest integer for which the equality holds.
         \item $d$ is the smallest positive integer such that $db\phi\equiv0\bmod1$.
    \end{enumerate}
\end{lemma}
\begin{proof}
    A common formula that relates the least common multiple of two positive integers with their greatest common divisor is \[\gcd(n,q){\textrm lcm}(n,q)=nq.\] It follows that $b=q/d$. Now if we let $a=n/d$, then $b$ is the least such integer such that $b/q=a/n$.  Consequently, \[b\phi+\theta_m=\frac{bp}{q}+\frac{m}{n}+\frac{1}{2n}=\frac{ap}{n}+\frac{m}{n}+\frac{1}{2n}=\theta_{ap+m}.\] Therefore, $j=ap=np/d$.
    
    Now assume that for some integer $i<b$ there exists some integer $t$ such that $i\phi+\theta_m=\theta_{m+t}$, then it would imply that $ip/q=t/n$ but this contradicts the minimality of $b$. This proves the first claim.

    For the second claim we firstly show that the equivalence is satisfied. Notice that $n$ and $q$ are such that $nb\phi\equiv 0\bmod1$ and $qb\phi\equiv0\bmod1$: for $q$ is trivial since $\phi=p/q$; for $n$ the claim is true by definition of $b$\[nb\phi=nb\frac{p}{q}=n\frac{q}{d}\frac{p}{q}=\frac{n}{d}p\equiv0\bmod1.\] The claimed equivalence is then satisfied since, by definition, $d$ is the smallest integer such that $d=nx+qy$ for integers $x,y$: 
    \[
        db\phi=(nx+qy)b\phi=(nb\phi)x+(qb\phi)y\equiv 0 \bmod 1.
    \]
    Secondly, suppose by contradiction that there is an integer $d'<d$ for which $d'b\phi\equiv0\bmod1$, then it would imply that $d'$ be a multiple of both $n$ and $q$ which is a contradicts the minimality of $d$. 
\end{proof}
\begin{lemma}
\label{lem:irrationalphi}
    Let $N$ be any integer. Then if $\phi\not\in\mathbb{Q}$
    \begin{enumerate}
        \item there is no integer $k$ for which $k\phi+\theta_m=N\phi$;
        \item there exists exactly one pair of integer $k$ and $m$ for which $k\phi+\theta_m=N\phi+\alpha/2$.
        \item there exists exactly one pair of integers $k$ and $j$ such that $k\phi+\theta_m=N\phi+\theta_{m+j}$.
    \end{enumerate}
\end{lemma}
\begin{proof}
    For part (1) we will prove the contrapositive: assume that there exists positive $k<N$ such that $k\phi+\theta_m=N\phi$, then rearranging the sides of the equation we have \[(N-k)\phi=\theta_m.\] The expression on the right is rational and never $0$ and $(N-k)$ is an integer, thus $\phi$ must be rational. 

    We will prove part (2) directly:
    if $N$ is any integer then 
    \[k\phi+m\alpha+\frac{\alpha}{2}=N\phi +\frac{\alpha}{2}\implies (N-k)\phi=m\alpha\]
    which is true only when both sides are $0$, since $\phi$ is irrational and $\alpha$ is rational. Therefore, $k=N$ and $m=0$.

    The argument is similar for part (3):
    \[ k\phi+m\alpha+\frac{\alpha}{2}=N\phi+(m+j)\alpha+\frac{\alpha}{2}\implies(k-N)\phi=j\alpha\]
    which is true only when $k=N$ and $j=0$, since $\phi$ is irrational and $\alpha$ is rational.

\end{proof}

These lemmas are essentially describing the behaviour of any point $z$ on some bounding line when multiplied by $\lambda$, that is they show that $\Theta$ is indeed the one stated in Theorem \ref{thm:maincplxpolygon}.

\begin{proof}[Proof of Theorem \ref{thm:maincplxpolygon}]
    The following two facts are properties of the attractor $A$ relatively simple to prove:
    \begin{itemize}
        \item the attractor $A$ and, consequently, $\conv(A)$ have rotational symmetry of order $n$ about the origin; 
        \item the attractor $A$ is the union of $f_j(A)$ for $0\leq j\leq n$ which are copies of $A$ scaled by $\lvert\lambda\rvert$, rotated by $\arg(\lambda)$, and translated relative to each other by $\xi_n^j-\xi_n^k$.
    \end{itemize}
    Because of the rotational symmetry there are necessarily extreme points on the bounding line at angle $\theta_m=(m+1/2)\alpha$ for each $0\leq m<n$. The second fact implies that, for fixed $m$, the extreme points on the bounding line at angle $\theta_m$ must be the image of other extreme points under the maps $f_m$ and $f_{m+1}$.

    Let $\phi=p/q$ for coprime integers $p$ and $q$. Following the argument in Theorem \ref{thm:maincplxarbitrary} denote by $B_j$ the set of extreme points on each line at angles from $\Theta$: $B_0$ corresponds to set of extreme points on $\widehat{\ell}_{1,0}$ which is at angle $0\phi+\theta_0$; $B_1$ is a subset of $f_k(\widehat{\ell}_{1,0})$ which is at angle $1\phi+\theta_0\in I_k$ for some $k$; $B_b$ corresponds to set of extreme points on $\widehat{\ell}_{2,1}$ which is at angle $0\phi+\theta_1$; and so on up until $B_{nb}$ which is at angle $(b-1)\phi+\theta_{m-1}$. From Lemma \ref{lem:rationalphi} we know that $B_{mb-np/d}$ is the only set of extreme points that maps to $B_{mb}$ with the maps $f_m$ and $f_{m+1}$, that is \[B_{mb}=f_m(B_{mb-np/d})\cup f_{m+1}(B_{mb-np/d}).\]
    Now, if there were an integer $mb<k<(m+1)b$ then we claim that $B_k$ is necessarily the image of $B_{k-np/d}$ under $f_{m+1}$: since the interval $I_{m+1}=[\theta_m,\theta_{m+1}]$ is centered around $(m+1)\alpha$, then $B_k$ must be a subset of $f_{m+1}(A)$, i.e. it must be the image under $f_{m+1}$ of some subset of $A$; by Lemma \ref{lem:rationalphi} such set is $B_{k-pn/d}$, so \[B_k=f_{m+1}(B_{k-np/d}).\]

    If the attractor $A$ is convex, then each $B_k$ must be a segment of positive length entirely contained in $A$. We have just shown that the extreme points in $B_{mb}$ are the union of images of extreme points under $f_m$ and $f_{m+1}$. Moreover, each $B_{k}$ for $mb<k<(m+1)b$ is the image under $f_{m+1}$ of $B_{k-np/d}$. It follows that \[B_{mb-\frac{np}{d}}=f_{m+1}(B_{mb-2\frac{np}{d}})=f_{m+1}^2(B_{mb-3\frac{np}{d}})=\cdots=f_{m+1}^{b-1}(B_{mb-b\frac{np}{d}})=f_{m+1}^{b-1}(B_{mb})\] where the last equality holds because $nbp/d=nb^2\phi\equiv0\bmod1$. Therefore, we have \[\lvert f_{m}(B_{mb-\frac{np}{d}})\rvert=\lvert f_{m+1}(B_{mb-\frac{np}{d}})\rvert=\lvert \lambda\rvert\lvert B_{mb-\frac{np}{d}}\rvert\text{ and }\lvert B_{mb-\frac{np}{d}}\rvert=\lvert\lambda\rvert^{b-1}\lvert B_{mb}\rvert\] which implies that the segment $B_{mb}$ is contained in $A$ if $2\lvert \lambda\rvert^b\geq1$.
\end{proof}

\begin{corollary}
    With the assumptions of Lemma \ref{lem:rationalphi} the directed graph associated to the DGIFS in Theorem \ref{thm:maincplxpolygon} has $n/gcd(n,q)$ connected components each of length $b\gcd(n,q)$. Moreover, each connected components is strongly connected.
\end{corollary}
\begin{proof}
    The second claim of Lemma \ref{lem:rationalphi} shows that for a fixed $m$ there is a cyclic path of length $db$, with $d=\gcd(n,q)$, passing through distinct vertices of the graph. Considering that there are $nb$ total vertices, then the number of cycles is $nb/db=n/d$. 

    Each connected component is strongly connected because it is cyclic.
\end{proof}

\begin{remark}
    We included the corollary since if the graph of the DGIFS is composed by strongly connected components and the DGIFS satisfies the open set condition,  then it possible to calculate \cite{MW88,B89b} the similarity, and consequently, the Hausdorff dimension of the vector of attractors. However, this is beyond the scope of this article. 
\end{remark}

\subsection{Examples}
\label{subsect:complexpolygonalexamples}
For each of the following examples we fix $n=3$ and $w=1$.
\subsubsection{Sierpi\'nski Triangle}
As a first example we consider the famous Sierpi\'nski Triangle. The fractal can be obtained by setting, for instance, $\lambda=1/2$ or $\lambda=\frac{1}{2}\exp(2\pi\I/3)$. In either case, $\ext(A)$ is evidently the union of three segments which are the sides of an equilateral triangle. However, the itinerary of $x\in\ext(A)$ is different accroding to which $\lambda$ we use: when $\lambda=1/2$ its argument is $\phi=0$ so $p=0$ and $q=1$, then by Theorem \ref{thm:maincplxpolygon} $\ext(A)$ is described by the DGIFS on the left in Figure \ref{fig:SierpisnkiDigraph};
%\[\begin{cases}
%     B_0=f_0(B_0)\cup f_1(B_0)&\\
%     B_1=f_1(B_1)\cup f_2(B_1)&\\
%     B_2=f_2(B_2)\cup f_0(B_2)&
% \end{cases};\]
the other case $\phi=p/q=1/3$ is shown on the right of the same figure.
% \[\begin{cases}
%     B_0=f_0(B_2)\cup f_1(B_2)&\\
%     B_1=f_1(B_0)\cup f_2(B_0)&\\
%     B_2=f_2(B_1)\cup f_0(B_1)&
% \end{cases}.\]
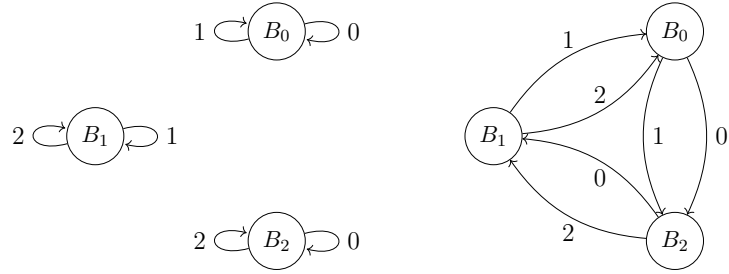
\begin{figure}
    \centering
    \begin{tikzpicture}
    
    \node [draw, circle] (0+60) at (0+60:2) (s0) {\small $B_0$};
    \node [draw, circle] (120+60) at (120+60:2) (s1) {\small $B_1$};
    \node [draw, circle] (240+60) at (240+60:2) (s2) {\small $B_2$};
        
    \path (s0) edge [loop right] node {$0$} (s0);
    \path (s0) edge [loop left] node {$1$} (s0);
    \path (s1) edge [loop right] node {$1$} (s1);
    \path (s1) edge [loop left] node {$2$} (s1);
    \path (s2) edge [loop right] node {$0$} (s2);
    \path (s2) edge [loop left] node {$2$} (s2);
    
\end{tikzpicture}
    \hspace{1cm}
    \begin{tikzpicture}
    
    \node [draw, circle] (0+60) at (0+60:2) (s0) {\small $B_0$};
    \node [draw, circle] (120+60) at (120+60:2) (s1) {\small $B_1$};
    \node [draw, circle] (240+60) at (240+60:2) (s2) {\small $B_2$};
        
    \path[->] (s0) edge [bend right=25, midway, right] node {$1$} (s2);
    \path[->] (s0) edge [bend left=25, midway, right] node {$0$} (s2);
    \path[->] (s1) edge [bend right=25, midway, above] node {$2$} (s0);
    \path[->] (s1) edge [bend left=25, midway, above] node {$1$} (s0);
    \path[->] (s2) edge [bend right=25, midway, below] node {$0$} (s1);
    \path[->] (s2) edge [bend left=25, midway, below] node {$2$} (s1);
    
\end{tikzpicture} 
    \caption{The two DGIFS describing the boundary of the convex hull of the Sierpinski Triangle when the maps have either no rotation (left) or rotation by a third of a turn (right).}
    \label{fig:SierpisnkiDigraph}
\end{figure}
\begin{figure}
    \centering
    \includegraphics[width=0.3\linewidth]{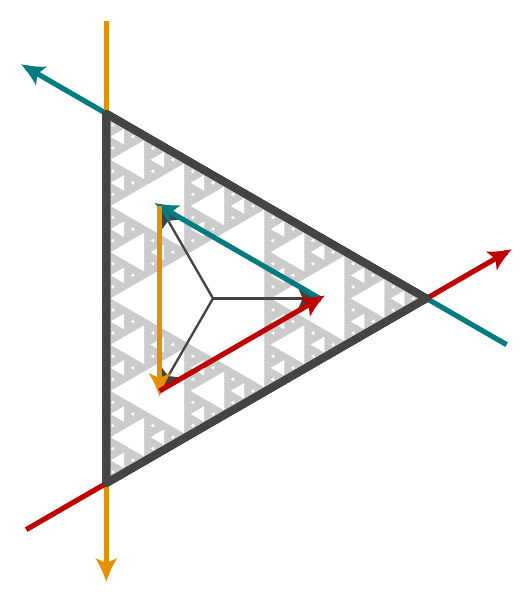}
    \caption{The Sierpi\'nski Triangle with its extreme points (the thick black segments) and bounding set. Also indicated are the segments of the core lines that form the convex hull of the translation vectors. }
    \label{fig:Sierpinski}
\end{figure}
We remark that in both cases $b=1$ and $\lvert\lambda\rvert=2^{-1/b}$ but the attractor is not convex, however the sets $B_j$ are intervals. The attractor is convex when $|\lambda|=2^{-1/2}>0.5$.

\subsubsection{Sierpi\'nski relatives}
In the second example we consider a sequence of parameters $\lambda=r\exp(2\pi\I/6)$, where $0<r<1$. First we apply the Theorem \ref{thm:maincplxpolygon} to find the DGIFS which is valid for all values of $r$: $\phi=p/q=1/6$ so there are $nb={\textrm lcm}(n,q)=6$ sides to the convex hull which are described by the DGIFS in Figure \ref{fig:SierpisnkiRelDigraph}.
% \[\begin{cases}
%     B_0=f_0(B_5)\cup f_1(B_5)&\\
%     B_1=f_1(B_0)&\\
%     B_2=f_1(B_1)\cup f_2(B_1)&\\
%     B_3=f_2(B_2)&\\
%     B_4=f_2(B_3)\cup f_0(B_3)&\\
%     B_5=f_0(B_4)&
% \end{cases}.\]
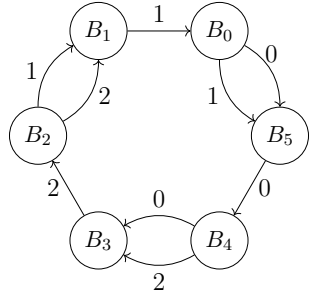
\begin{figure}[!htb]
    \centering
    \begin{tikzpicture}
    
    \node [draw, circle] (60) at (60:2) (s0) {\small $B_0$};
    \node [draw, circle] (120) at (120:2) (s1) {\small $B_1$};
    \node [draw, circle] (180) at (180:2) (s2) {\small $B_2$};
    \node [draw, circle] (240) at (240:2) (s3) {\small $B_3$};
    \node [draw, circle] (300) at (300:2) (s4) {\small $B_4$};
    \node [draw, circle] (0) at (0:2) (s5) {\small $B_5$};

    \path[->] (s0) edge [bend left,above] node {$0$} (s5);
    \path[->] (s0) edge [bend right,left] node {$1$} (s5);
    %\path[->] (s1) edge [double, midway, left] node {$0,1$} (s6);
    \path[->] (s2) edge [bend left, left] node {$1$} (s1);
    \path[->] (s2) edge [bend right, right] node {$2$} (s1);
    \path[->] (s3) edge [midway,left] node {$2$} (s2);
    %\path[->] (s3) edge [double, midway, right] node {$1,2$} (s2);
    \path[->] (s4) edge [bend left, below] node {$2$} (s3);
    \path[->] (s4) edge [bend right, above] node {$0$} (s3);
    \path[->] (s5) edge [midway, right] node {$0$} (s4);
    %\path[->] (s5) edge [double, midway, above] node {$0,2$} (s4);
    \path[->] (s1) edge [midway, above] node {$1$} (s0);
\end{tikzpicture}
    \caption{The DGIFS describing the set of external points of the Sierpinski relative.}
    \label{fig:SierpisnkiRelDigraph}
\end{figure}

When $r=0.5$ the attractor is a relative of the Sierpi\'nski Triangle \cite{TR18}, the sides of the convex hull are therefore Cantor sets. We can obtain the IFS generating each side by decoupling the system of equations for $B_0$, $B_2$, and $B_3$: for example,
\[ B_0=f_0(B_5)\cup f_1(B_5)=f_{00}(B_4)\cup f_{10}(B_4)=\cdots=\bigcup_{\omega\in S} f_\omega(B_0)\]
where \[S=\{000211,000221,002211,002221,100211,100221,102211,102221\}.\]
In other words, the extreme points in $B_0$ are the points $z\in A$ whose itinerary is $\omega=w_0w_1\cdots$ where each $w_k\in S$.

By increasing $r$ eventually we will reach a value for which the Cantor set $B_0$ becomes an interval: such value is $r=2^{-1/b}=2^{-1/2}$. The attractor, in this case, is convex.
\begin{figure}[!htb]
    \centering
    \includegraphics[width=0.3\linewidth]{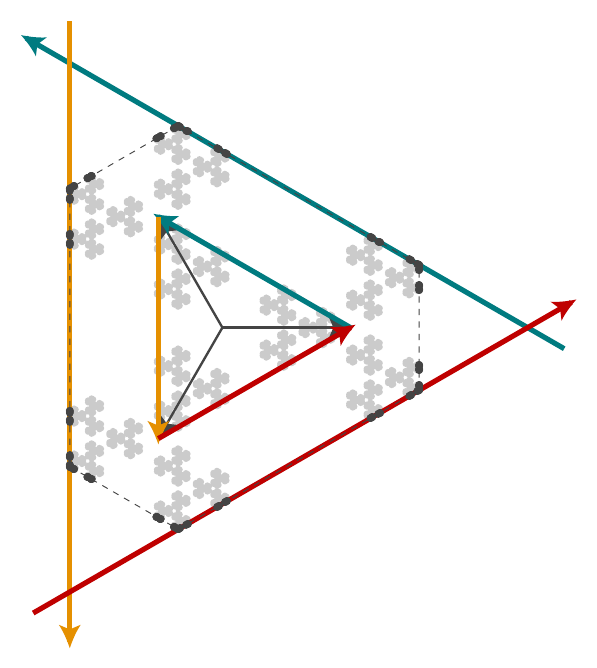}%
    \includegraphics[width=0.3\linewidth]{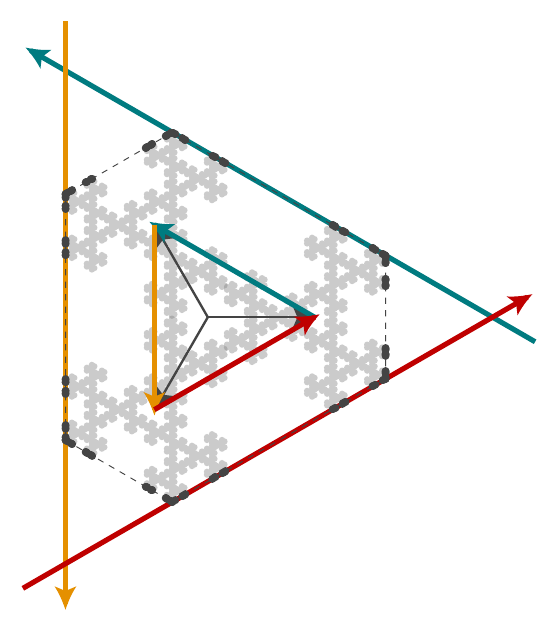}
    \includegraphics[width=0.33\linewidth]{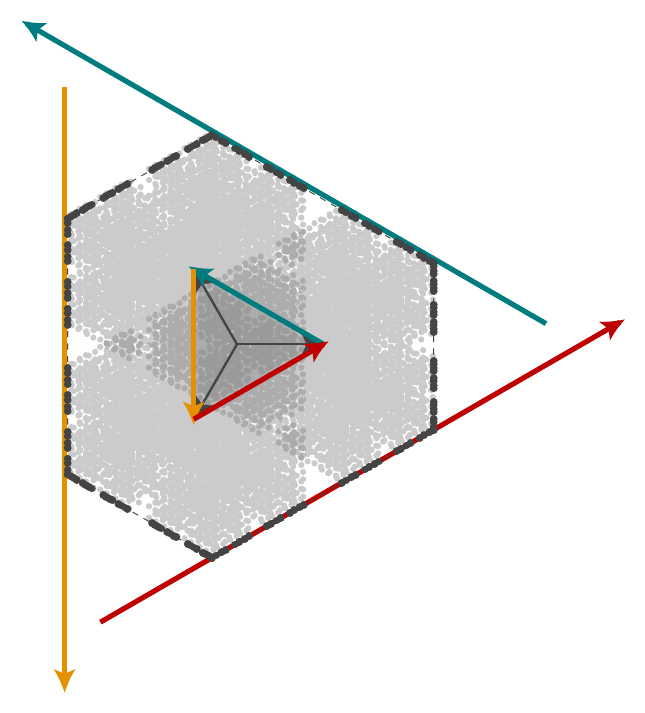}%
    \includegraphics[width=0.3\linewidth]{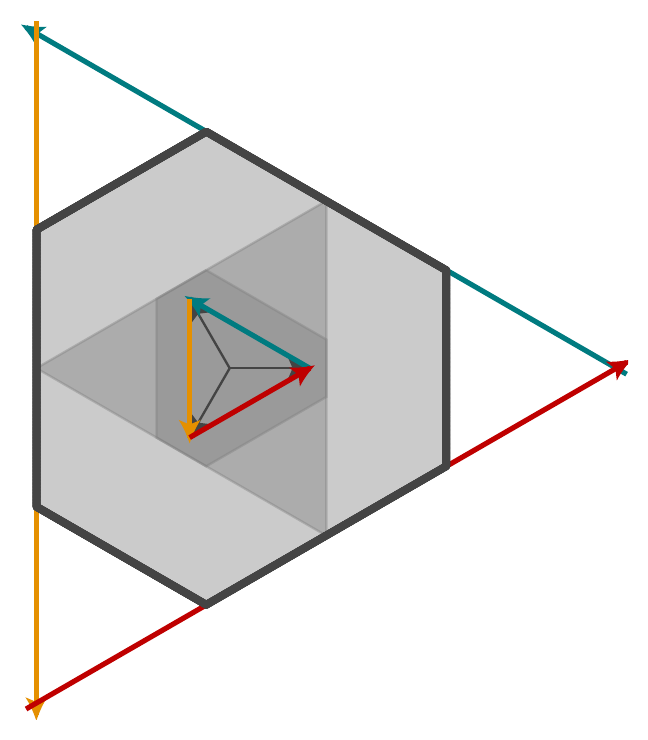}
    \caption{The extreme points, the convex hull, and bounding set of the Sierpi\'nksi relatives. From top left to bottom right, in order are the attractors when $\lvert\lambda\rvert=0.45$, $0.5$, $0.65$, and $2^{-1/2}$.}
    \label{fig:SierpisnkiRelatives}
\end{figure}

\subsubsection{Terdragon}
In the last example we consider a fractal which is known to tile the plane \cite{BG94}, called \emph{Terdragon} or, sometimes, \emph{Fudgeflake} \cite{Mandelbrot83}. Let $\lambda=2^{-1}\left(1+3^{-1/2}\I\right)$, one of the roots of $1-3x+3x^2$, then $\lvert\lambda\rvert=3^{-1/2}$ and $\phi=p/q=1/12$. By Theorem \ref{thm:maincplxpolygon} the convex hull is a polygon with $nb={\textrm lcm}(n,q)=12$ sides and $\ext(A)$ is described by the DGIFS illustrated in Figure \ref{fig:Terdragon}.
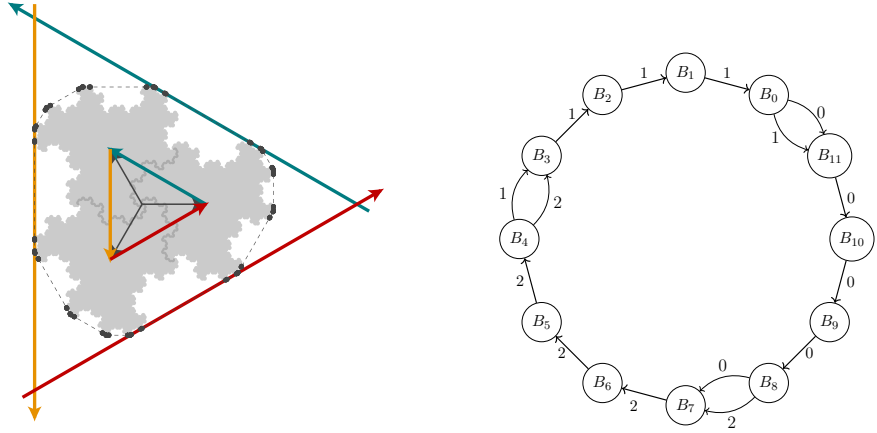
\begin{figure}[!htb]
    \centering
    \includegraphics[width=0.35\linewidth]{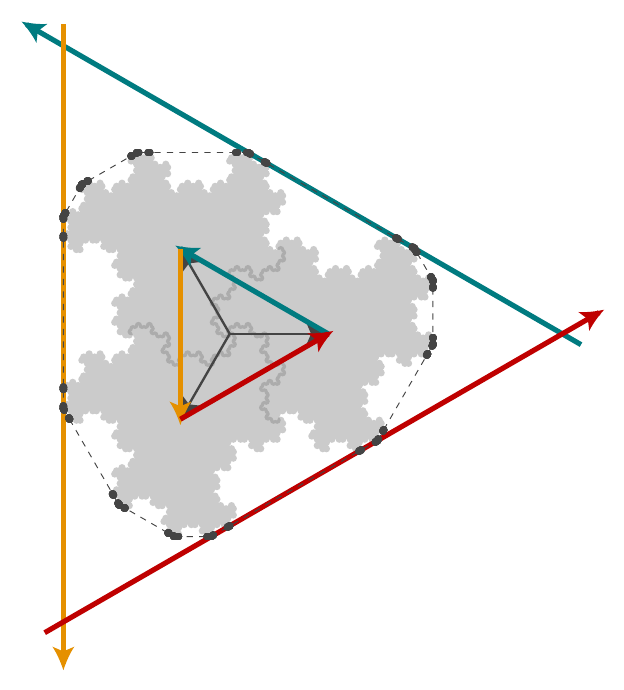}
    \hspace{1cm}
    \begin{tikzpicture}
    
    \node [draw, circle] (0+60) at (0+60:4) (s0) {\small $B_{0}$};
    \node [draw, circle] (30+60) at (30+60:4) (s1) {\small $B_1$};
    \node [draw, circle] (60+60) at (60+60:4) (s2) {\small $B_2$};
    \node [draw, circle] (90+60) at (90+60:4) (s3) {\small $B_3$};
    \node [draw, circle] (120+60) at (120+60:4) (s4) {\small $B_4$};
    \node [draw, circle] (150+60) at (150+60:4) (s5) {\small $B_5$};
    \node [draw, circle] (180+60) at (180+60:4) (s6) {\small $B_6$};
    \node [draw, circle] (210+60) at (210+60:4) (s7) {\small $B_7$};
    \node [draw, circle] (240+60) at (240+60:4) (s8) {\small $B_8$};
    \node [draw, circle] (270+60) at (270+60:4) (s9) {\small $B_9$};
    \node [draw, circle] (300+60) at (300+60:4) (s10) {\small $B_{10}$};
    \node [draw, circle] (330+60) at (330+60:4) (s11) {\small $B_{11}$};
    
    \path[->] (s0) edge [bend right,left] node {$1$} (s11);
    \path[->] (s0) edge [bend left,right] node {$0$} (s11);
    %\draw[->,double,thick] (s0)--(s11) node[midway, below left]{\small $0,1$};
    \draw[->,thick] (s1)--(s0) node[midway, above] {\small $1$};
    \draw[->,thick] (s2)--(s1) node[midway, above] {\small $1$};
    \draw[->,thick] (s3)--(s2) node[midway, above] {\small $1$};
    \path[->] (s4) edge [bend right,right] node {$2$} (s3);
    \path[->] (s4) edge [bend left,left] node {$1$} (s3);
    %\draw[->,double,thick] (s4)--(s3) node[midway, right] {\small $1,2$};
    \draw[->,thick] (s5)--(s4) node[midway, left] {\small $2$};
    \draw[->,thick] (s6)--(s5) node[midway, left] {\small $2$};
    \draw[->,thick] (s7)--(s6) node[midway, below left] {\small $2$};
    \path[->] (s8) edge [bend right,above] node {$0$} (s7);
    \path[->] (s8) edge [bend left,below] node {$2$} (s7);
    %\draw[->,double,thick] (s8)--(s7) node[midway, above] {\small $2,0$};
    \draw[->,thick] (s9)--(s8) node[midway, right] {\small $0$};
    \draw[->,thick] (s10)--(s9) node[midway, right] {\small $0$};
    \draw[->,thick] (s11)--(s10) node[midway, right] {\small $0$};
    
\end{tikzpicture}
    \caption{The extreme points, the convex hull, and bounding set of the Terdragon/Fudgeflake. On the right the DGIFS for $\ext(A)$.}
    \label{fig:Terdragon}
\end{figure}

\subsection{Collinear IFS}
\label{subsect:complexcollinear}
Fix $n\geq2$ and $\lambda\in\mathbb{D}\setminus\{0\}$. We consider an IFS with $n$ transformations but whose translations are along a line. For ease of explanation we choose to parametrize the translation terms but, in fact, the choice does not affect the results as long as the terms are collinear. Set a fixed non-zero $w\in\mathbb{C}$ and let $\Psi_{n}=\{f_k\}_{k=0}^{n-1}$ where \[f_k(z)=\lambda z+(n-2k-1)w.\]

This type of IFS has been studied most recently by \cite{EJS24}. However, we have not found anything in the literature about the convex hull or extreme points of the attractor for this family of IFS.

\begin{remark}
    The argument $\alpha$ of the translation $(n-2k-1)w$ is simply the argument of $w$, which does not depend on $n$. Therefore, there are essentially only two core lines $\ell_{0,1}$ and $\ell_{1,0}$ which are at supplementary angles and have opposite directions: for $m=0,1$ we have $\theta_m=(m+1/2)/2+\alpha$.
\end{remark}

\begin{theorem}
\label{thm:maincplxcollinear}
    Let $\phi$ be such that $\arg(\lambda)=2\pi\phi$.
    \begin{itemize}
        \item If $\phi=p/q$ with coprime integers $p$ and $q$, let $b$ be such that $2b={\textrm lcm}(2,q)$ and set $d=\gcd(2,q)$. Define \[\Theta=\left\{k\phi+\left(m+\frac{1}{2}\right)\frac{1}{2}+\alpha\right\}_{k=0,m=0}^{k=b-1,m=1}\]
        Then:
        \begin{enumerate}
            \item the convex hull of $A$ is a polygon with $2b$ sides at the angles of $\Theta$.
            \item The set of extreme points of $A$ is the unique vector of non-empty compact sets $\{B_j\}_{j=0}^{2b-1}$ which solves the following system of equations: for $m=0$ and $m=1$ 
            \[\begin{cases}
                B_{mb}=\bigcup_{j=0}^{n-1} f_{j}(B_{mb-2p/d}) &\\
                B_k=f_{n-1+m}(B_{k-2p/d}) &
            \end{cases}\]
            where $mb< k< (m+1)b$ (the indices of the sets are to be considered modulo $2b$, while those of the maps modulo $n$).
            \item If the attractor $A$ is convex, then $\lvert\lambda\rvert\geq n^{-1/b}$ and the system of equations identifies the itineraries of the points on $\partial A$.
        \end{enumerate}
        \item If $\phi\not\in\mathbb{Q}$ then 
        \begin{enumerate}
            \item the convex hull is not a polygon.
            \item The attractor $A$ is not convex.
        \end{enumerate}
    \end{itemize}
\end{theorem}
The lemmas of the previous section still apply in this setting but need some adjustments. 
\begin{lemma}
\label{lem:rationalphi_collinear}
     Assume $\phi=p/q$ with coprime integers $p$ and $q$. Let $b$ such that $2b={\textrm lcm}(2,q)$ and set $d=\gcd(2,q)$. Then
     \begin{enumerate}
         \item for all integers $k$ the integer $j=2p/d$ is such that \[(k+b)\phi+\theta_m=k\phi+\theta_{m+j}\] and $b$ is the smallest integer for which the equality holds.
         \item $d$ is the smallest positive integer such that $db\phi\equiv0\mod1$.
    \end{enumerate}
\end{lemma}
\begin{proof}
    Given that \[\gcd(2,q){\textrm lcm}(2,q)=2q\] then $b=q/d$. Setting $a=2/d$, $b$ is the least integer such that $b/q=a/2$ and thus,
    \[b\phi+\theta_m=\frac{bp}{q}+\frac{m\alpha}{2}+\frac{\alpha}{4}=\frac{ap}{2}+\frac{m\alpha}{2}+\frac{\alpha}{4}=\theta_{m+ap}.\] It follows that $j=ap=2p/d$.

    The rest of the argument is equivalent to Lemma \ref{lem:rationalphi}.
\end{proof}
\begin{lemma}
    Let $N$ be any integer. Then if $\phi\not\in\mathbb{Q}$
    \begin{enumerate}
        \item there is no integer $k$ for which $k\phi+\theta_m=N\phi$;
        \item there exist exactly one integers $k$ and one integer $j$ such that $k\phi+\theta_m=N\phi+\theta_{m+j}$.
    \end{enumerate}
\end{lemma}
\begin{proof}
    Assume that there exists a positive integer $k<N$ such that $k\phi+\theta_m=N\phi$ for some integer $N>0$, then \[(N-k)\phi=\theta_m=\left(m+\frac{1}{2}\right)\frac{1}{2}+\alpha.\]
    If $\alpha\in\mathbb{Q}$, then the above equation implies that $\phi$ is rational as well, which is a contradiction. If $\alpha\not\in\mathbb{Q}$, the equation has no solution even if $\phi=\alpha$.

    For part (2), 
    \[k\phi+\left(m+\frac{1}{2}\right)\frac{1}{2}+\alpha=N\phi+\left(m+j+\frac{1}{2}\right)\frac{1}{2}+\alpha\implies (N-k)\phi=-\frac{j}{2}\] which is true only when $k=N$ and $j=0$, since $\phi$ is irrational.
\end{proof}

\begin{proof}[Proof of Theorem \ref{thm:maincplxcollinear}]
    The argument is analogous to that of Theorem \ref{thm:maincplxpolygon} and uses the above two lemmas. The main difference, is that now we there is a translational symmetry rather than a rotational one. The only two bounding lines, which are at angles $\theta_m$ for $m=0,1$, connect extreme points which belong to $f_j(A)$ for all $0\leq j\leq n-1$. Therefore, if $\phi=p/q$ for coprime integers $p$ and $q$ and we let $B_{mb}$ be the set of extreme points on the bounding lines, then for $m=0,1$ \[B_{mb}=\bigcup_{k=0}^{n-1}f_k(B_{mb-2p/d}).\]
    By a similar argument to the one of Theorem \ref{thm:maincplxpolygon}, if $A$ is convex, then each $B_k$ is an interval of positive length and the above equation with the fact that each $v_k$ is equidistant from $v_{k+1}$ imply that $n\lvert\lambda\rvert^b>1$.
    
    Moreover, because $f_0(A)$ and $f_{n-1}(A)$ are the outer most copies of the attractor, the extreme points along line at any angles between $\theta_m$ and $\theta_{m+1}$ have to be in either $f_0(A)$ or $f_{n-1}(A)$. It follows that if $\phi=p/q$ then \[B_k=f_{n-1+m}(B_{k-2p/d})\] for $m=0,1$ and $mb<k<(m+1)b$.
\end{proof}

\begin{corollary}
    With the assumptions of Lemma \ref{lem:rationalphi_collinear} the directed graph associated to the DGIFS in Theorem \ref{thm:maincplxcollinear} has $2/gcd(2,q)$ connected components each of length $b\gcd(2,q)$. Moreover, each connected components is strongly connected.
\end{corollary}
Once again the proof follows the same logic of Corollary 3.2.

\subsection{Examples}
\label{subsect:complexcollinearexamples}
In what follows we will consider other fractals that tile the plane, for more information on them we refer the reader to \cite{BG94}. For all the examples we assume that $w=1$.

\subsubsection{Wild Dragons}
Let $\lambda$ be a root of $1-2x+n x^2$ where $n$ is the number of functions of the IFS, then 
\[\lambda=\frac{1\pm\I\sqrt{n-1}}{n}\implies\lvert\lambda\rvert=n^{-1/2},\quad\arg(\lambda)=\arctan\left(\pm\sqrt{n-1}\right).\]
The attractor of the IFS $\Psi_n$ with these parameters are called \emph{wild $n$-dragons}. Their convex hull is a polygon only when $n=2$ and $n=4$, since only in those cases the argument of $\lambda$ is a rational multiple of $2\pi$. 

When $n=2$, $\lambda=2^{-1}(1+\I)$ and the attractor is the famous Twin Dragon that we mentioned in the introduction. Since $\phi=1/8$, Theorem \ref{thm:maincplxcollinear} tells us that the convex hull is an octagon and $\ext(A)$ is the union of Cantor sets which are related according to the directed graph in Figure \ref{fig:TwinDragon}.
%\[\begin{cases}
%     B_0=f_0(B_7)\cup f_1(B_7)& \\
%     B_1=f_1(B_0)& \\
%     B_2=f_1(B_1)& \\
%     B_3=f_1(B_2)& \\
%     B_4=f_0(B_3)\cup f_1(B_3)& \\
%     B_5=f_0(B_4)& \\
%     B_6=f_0(B_5)& \\
%     B_7=f_0(B_6)& 
% \end{cases}.\]
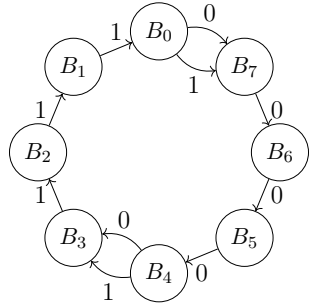
\begin{figure}
    \centering
    % \includegraphics[width=0.4\linewidth]{TwinDragon_hull.pdf}
    % \hspace{0.7cm}
    \begin{tikzpicture}
    
    \node [draw, circle] (90) at (90:2) (s0) {\small $B_0$};
    \node [draw, circle] (135) at (135:2) (s1) {\small $B_1$};
    \node [draw, circle] (180) at (180:2) (s2) {\small $B_2$};
    \node [draw, circle] (225) at (225:2) (s3) {\small $B_3$};
    \node [draw, circle] (270) at (270:2) (s4) {\small $B_4$};
    \node [draw, circle] (315) at (315:2) (s5) {\small $B_5$};
    \node [draw, circle] (0) at (0:2) (s6) {\small $B_6$};
    \node [draw, circle] (45) at (45:2) (s7) {\small $B_7$};
        
    \path[->] (s0) edge [bend left, above] node {$0$} (s7);
    \path[->] (s0) edge [bend right, below] node {$1$} (s7);
    \path[->] (s1) edge [midway, above] node {$1$} (s0);
    \path[->] (s2) edge [midway, left] node {$1$} (s1);
    \path[->] (s3) edge [midway, left] node {$1$} (s2);
    \path[->] (s4) edge [bend right, above] node {$0$} (s3);
    \path[->] (s4) edge [bend left, below] node {$1$} (s3);
    \path[->] (s5) edge [midway, below] node {$0$} (s4);
    \path[->] (s6) edge [midway, right] node {$0$} (s5);
    \path[->] (s7) edge [midway, right] node {$0$} (s6);
    
\end{tikzpicture}
    \caption{The sofic system describing the extreme points of the Twin dragon.}
    \label{fig:TwinDragon}
\end{figure}

When $n=4$, $\lambda=4^{-1}\left(1+\I\sqrt{3}\right)$ which implies $\phi=1/6$. By Theorem \ref{thm:maincplxcollinear} the convex hull is an hexagon and the extreme points are described by the sofic system in Figure \ref{fig:Wild4Dragon}.
% \[\begin{cases}
%     B_0=f_0(B_5)\cup f_1(B_5)\cup f_2(B_5)\cup f_3(B_5)& \\
%     B_1=f_3(B_0)& \\
%     B_2=f_3(B_1)& \\
%     B_3=f_0(B_2)\cup f_1(B_2)\cup f_2(B_2)\cup f_3(B_2)& \\
%     B_4=f_0(B_3)& \\
%     B_5=f_0(B_4)& 
% \end{cases}.\]
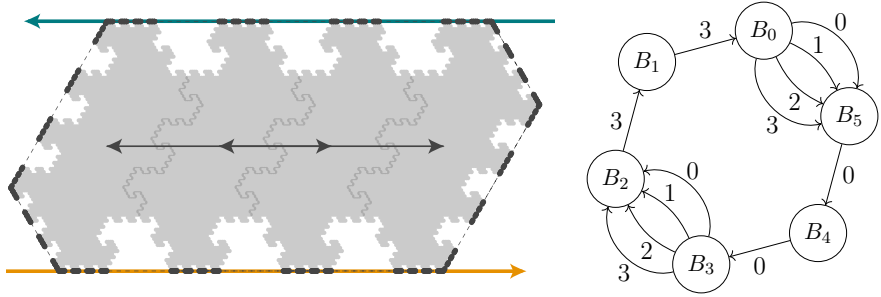
\begin{figure}
    \centering
    \includegraphics[width=0.5\linewidth]{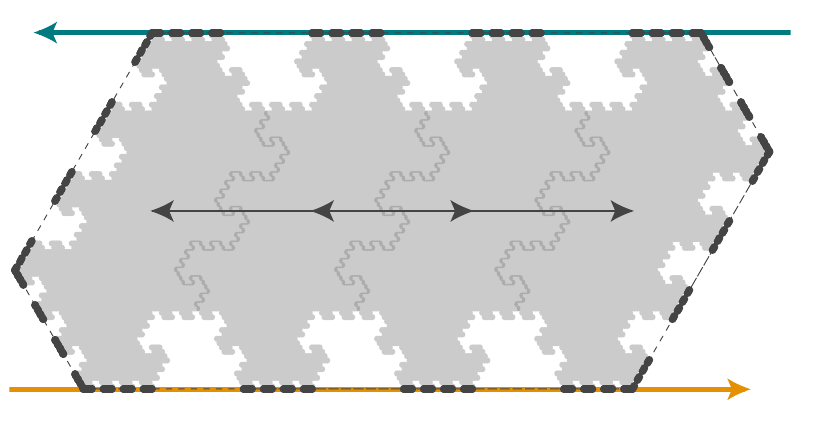}
    % \hspace{0.7cm}
    \begin{tikzpicture}
    
    \node [draw, circle] (60+15) at (60+15:2) (s0) {\small $B_0$};
    \node [draw, circle] (120+15) at (120+15:2) (s1) {\small $B_1$};
    \node [draw, circle] (180+15) at (180+15:2) (s2) {\small $B_2$};
    \node [draw, circle] (240+15) at (240+15:2) (s3) {\small $B_3$};
    \node [draw, circle] (300+15) at (300+15:2) (s4) {\small $B_4$};
    \node [draw, circle] (0+15) at (0+15:2) (s5) {\small $B_5$};
        
    \path[->] (s0) edge [bend left=60, above] node {$0$} (s5);
    \path[->] (s0) edge [bend left=20, above] node {$1$} (s5);
    \path[->] (s0) edge [bend right=20, below] node {$2$} (s5);
    \path[->] (s0) edge [bend right=60, below] node {$3$} (s5);
    \path[->] (s1) edge [midway, above] node {$3$} (s0);
    \path[->] (s2) edge [midway, left] node {$3$} (s1);
    \path[->] (s3) edge [bend right=60, above] node {$0$} (s2);
    \path[->] (s3) edge [bend right=20, above] node {$1$} (s2);
    \path[->] (s3) edge [bend left=20, below] node {$2$} (s2);
    \path[->] (s3) edge [bend left=60, below] node {$3$} (s2);
    \path[->] (s4) edge [midway, below] node {$0$} (s3);
    \path[->] (s5) edge [midway, right] node {$0$} (s4);
    
\end{tikzpicture}
    \hspace{0.7cm}
    \caption{The extreme points, the convex hull, and bounding set of the Wild $4$-Dragon. On the right the associated sofic system for $\ext(A)$.}
    \label{fig:Wild4Dragon}
\end{figure}

\subsubsection{Tame Dragons}
Let $\lambda$ be a root of $1-x+n x^2$ where $n$ is the number of functions of the IFS, then 
\[\lambda=\frac{1\pm\I\sqrt{4n-1}}{n}\implies\lvert\lambda\rvert=n^{-1/2},\quad\arg(\lambda)=\arctan\left(\pm\sqrt{4n-1}\right).\]
For every $n\ge2$ the argument of $\lambda$ is irrational. Therefore, the convex hull does not have finitely many sides and so we can only approximate it. We illustrate how approximations of $\phi$ produce subsets of $A$, which approximate $\ext(A)$. 

Fix $n=3$. Good rational approximations of an irrational number are obtained by truncating its infinite continued fraction. Therefore, since $\phi=[0;4,1,10,1,101,1,\ldots]$ the first three approximations are $[0;4]=1/4$, $[0;4,1]=1/5$, and $[0;4,1,10]=11/54$. We show in Figure \ref{fig:TameDragons} which points in the attractor are described by the resulting DGIFS assuming one of those approximations. Note that if $\phi$ were in fact equal to those rational number, then the convex hull of the attractor would be a polygon with $4$, $10$, or $54$ sides.
\begin{figure}
    \centering
    \includegraphics[width=0.4\linewidth]{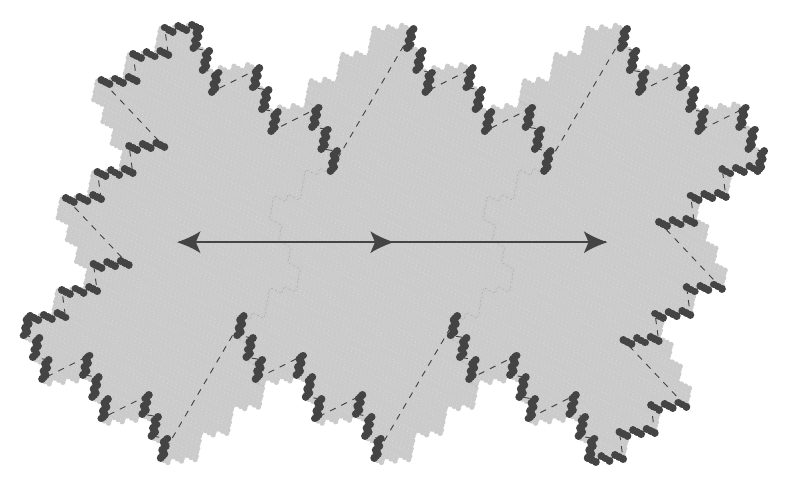}
    % \hspace{1cm}
    \includegraphics[width=0.4\linewidth]{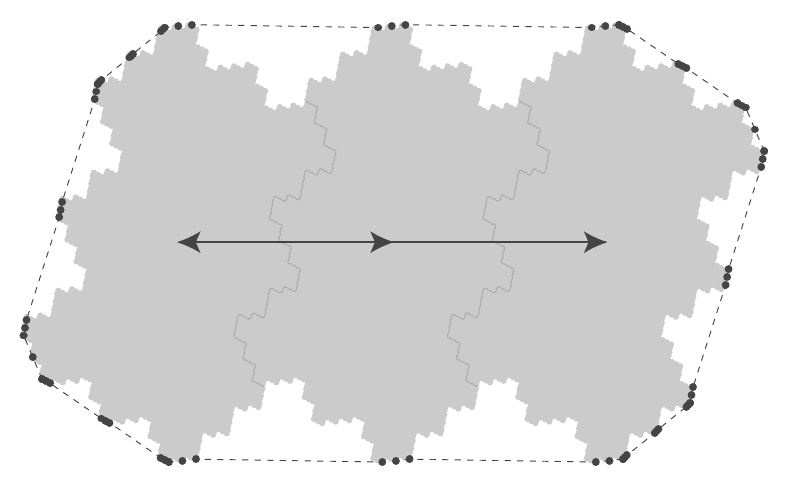}%
    \includegraphics[width=0.4\linewidth]{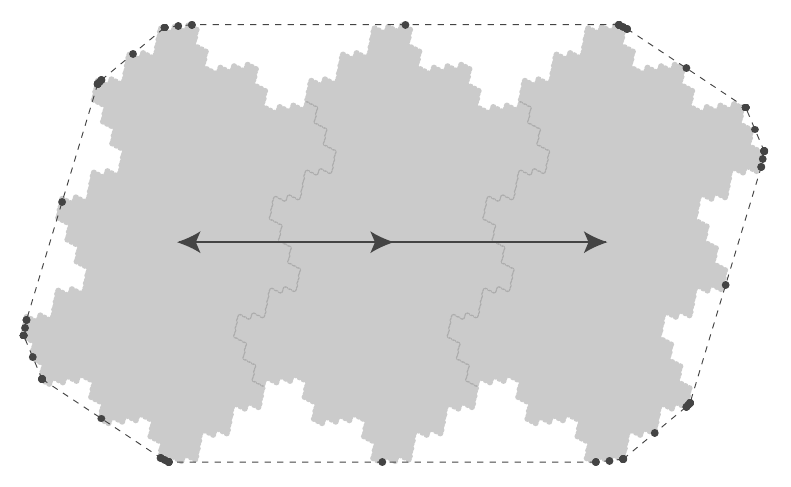}%
    \caption{The Tame $3$-Dragon and the points described by a DGIFS obtained by approximating the irrational $\phi=(2\pi)^{-1}\arctan\left(\sqrt{11}\right)$ with $1/4$, $1/5$, and $11/54$.}
    \label{fig:TameDragons}
\end{figure}

\section{The real case}
\label{sect:real}
The second family of IFS that we will consider is the case when the affine transformations have a contracting matrix of the form $T=\begin{pmatrix}\mu&0\\0&\eta\end{pmatrix}$ for non-zero real $-1<\mu,\eta<1$.

Due to the absence of a rotational component and the differing scaling factors in $T$, the equations governing the DGIFS are sensitive to the choice of translations. Thanks to Proposition \ref{prop:extremeconvexcore}, however, we are guaranteed to find a sofic system on the code space whose projection is $\ext(A)$. We first find such system for two specific families of translations to build intuition and, then, generalize it to arbitrary translations. The first family of IFS is the collinear one, where the translation vectors lie along a given line; this is closely connected to $\beta$-expansions \cite{S94,G08,HS17,T17}, an active area of research. The second family is that of Bedford-McMullen IFS, whose attractors have been studied since they were introduced independently by Bedford \cite{B84} and McMullen \cite{M84}. This last family can be easily viewed as the arbitrary translation case when both $\mu$ and $\eta$ are positive.

In contrast to the results of Section \ref{sect:complex}, a particular consequence of what we will discover is that the itinerary of corner points are not necessarily periodic but can be pre-periodic.

\subsection{Collinear IFS}
\label{subsect:realcollinear}
Fix $n\geq2$ and let $\vec{w}$ be any vector in $\mathbb{R}^2$ excluding those along the axes. Consider the IFS $\Psi_n=\{F_k\}_{k=0}^{n-1}$ generated by the maps \[F_k(\vec{x})=T\vec{x}+(n-2k-1)\vec{w}\quad 0\leq k\leq n-1\]
and denote with $A$ the attractor of such an IFS. Note that we discarded vectors $w$ along the axes in order to avoid dealing with the degenerate cases: the attractor in these cases is contained along the axes so its convex hull is an interval whose endpoints are the fixed points of $F_0$ and $F_{n-1}$.

Note that the itineraries of the extreme points of the attractor for the IFS $\Psi_2$ have been identified in \cite{HS16, HS17} when $0<\mu<\eta$, and \cite{HS16, R25} when $\mu<0<\eta$ with $\lvert\mu\rvert<\eta$ using analytical arguments. We propose a new geometrical proof for the same results and extend them to the more general case of $n\geq2$ maps.
\begin{theorem}
\label{thm:mainrealcollinear}
    The set of extreme points of the attractor associated to the IFS $\Psi_n$ is the unique vector solution to the following system of equations
    \begin{itemize}
        \item if $\mu<0<\eta$ with $\lvert\mu\rvert<\eta$,
        \[\begin{cases}
            B_0=F_0(B_0)&\\
            B_1=F_{n-1}(B_1)&\\
            B_2=\bigcup_{k=0}^{n-2}F_k(B_0) \cup F_{n-1}(B_3)&\\
            B_3=F_0(B_2) \cup \bigcup_{k=1}^{n-1}F_k(B_1)&
        \end{cases};\]
        \item if $\mu<0<\eta$ with $\lvert\mu\rvert=\eta$,
        \[\begin{cases}
            B_0=\bigcup_{k=0}^{n-1}F_k(B_1)&\\
            B_1=F_{n-1}(B_0)&\\
            B_2=\bigcup_{k=0}^{n-1}F_k(B_3)&\\
            B_3=F_{0}(B_2)&
        \end{cases}\]
        the convex hull is a quadrilateral. If the attractor is convex, then $\eta\geq n^{-1/2}$ and the each $B_k$ is an edge of $\partial A$;
        \item if $0<\mu<\eta$, 
        \[\begin{cases}
            B_0=F_0(B_0)&\\
            B_1=\bigcup_{k=0}^{n-2}F_k(B_0) \cup F_{n-1}(B_1)&\\
            B_2=F_{0}(B_2)\cup\bigcup_{k=1}^{n-1}F_k(B_3) &\\
            B_3=F_{n-1}(B_3)&
        \end{cases}.\]
        \item if $0<\mu=\eta$, the attractor is either a Cantor set ($\eta<1/n$) or an interval ($\eta\geq1/n$), so the convex hull is an interval along the line spanned by $\vec{w}$ whose endpoint are the fixed points of $F_0$ and $F_{n-1}$.
    \end{itemize}
\end{theorem}
\begin{proof}
Some important properties of the attractor $A$ are:
\begin{enumerate}[label=(\roman*)]
    \item The attractor is symmetric about $0$;
    \item the attractor has translational symmetry along the line through the origin spanned by $\vec{w}$: $F_k(A)$, for $0<k\leq n-1$, is a copy of $F_0(A)$ translated by $(2-2k)\vec{w}$;
\end{enumerate}
Just like we discussed in Section \ref{subsect:complexcollinear} there are only two core lines $\ell_{1,0}$ and $\ell_{0,1}$ both parallel to $\vec{w}$, but with different direction. Let $\alpha$ be the angle that $\vec{w}$ makes with the horizontal axis; let $\theta_0=\alpha+1/4\bmod1$ and $\theta_1=\theta_0+1/2\bmod 1$ indicate the normal directions of each core line.
The symmetries of the attractor mentioned above imply that any edge of $\conv(A)$ at angle $\pm t\in(\theta_0,\theta_1)$ must go through extreme points of  either $F_0(A)$ or $F_{n-1}(A)$.

The result when $0<\mu=\eta$ is then a direct consequence. In fact, since $T$ is a scalar multiple of the identity, then the limit set is a Cantor set along a core line. In particular, if we let $d$ be the diameter of $A$, then, since the length of $F_k(A)$ is $\eta d$ for every $0\leq k\leq n-1$, the minimum value so that $A=\bigcup_{k=0}^{n-1}F_k(A)$ is connected is $\eta=1/n$. Another way to think about this case, is to observe that it is equivalent to the complex collinear case with $\lambda=\eta$.

Consider the case when $\mu<0<\eta$ with $\lvert\mu\rvert=\eta$. From Proposition \ref{prop:extremeconvexcore} we know that any edge of $\conv(A)$ is eventually covered by images of bounding lines. The line $\widehat\ell_{1,0}$ is at angle $\theta_0$ and by definition intersects the attractor at extreme points that belong to $F_k(A)$ for each $0\leq k<n$. Similarly, $\widehat\ell_{0,1}$ is at angle $\theta_1$ and intersects each $F_k(A)$ non-trivially. Since $\lvert\mu\rvert=\eta$ and $\mu<0$ then one application of $F_k$ changes the sign of the slope of any line: $F_k(\widehat\ell_{1,0})$ is at angle $\theta_0+1/4$ while $F_j(F_k(\widehat\ell_{1,0}))$ at $\theta_0+1/2=\theta_1$. We emphasize that $F_{jk}(\widehat\ell_{1,0})$ is parallel to $\widehat\ell_{0,1}$, not that it is equal to it. The DGIFS is readily obtained: let $B_0$ indicate the extreme points on $\widehat\ell_{1,0}$, then they must be the image under $F_k$ for any $0\leq k<n$ of extreme points, call them $B_1$, which are at angle $\theta_0-1/4=\alpha$, that is \[B_0=F_0(B_1)\cup F_1(B_1)\cup\cdots\cup F_{n-1}(B_1) \text{ and } B_1=F_{n-1}(B_0);\]by symmetry we have
\[ B_2=\bigcup_{k=0}^{n-1}F_k(B_3) \text{ and } B_3=F_0(B_2).\]
It follows that $\conv(A)$ is a quadrilateral. However, the attractor might be totally disconnected and thus each $B_k$, for $0\leq k\leq4$, is a Cantor set. The attractor $A$ is connected, and thus convex in this case, if each $B_k$ is also connected: notice that we can rewrite the equation for $B_0$ using the one for $B_1$
\[B_0=\bigcup_{k=0}^{n-1}F_k(B_1)=\bigcup_{k=0}^{n-1}F_k(F_{n-1}(B_0))\]
and since $F_k\circ F_{n-1}$ contracts by $\eta^2$, then the minimum value for which $B_0$ is connected is $n\eta^{2}\geq 1$.

Let now $0<\mu<\eta$. The image $F_k(\widehat\ell_{1,0})$, for any $0\leq k<n$, is a line at angle $t>\theta_0$, because $\mu<\eta$. In particular, a vertical line will remain fixed under the action of $F_j$. By definition of a bounding line, the attractor $A$ is in its closed left half plane, thus $A$ is in the closed left half plane of $F_j(\widehat\ell_{1,0})$ only if $j=n-1$. The same reasoning can be applied to the line $F_{n-1}(\widehat\ell_{1,0})$ to obtain that $A$ also lies in the closed left half plane of $F_{n-1}(F_{n-1}(\widehat\ell_{1,0}))$. Therefore, the attractor is in the closed left half plane of $F_\omega(\widehat\ell_{1,0})$ for $\omega=w_0w_1w_2\cdots w_k$ with $w_j=n-1$. In the limit, $F_\omega(\widehat\ell_{1,0})$, where $\omega$ is the infinite sequence of the digit $n-1$, is a vertical line and it must be tangent to $A$. We conclude then that the fixed point, $p$, of $F_{n-1}$ must be in $\ext(A)$. Moreover, the bounding line $\widehat\ell_{0,1}$ must intersect this vertical line at $p$ since $\widehat\ell_{0,1}$ goes through some extreme points of $A$ and the images $F_k(\widehat\ell_{0,1})$ are steeper than it.

Now, since $p\in\widehat\ell_{0,1}$ then so are $F_k(p)$ for any $0\leq k<n$. By symmetry, $F_k(q)$ for any $0\leq k<n$ are elements of $\widehat\ell_{1,0}$, where $q$ is the fixed point of $F_0$.

The system of equation defining the DGIFS is then constructed as follows. Let $B_0=F_0(B_0)$ and set $B_1$ to be the set of extreme points which are on $\widehat\ell_{1,0}$, that is $F_k(B_0)$, and their images under $F_{n-1}$:\[B_1=\bigcup_{k=0}^{n-2}F_k(B_0)\cup F_{n-1}(B_1).\] By symmetry we obtain the remaining equations \[B_2= F_{0}(B_2)\cup \bigcup_{k=1}^{n-1}F_k(B_3) \qquad B_3=F_{n-1}(B_3).\]

The argument for the case $\mu<0<\eta$ with $\rvert\mu\lvert<\eta$ is quite similar: since $\mu<0$, applying $T$ to any vector $\vec{v}$ reflects it across the $y$-axis and rescales both coordinates. Therefore, given that by definition $A$ is in the left half plane of $\widehat\ell_{1,0}$, we need that, for $0\leq j,k<n$, $F_{k}(\widehat\ell_{1,0})$ contains $A$ in its right half plane and, subsequently, that $A$ lies in the left half plane of $F_{jk}(\widehat\ell_{1,0})$. It follows that we want to alternate between applying $F_0$ and $F_{n-1}$. Observe, also, that because of the reflection across the $y$-axis, the bounding line $\widehat\ell_{1,0}$ must intersect $F_0(\widehat\ell_{1,0})$ and such intersection point $p$ is necessarily the fixed point of $F_0$. Thus, $p$ and $F_k(p)$ for any $0\leq k<n$ are extreme points of $A$ on the bounding line $\widehat\ell_{1,0}$. By symmetry we have that $q=F_{n-1}(q)$ and $F_k(q)$ for any $0\leq k<n$ are in $\ext(A)$ and on the line $\widehat\ell_{0,1}$.

The DGIFS describing points of $\ext(A)$ is then given by
\[\begin{cases}
    B_0=F_0(B_0)&\\
    B_1=F_{n-1}(B_1)&\\
    B_2=\bigcup_{k=0}^{n-2}F_k(B_0) \cup F_{n-1}(B_3)&\\
    B_3=F_0(B_2) \cup \bigcup_{k=1}^{n-1}F_k(B_1)&
\end{cases}.\]

\end{proof}

\begin{corollary}
    When $-1<\mu<0<\eta<1$, the DGIFS given in Theorem \ref{thm:mainrealcollinear} has one connected component but it is not strongly connected. In the remaining cases, the directed graph has exactly two connected components. But only in the cases where $\lvert\mu\rvert=\eta$ each of the components is strongly connected.
\end{corollary}
\begin{remark}
    Observe that the system of equations in each case of Theorem \ref{thm:mainrealcollinear} can not be decoupled. Thus, in contrast to what we did in Section \ref{sect:complex}, we have no hope to find an IFS of affine transformations whose attractor is $B_k$ for $0\leq k\leq3$.
\end{remark}

\subsection{Examples}
\label{subsect:realcollinearexamples}
With this type of IFS unfortunately we do not have particularly famous attractors, so we reuse some of the examples chosen by previous authors. In all cases, we fix $w=1+\I$.

\subsubsection{Positive $\eta$ and negative $\mu$}
This is an example taken from \cite[Theorem 5.2]{R25}. Let $n=2$ and set $\mu=-10/17$, $\eta=10/13$%, then $\ext(A)$ is described by 
% \[\begin{cases}
%     B_0=F_0(B_0)&\\
%     B_1=F_1(B_1)&\\
%     B_2=F_0(B_0)\cup F_1(B_3)&\\
%     B_3=F_0(B_2)\cup F_1(B_1)&
% \end{cases}\]
\begin{figure}
    \centering
    \includegraphics[width=0.25\linewidth]{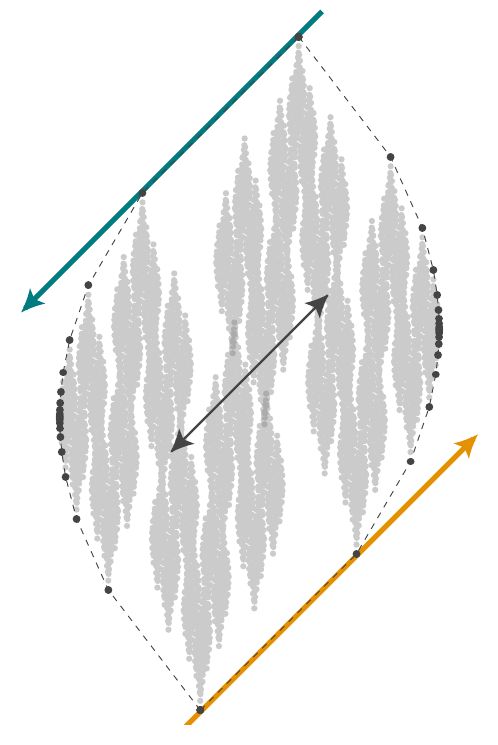}
    % \hspace{1cm}
    \begin{tikzpicture}
    
    \node [draw, circle] (90) at (90:2) (s0) {\small $B_0$};
    \node [draw, circle] (270) at (270:2) (s1) {\small $B_1$};
    \node [draw, circle] (0) at (0:2) (s2) {\small $B_2$};
    \node [draw, circle] (180) at (180:2) (s3) {\small $B_3$};
        
    \path (s0) edge [loop left] node {$0$} (s0);
    \path (s1) edge [loop right] node {$1$} (s1);
    \path[->] (s2) edge [midway, bend right, above] node {$1$} (s3);
    \path[->] (s2) edge [midway, right] node {$0$} (s0);
    \path[->] (s3) edge [midway, bend right, below] node {$0$} (s2);
    \path[->] (s3) edge [midway, left] node {$1$} (s1);
    
\end{tikzpicture}
    \caption{On the left, the extreme points, the convex hull, and bounding set of the attractor of the IFS $\Psi_2$ when $\mu=-10/17$ and $\eta=10/13$. On the right, the DGIFS for $\ext(A)$.}
    \label{fig:RealPosNeg}
\end{figure}
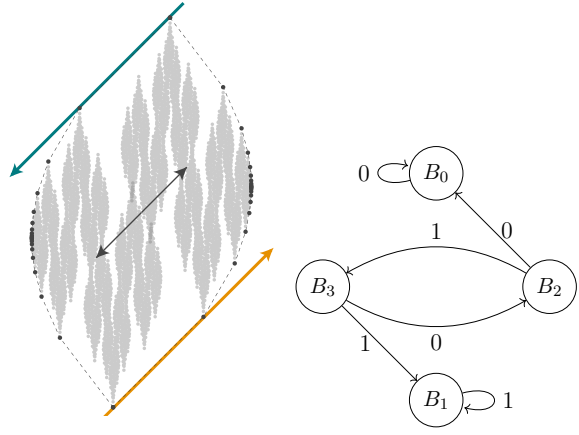
The attractor inside its convex hull is shown in Figure \ref{fig:RealPosNeg}. Note that the sets $B_2$ and $B_3$ are the infinite collection of points to the right and left side of the attractor, respectively. Moreover, $B_2$ contains the fixed point $B_0$ while $B_1$ is contained in $B_3$.

\subsubsection{Positive $\eta$ and positive $\mu$}
Let $n=2$ and let $1/\mu\approx1.81618$, $1/\eta\approx1.30022$ be two of the roots of $x^5-2x^4+2$. The authors in \cite{HS17} have shown that the origin is the only point belonging to $F_0(A)\cup F_1(A)$, by essentially proving that $0$ is the unique element of $F_0(\conv(A))\cap F_1(\conv(A))$.
% The set of extreme points is described by
% \[\begin{cases}
%     B_0=F_0(B_0)&\\
%     B_1=F_0(B_0)\cup F_1(B_1)&\\
%     B_2=F_0(B_2)\cup F_1(B_3)&\\
%     B_3=F_1(B_3)&
% \end{cases}\]
The set of extreme points is shown in Figure \ref{fig:RealPosPos}.
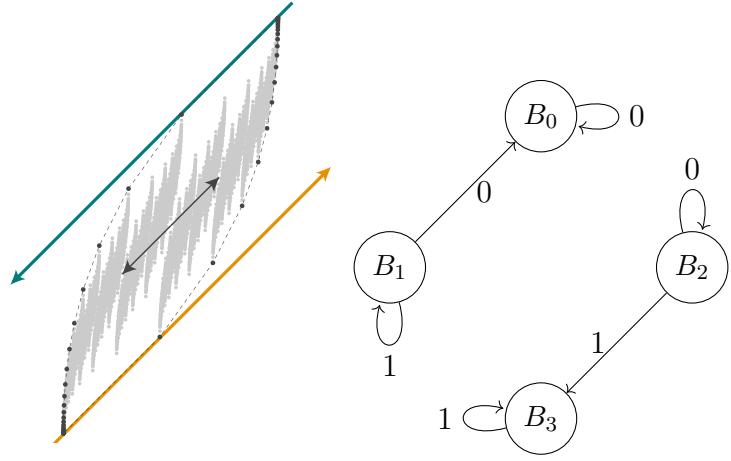
\begin{figure}
    \centering
    \includegraphics[width=0.3\linewidth]{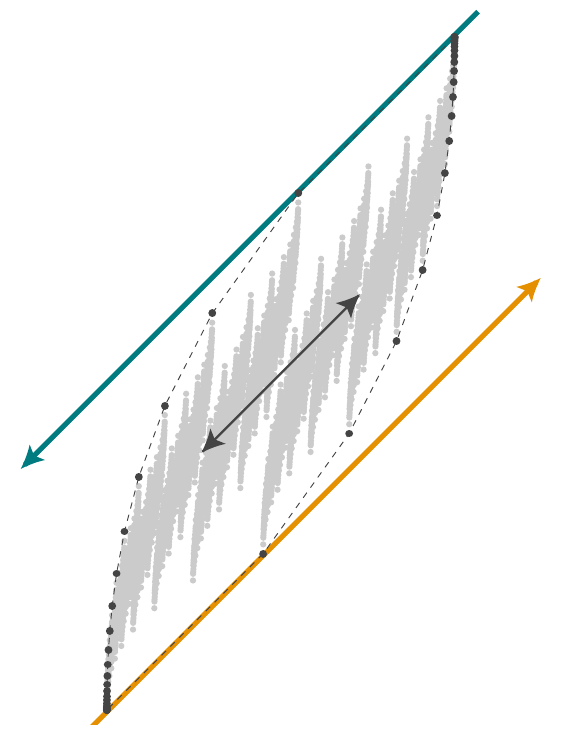}
    % \hspace{1cm}
    \begin{tikzpicture}
    
    \node [draw, circle] (90) at (90:2) (s0) {\small $B_0$};
    \node [draw, circle] (180) at (180:2) (s1) {\small $B_1$};
    \node [draw, circle] (0) at (0:2) (s2) {\small $B_2$};
    \node [draw, circle] (270) at (270:2) (s3) {\small $B_3$};
        
    \path (s0) edge [loop right] node {$0$} (s0);
    \path (s1) edge [loop below] node {$1$} (s1);
    \path[->] (s1) edge [midway, right] node {$0$} (s0);
    \path (s3) edge [loop left] node {$1$} (s3);
    \path[->] (s2) edge [loop above] node {$0$} (s2);
    \path[->] (s2) edge [midway, left] node {$1$} (s3);
    
\end{tikzpicture}
    \caption{On the left, the extreme points, the convex hull, and bounding set of the attractor of the IFS $\Psi_2$ when $0<\mu<\eta$ with $\mu^{-1}$ and $\eta^{-1}$ roots of $x^5-2x^4+2$. On the right, the DGIFS for $\ext(A)$.}
    \label{fig:RealPosPos}
\end{figure}
\subsubsection{Equal absolute value}
Let $n=4$ and set $\eta=-\mu=n^{-1/2}=0.5$ then by Theorem \ref{thm:mainrealcollinear} the convex hull is a quadrilateral. 
In particular, the attractor is also convex. See Figure \ref{fig:RealEqual}.
% \[\begin{cases}
%     B_0=F_0(B_1)\cup F_1(B_1)\cup F_2(B_1)\cup F_3(B_1)&\\
%     B_1=F_3(B_0)&\\
%     B_2=F_0(B_3)\cup F_1(B_3)\cup F_2(B_3)\cup F_3(B_3)&\\
%     B_3=F_0(B_2)&
% \end{cases}\]
\begin{figure}
    \centering
    \includegraphics[width=0.4\linewidth]{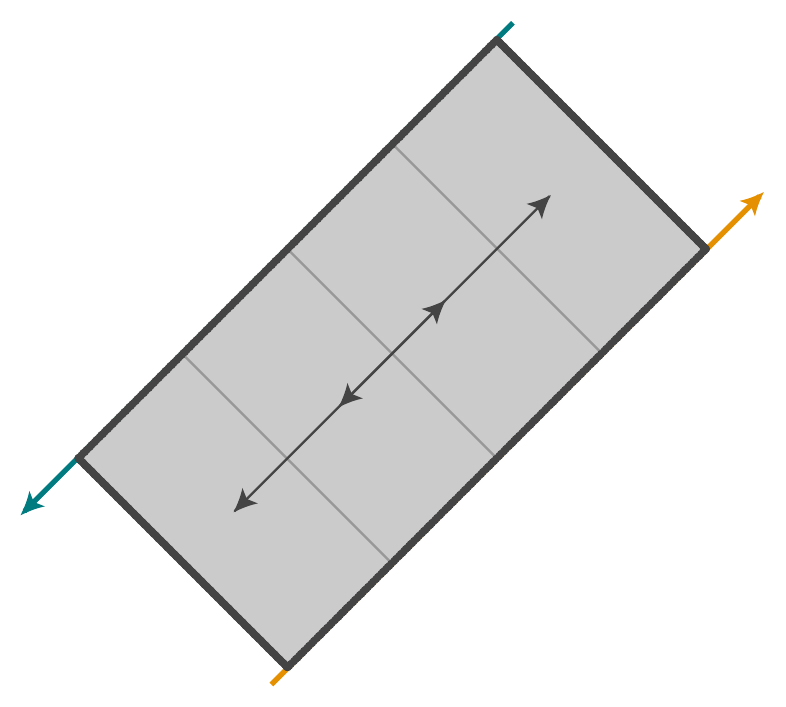}
    % \hspace{1cm}
    \begin{tikzpicture}
    
    \node [draw, circle] (90) at (90:2) (s0) {\small $B_0$};
    \node [draw, circle] (180) at (180:2) (s1) {\small $B_1$};
    \node [draw, circle] (0) at (0:2) (s2) {\small $B_2$};
    \node [draw, circle] (270) at (270:2) (s3) {\small $B_3$};
        
    \path[->] (s0) edge [bend right=100, above] node {$0$} (s1);
    \path[->] (s0) edge [bend right=60, above] node {$1$} (s1);
    \path[->] (s0) edge [bend right=35, above] node {$2$} (s1);
    \path[->] (s0) edge [bend right=10, above] node {$3$} (s1);
    \path[->] (s1) edge [bend right, below] node {$3$} (s0);
    
    \path[->] (s2) edge [bend left=100, below] node {$0$} (s3);
    \path[->] (s2) edge [bend left=60, below] node {$1$} (s3);
    \path[->] (s2) edge [bend left=35, below] node {$2$} (s3);
    \path[->] (s2) edge [bend left=10, below] node {$3$} (s3);
    \path[->] (s3) edge [bend left, above] node {$0$} (s2);
        
\end{tikzpicture}
    \caption{On the left, the boundary points and the bounding set of the attractor of the IFS $\Psi_2$ when $\eta=-\mu=0.5$. On the right, the directed graph relating the sides of the convex hull.}
    \label{fig:RealEqual}
\end{figure}
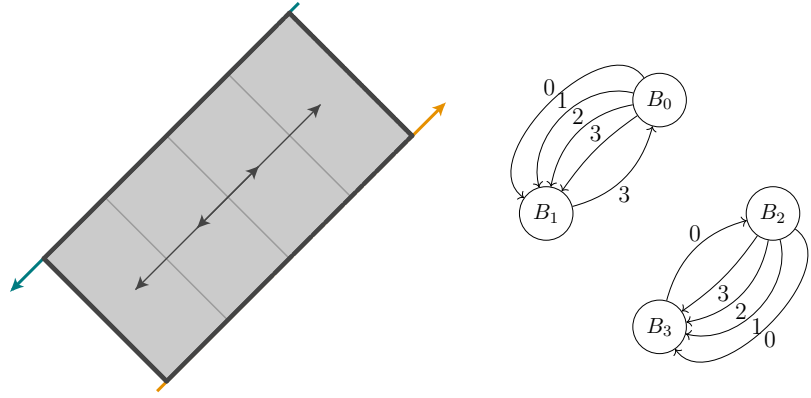

\subsection{Bedford-McMullen Carpets}
\label{subsect:realcarpets}
Fix integers $n>m\geq2$ and a finite non-empty \emph{digit set} $\mathcal{D}$ which is a subset of $\{0,\cdots,n-1\}\times\{0,\cdots,m-1\}$. The Bedford-McMullen IFS consists of affine maps \[F_d(\vec{x})=T(\vec{x}+d),\qquad d\in\mathcal{D}\]
where $\mu=1/n$ and $\eta=1/m$. The attractor $A$, often called a Bedford-McMullen carpet, is the unique non-empty compact set satisfying \[A=\bigcup_{d\in \mathcal{D}}F_d(A).\]

Usually, a more visual way is favored when defining the Bedford-McMullen IFS. Divide the unit square, $[0, 1]^2$ into an $n \times m$ grid. Select a subset of the rectangles formed by the grid and consider the IFS consisting of the affine maps which map $[0, 1]^2$ onto each chosen rectangle while preserving orientation.

Observe that with this setup we cannot hope to find a unique system of equations that works for all the possible choices of $\mathcal{D}$. However, we will argue that it is always possible to find a DGIFS, with finitely many equations, for the extreme points of $A$. In fact, the proof is constructive and an example will be shown in the next section.

\begin{theorem}
\label{thm:mainrealcarpet}
    Fix integers $n>m\geq2$ and consider a Bedford-McMullen IFS $\{F_d\}_{d\in\mathcal{D}}$ for a non-empty digit set $\mathcal{D}\subset\{0,\cdots,n\}\times\{0,\cdots,m\}$. The extreme points of the limit set $A$ are described by a system of finitely many equations and functions $F_d$ where  $d\in\overline{\conv(\mathcal{D})}\setminus\conv(\mathcal{D})$.
\end{theorem}
\begin{proof}
    From Proposition \ref{prop:extremeconvexcore} we know that edges of $\conv(A)$ are contained in images of bounding lines to $A$ and these are lines parallel to the core lines. 
    
    Now, suppose $\ell$ is a core line containing an edge of $\conv(\mathcal{D})$. Let $\mathcal{D}(\ell):=\{d_j\}_{j=0}^{r_\ell}$ denote the ordered digits through which $\ell$ passes. Finally, let $\widehat\ell$ be the bounding line to $A$ parallel to $\ell$ with the attractor contained in its left half plane. There are two cases to consider: either $\ell$ has non-zero and finite slope, or it is a vertical or horizontal line. 
    
    In the former case, by the argument in the proof of Theorem \ref{thm:mainrealcollinear} (when $0<\mu<\eta$) the extreme points belong to $\widehat\ell$ and to its images under $F_{d_{r_\ell}}$. Therefore, the extreme points are described by the following pair of equations:
    \[B_{0}=F_{d_0}(B_0),\qquad B_1=\bigcup_{k=0}^{r_\ell-1}F_{d_k}(B_0)\cup F_{d_{r_\ell}}(B_1).\]
    
    Let $\widehat\ell$ be a vertical or horizontal bounding line, then its slope does not change under application of $F_d$, but only its relative position. Consequently, the extreme points of $A$ on $\widehat\ell$ are described by 
    \[B_2=\bigcup_{d\in\mathcal{D}(\ell)}F_d(B_2).\]

    The sofic system on the symbol space $\{d_0,d_1,\ldots,d_{n-1}\}$ describing the itineraries of all extreme points of the attractor $A$ is then easily found by applying the arguments above for all the sides of $\conv(\mathcal{D})$.
\end{proof}

\subsection{Example}
\label{subsect:realcarpetsexamples}
Let $n=5$, $m=3$, and choose the following set of digits \[\mathcal{D}=\left\{ d_0=\begin{pmatrix}0\\1\end{pmatrix},d_1=\begin{pmatrix}1\\2\end{pmatrix},d_2=\begin{pmatrix}2\\2\end{pmatrix},d_3=\begin{pmatrix}2\\1\end{pmatrix},d_4=\begin{pmatrix}3\\1\end{pmatrix},d_5=\begin{pmatrix}3\\0\end{pmatrix} \right\}.\]
The convex hull of $\mathcal{D}$ is easily found to be the polygon with vertices 
$\{ d_0,d_1,d_2,d_4,d_5 \}$.
Let $\ell_k$ for $0\leq k\leq3$ denote the core lines, starting from the top horizontal line and continuing counterclockwise. The corresponding digits to each of these lines are the digits
\[
    \mathcal{D}(\ell_0)=\{d_2,d_1\}\quad\mathcal{D}(\ell_1)=\{d_1,d_0\}\quad
    \mathcal{D}(\ell_2)=\{d_0,d_5\}\quad\mathcal{D}(\ell_3)=\{d_5,d_4,d_2\}.
\]
From the proof of Theorem \ref{thm:mainrealcarpet} we obtain the DGIFS shown in Figure \ref{fig:Carpet}.
% \[\begin{cases}
%     B_0=F_{d_2}(B_0)\cup F_{d_1}(B_0)&\\
%     B_1=F_{d_1}(B_1)&\\
%     B_2=F_{d_1}(B_1)\cup F_{d_0}(B_2)&\\
%     B_3=F_{d_5}(B_3)&\\
%     B_4=F_{d_5}(B_3)\cup F_{d_0}(B_4)&\\
%     B_5=F_{d_2}(B_5)&\\
%     B_6=F_{d_2}(B_5)\cup F_{d_4}(B_5)\cup F_{d_5}(B_6)&
% \end{cases}\]
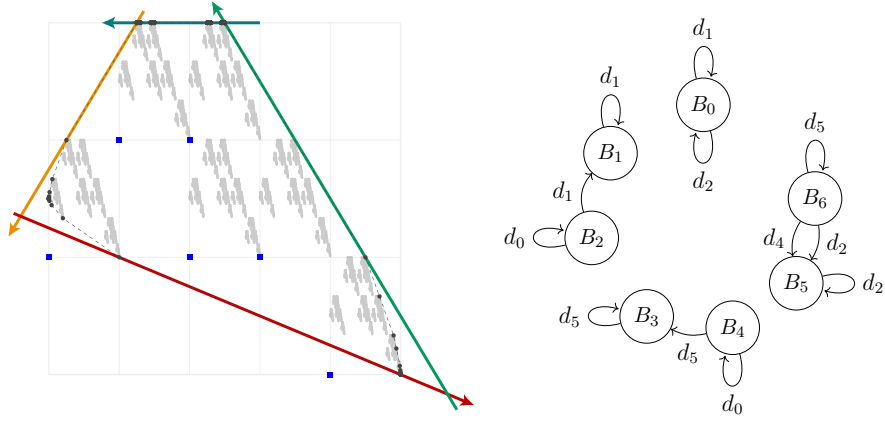
\begin{figure}[!htb]
    \centering
    \includegraphics[width=0.425\linewidth]{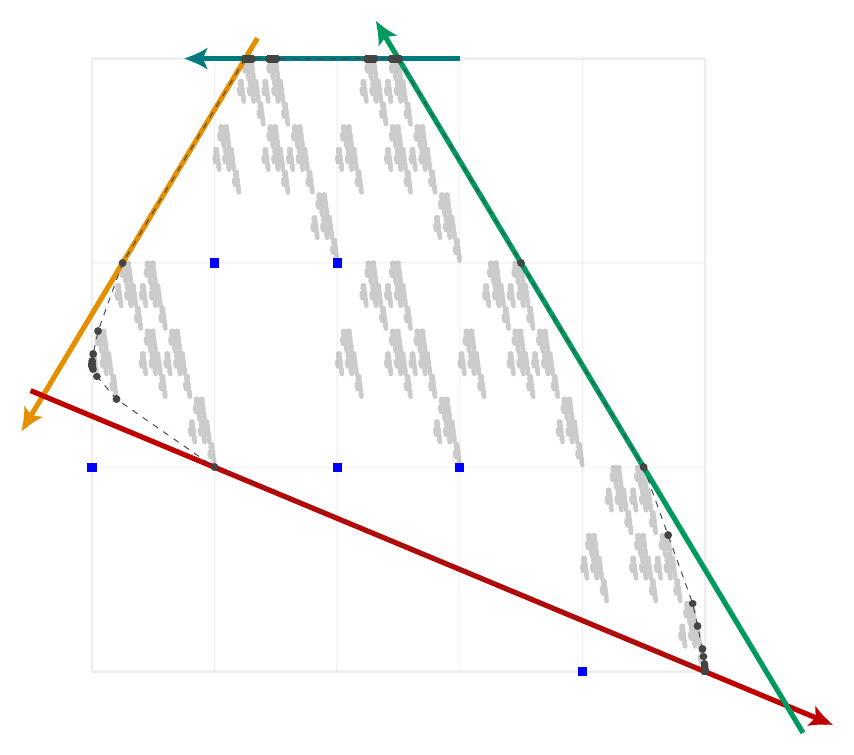}
    % \hspace{1cm}
    \begin{tikzpicture}
    
    \node [draw, circle] (90) at (90:2) (s0) {\small $B_0$};
    \node [draw, circle] (145) at (145:2) (s1) {\small $B_1$};
    \node [draw, circle] (190) at (190:2) (s2) {\small $B_2$};
    \node [draw, circle] (240) at (240:2) (s3) {\small $B_3$};
    \node [draw, circle] (285) at (285:2) (s4) {\small $B_4$};
    \node [draw, circle] (325) at (325:2) (s5) {\small $B_5$};
    \node [draw, circle] (10) at (10:2) (s6) {\small $B_6$};
        
    \path[->] (s0) edge [loop above] node {$d_1$} (s0);
    \path[->] (s0) edge [loop below] node {$d_2$} (s0);
    \path[->] (s1) edge [loop above] node {$d_1$} (s1);
    \path[->] (s2) edge [bend left, left] node {$d_1$} (s1);
    \path[->] (s2) edge [loop left] node {$d_0$} (s2);
    \path[->] (s3) edge [loop left] node {$d_5$} (s3);
    \path[->] (s4) edge [bend left=20, below] node {$d_5$} (s3);
    \path[->] (s4) edge [loop below ] node {$d_0$} (s4);
    \path[->] (s5) edge [loop right] node {$d_2$} (s5);
    \path[->] (s6) edge [bend left=20, right] node {$d_2$} (s5);
    \path[->] (s6) edge [bend right=20, left] node {$d_4$} (s5);
    \path[->] (s6) edge [loop above] node {$d_5$} (s6);

\end{tikzpicture}
    \caption{On the left, a Bedford-McMullen carpet, its bounding lines, and its extreme points, when $n=5$ and $m=3$. The carpet lives in the unit square and each rectangle of the grid has size $1/n\times1/m$. The corners of the grid marked with a rectangle indicate the elements of $\mathcal{D}$. On the right, the DGIFS associated with the extreme points.}
    \label{fig:Carpet}
\end{figure}

\subsection{Arbitrary translations}
\label{subsect:realarbitrary}
The proof of Theorem \ref{thm:mainrealcarpet} can be applied to the more general case of arbitrary translations, because it only needs the assumption that the two numbers on the main diagonal of the matrix $T$ are positive and that one is smaller than the other. The case when $F_k(\vec{x})=T\vec{x}+\vec{v}_k$, where $T$ is a scalar multiple of the identity matrix, is equivalent to the complex case (Section \ref{subsect:complexarbitrary}) with no rotation. In this section we will consider the remaining two real cases with arbitrary translations.

Let $\mathcal{D}$ be a non-empty discrete subset of $\mathbb{R}^2$ of size $n$, and denote by $\Psi_n$ the IFS consisting of \[F_d(\vec{x})=T\vec{x}+d=\begin{pmatrix}\mu&0\\0&\eta\end{pmatrix}\vec{x}+d,\quad d\in\mathcal{D}\]
where $0<\lvert\eta\rvert\leq\mu<1$.
\begin{theorem}
\label{thm:mainrealarbitrary}
    Fix integers $n\geq2$ and consider the IFS $\Psi_n$. The extreme points of the limit set $A$ are described by a system of finitely many equations and functions $F_d$ where $d\in\overline{\conv(\mathcal{D})}\setminus\conv(\mathcal{D})$. Moreover, if $A$ is convex and $\mu<0$ with $\lvert\mu\rvert=\mu$, then $\eta\geq2^{-0.5}$ and $A$ is a polygon with at most twice the number of sides of $\conv(\mathcal{D})$.
\end{theorem}
\begin{proof}
    From Proposition \ref{prop:extremeconvexcore} we know that edges of $\conv(A)$ are contained in images of bounding lines to $A$ and these are lines parallel to the core lines. 
    
    Now, suppose $\ell$ is a core line containing an edge of $\conv(\mathcal{D})$. Let $\mathcal{D}(\ell):=\{d_j\}_{j=0}^{r_\ell}$ denote the ordered digits through which $\ell$ passes. Finally, let $\widehat\ell$ be the bounding line to $A$ parallel to $\ell$. For each choice of the pair $\eta,\mu$ there are two cases to consider: either $\widehat\ell$ has non-zero and finite slope, or it is a vertical or horizontal line. 

    If $\widehat{\ell}$ is a vertical or horizontal line, then applying $F_d$ to it does not change its slope but only its relative position. Therefore, the extreme points on $\widehat{\ell}$ are described by 
    \[B_0=\bigcup_{d\in\mathcal{D}(\ell)}F_d(B_0).\]
    Let $\widehat{\ell}$ be a bounding line with non-zero and finite slope.
    If $\eta=\mu>0$, this reduces to the complex setting with no rotation, i.e. Theorem \ref{thm:maincplxarbitrary}. If $\eta>\mu>0$, then this is equivalent to Theorem \ref{thm:mainrealcarpet}. Suppose then that $\eta>0>\mu$.
    
    Since $\mu<0$, applying $T$ to any bounding line reflects it across the $y$-axis and rescales both coordinates. Assuming $\lvert\mu\rvert=\eta$, the contraction in both coordinates is the same, so each bounding line is invariant under multiplication by $T^2$. Denote by $\theta$ the normal direction of $\widehat\ell$, then $F_k(\widehat\ell)$, for any $k$, is a line with normal direction $\theta+1/2$ and $F_{jk}(\widehat\ell)$, for any $j$, has normal direction $\theta$. By definition, the attractor is in the left half plane of $\widehat{\ell}$. Given that $A$ is compact, there is always at least one $d$ for which the entire attractor is in the right half plane of $F_d(\widehat{\ell})$. Similarly, we can find at least one $d'$ for which $F_{d'}(F_d(\widehat{\ell}))$ has the attractor in its left half plane. In terms of equations we then have:\[B_{0,m}=\bigcup_{d\in\mathcal{D}(\ell_m)}F_d(B_{1,m}),\quad B_{1,m}=\bigcup_{d'\in\mathcal{D}(T\ell_m)}F_{d'}(B_{0,m})\] where by $\mathcal{D}(T\ell_m)$ we mean the set of digits $d'$ for which $\conv(\mathcal{D})$ is contained in the right half plane of $F_{d'}(\ell_m)$. We conclude that $\conv(A)$ is a polygon with at most twice the number of sides of $\conv(\mathcal{D})$. Moreover, if $A$ is convex then $2\eta^2\geq1$, since $\lvert\mathcal{D}(\ell)\rvert\geq2$ for each $\ell$. 
    
    In the case $\lvert\mu\rvert<\eta$, we need to account for the fact that each application of $F_k$, for any $k$, increases the absolute value of the slope of $F_k(\gamma)$ for some segment $\gamma$. Starting with a segment $\gamma$ on some bounding line $\widehat\ell$, we have to check whether $F_k(\gamma)$ can be an edge of $\conv(A)$ by comparing its slope to that of $\widehat\ell'$, for each $k\in\mathcal{D}(\ell')$. Nonetheless, the process stops because $\mathcal{D}$ is finite and, since the slope of $F_{k_1\cdots k_m}(\gamma)$ tends to $\pm\infty$ as $m\mapsto\infty$, we eventually end up in a cycle between maps $F_d$ for which $d$ are the right and left most translations in $\overline{\conv(\mathcal{D})}\setminus\conv(\mathcal{D})$.
    
\end{proof}
The system of equations that describes $\ext(A)$ for the last case considered in the proof are sensitive to the choice of translations $d\in\mathcal{D}$. Thus, we do not have a constant number of equations but we can only say it is finite. In Section \ref{subsubsect:realarbposneg} we provide an example for this case and we go through the process step-by-step.

\subsection{Examples}
\label{subsect:realarbitraryexamples}
In the next two examples we let $v=\exp(2\pi\I(\sqrt{5}-1)/2)$ and we consider \[\mathcal{D}=\{d_0=v^5,d_1=v^2,d_2=-0.2+v^7,d_3=v^4,d_4=0.5+v^3\}.\]
Note that we have already ordered the set $\mathcal{D}$ based on the argument of each element. The edges of $\conv(\mathcal{D})$ are the difference of two consecutive pairs: 
\[
\begin{split}
    \mathcal{D}(\ell_0)=\{d_0,d_1\}&\quad\mathcal{D}(\ell_1)=\{d_1,d_2\}\quad
    \mathcal{D}(\ell_2)=\{d_2,d_3\}\\ &\mathcal{D}(\ell_3)=\{d_3,d_4\}\quad\mathcal{D}(\ell_4)=\{d_4,d_0\}.    
\end{split}
\]

\subsubsection{Equal absolute value}
\label{subsubsect:realarbequal} Let $\mu<0$ with $\lvert\mu\rvert=\eta=2^{-0.5}$. From the proof of Theorem \ref{thm:mainrealarbitrary} we know that for every bounding line there is a symmetric one across the $y$-axis which contains an edge of $\conv(A)$. Since $\conv(\mathcal{D})$ has $5$ sides, none of which are vertical or horizontal, then $\conv(A)$ will have exactly $10$ sides. Each symmetric pair of sides has their own set of equations, see Figure \ref{fig:RealArbEqual}. We chose the value of $\eta$ so that the attractor is actually convex. 

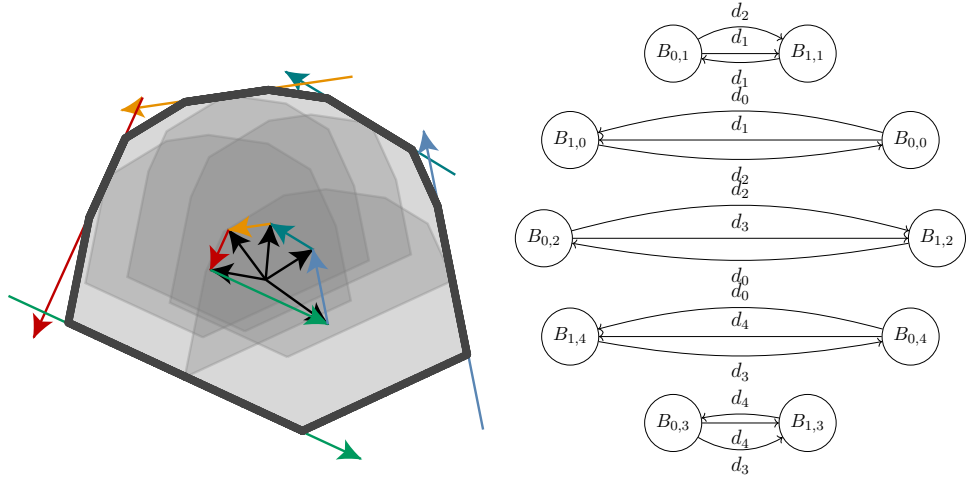
\begin{figure}
    \centering
    \includegraphics[width=0.45\linewidth]{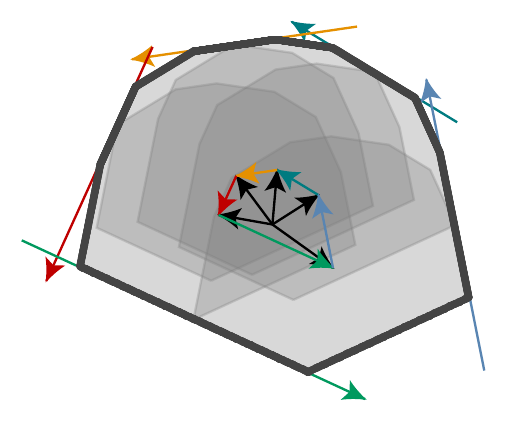}
    % \hspace{1cm}
    \begin{tikzpicture}
    
    \node [draw, circle] (110) at (110:4) (s0) {\small $B_{0,1}$};
    \node [draw, circle] (70) at (70:4) (s1) {\small $B_{1,1}$};
    \node [draw, circle] (30) at (30:4) (s2) {\small $B_{0,0}$};
    \node [draw, circle] (150) at (150:4) (s3) {\small $B_{1,0}$};
    \node [draw, circle] (180) at (180:4) (s4) {\small $B_{0,2}$};
    \node [draw, circle] (0) at (0:4) (s5) {\small $B_{1,2}$};
    \node [draw, circle] (330) at (330:4) (s6) {\small $B_{0,4}$};
    \node [draw, circle] (210) at (210:4) (s7) {\small $B_{1,4}$};
    \node [draw, circle] (250) at (250:4) (s8) {\small $B_{0,3}$};
    \node [draw, circle] (290) at (290:4) (s9) {\small $B_{1,3}$};
    
    \path[->] (s0) edge [bend left=30, above] node {$d_2$} (s1);
    \path[->] (s0) edge [bend left=0, above] node {$d_1$} (s1);
    \path[->] (s1) edge [bend left=10, below] node {$d_1$} (s0);
    
    \path[->] (s2) edge [bend right=15, above] node {$d_0$} (s3);
    \path[->] (s2) edge [bend right=0, above] node {$d_1$} (s3);
    \path[->] (s3) edge [bend right=10, below] node {$d_2$} (s2);
    
    \path[->] (s4) edge [bend left=15, above] node {$d_2$} (s5);
    \path[->] (s4) edge [bend left=0, above] node {$d_3$} (s5);
    \path[->] (s5) edge [bend left=10, below] node {$d_0$} (s4);
    
    \path[->] (s6) edge [bend right=15, above] node {$d_0$} (s7);
    \path[->] (s6) edge [bend right=0, above] node {$d_4$} (s7);
    \path[->] (s7) edge [bend right=10, below] node {$d_3$} (s6);
    
    \path[->] (s8) edge [bend right=30, below] node {$d_3$} (s9);
    \path[->] (s8) edge [bend right=0, below] node {$d_4$} (s9);
    \path[->] (s9) edge [bend right=10, above] node {$d_4$} (s8);
\end{tikzpicture}
    \caption{On the left, the attractor $A$ and its boundary, which coincides with $\ext(A)$, of the IFS from Section \ref{subsubsect:realarbequal}. On the right, the DGIFS. Note that the second index of the $B$ sets indicate the corresponding bounding line.}
    \label{fig:RealArbEqual}
\end{figure}

\subsubsection{Positive $\eta$ and negative $\mu$}
\label{subsubsect:realarbposneg} Let $\mu=-1/4$ and $\eta=2^{-0.5}$. The following discussion is easier to digest if it is accompanied by a picture. Some of the computation, like comparing two slopes, were done with the help of the computer.

Consider $\widehat\ell_1$ which is the bounding line whose slope is closest to $0$. By definition, it contains extreme points of $A$ that belong to $F_{d_1}(A)$ and $F_{d_2}(A)$. Since its slope is positive and $A$ is in its left half plane, the possible images of $\widehat\ell_1$ that contains $A$ in its right half plane are $F_{d_1}(\widehat\ell_1)$, $F_{d_0}(\widehat\ell_1)$, and $F_{d_4}(\widehat\ell_1)$. In the first case, $\widehat\ell_1\cap F_{d_1}(\widehat\ell_1)=\{p\}$ where $p$ is necessarily the fixed point of $F_{d_1}$. We conclude that $p$ is an extreme point of $A$ and consequently $p+(d_2-d_1)=F_{d_2}(p)\in\widehat\ell_1$ and $F_{d_1}(F_{d_2}(p))$ must also be extreme points. 

We remark also that $[p,F_{d_2}(p)]$ is the segment on $\widehat\ell_1$ intersecting $\conv(A)$ on an edge. The other two cases are thus reduced to checking whether the segments $F_{d_0}([p,F_{d_2}(p)])$ and $F_{d_4}([p,F_{d_2}(p)])$ are on $\partial\conv(A)$. The slope of $F_{d_0}([p,F_{d_2}(p)])$, which is negative and the same as $F_{d_1}([p,F_{d_2}(p)])$, is larger than $\widehat\ell_0$, a bounding line connecting some extreme points of $A$ in $F_{d_0}(A)$ to some on $F_{d_1}(A)$. Thus, we either have that $F_{d_0}(p)\in\widehat\ell_0$ or $F_{d_0d_2}(p)\in\widehat\ell_0$: the first case would imply that $F_{d_0d_2}(p)$ is to the right of $\widehat\ell_0$, which is impossible since $\widehat\ell_0$ is a bounding line; in the other case $F_{d_0}(p)$ would be in the interior of $\conv(A)$. Consequently $F_{d_0}([p,F_{d_2}(p)])$ is not an edge of $\conv(A)$ and neither is $F_{d_4}([p,F_{d_2}(p)])$ for similar reasons.

The segment $F_{d_2}(F_{d_1}([p,F_{d_2}(p)]))=F_{d_2}([p,F_{d_1d_2}(p)])$ intersects $\widehat\ell_1$ at $F_{d_2}(p)$, by construction. Its slope is smaller than that of $\widehat\ell_2$. Given that we established that $F_{d_2}(p)$ is an extreme point, if $F_{d_2}([p,F_{d_1d_2}(p)])$ were to intersect $\widehat\ell_2$ at $F_{d_2}(p)$, then the point $F_{d_2}(p)+(d_3-d_2)=F_{d_3}(p)$, along the line $\widehat\ell_2$, would be to the left of segment $F_{d_2}([p,F_{d_1d_2}(p)])$, making the segment $[F_{d_2}(p),F_{d_3}(p)]\subset\widehat\ell_2$ not an edge of $\conv(A)$. Therefore, the point $F_{d_2d_1d_2}(p)$ is a vertex of $\conv(A)$ and no other points of $F_{d_2}(A)$ lies on $F_{d_2d_1}(\widehat\ell_1)$. 

So far we have that the segments $I_0=[p,F_{d_2}(p)]\subset\widehat\ell_1$, $I_1=F_{d_1}([p,F_{d_2}(p)])$, and $I_2=F_{d_2d_1}([p,F_{d_2}(p)])$ are edges of $\conv(A)$. Now, the segment $F_{d_1}(I_2)$, whose left most point is the extreme point $F_{d_1d_2}(p)$, has slope even smaller than that of $I_1$ and $\widehat\ell_0$, so it cannot intersect $\conv(A)$ on an edge: $F_{d_1d_2d_1d_2}(p)$ is not an extreme point. We have two more options: $F_{d_0}(I_2)$ and $F_{d_4}(I_2)$. The latter needs to be excluded because its slope is larger than $\widehat\ell_4$. Therefore, me must have that $I_3=F_{d_0}(I_2)$ is an edge of $\conv(A)$. Note that $F_{d_0d_2}(p)$ lies on $F_{d_0}(\widehat\ell_1)$ which we discarded before for having a larger slope than $\widehat\ell_0$. However, $F_{d_1d_2}(p)=F_{d_0d_2}(p)+(d_1-d_0)$, so we conclude that $J_0=[F_{d_0d_2}(p),F_{d_1d_2}(p)]$ is the segment of $\widehat\ell_0$ on the boundary of $\conv(A)$.

The slope of $F_{d_2}(J_0)$ is smaller than that of $\widehat\ell_2$ and its right most endpoint we have found to be an extreme point. Therefore, $J_1=F_{d_2}(J_0)$ in another edge of $\conv(A)$ and $F_{d_2d_0d_2}(p)$ is an extreme point. This last extreme point is also on the line $\widehat\ell_2$: the left endpoint of $F_{d_3}(I_3)$, i.e. $F_{d_3d_0d_2}(p)$, is actually equal to $F_{d_2d_0d_2}(p)+(d_3-d_2)$; furthermore, $F_{d_2}(I_3)$, which intersects $J_1$ at $F_{d_2d_0d_2}(p)$, has a steeper slope than the one of $\widehat\ell_2$ which implies that it is not an edge of $\conv(A)$. We conclude then that $L_0=[F_{d_2d_0d_2}(p),F_{d_3d_0d_2}(p)]$ is an edge of $\conv(A)$ and a subset of $\widehat\ell_3$.

Given that slopes of the edges are getting steeper the only maps worth considering to find edges of $\conv(A)$ are $F_{d_4}$ and $F_{d_3}$. This means, for starters, we have $I_k=F_{d_3}(I_{k-1})$ and $I_{k+1}=F_{d_4}(I_{k})$ for $k\geq4$. Also the images of $L_0$ alternate between $F_{d_4}$ and $F_{d_3}$, since $\widehat\ell_4$ is less steep than $F_{d_0}(L_0)$. Let us consider the bounding line $\widehat\ell_3$, which connects some extreme points in $F_{d_3}(A)$ to some in $F_{d_4}(A)$. The line $F_{d_4}(\widehat\ell_3)$ must intersect $\widehat\ell_3$, so the fixed point, call it $q$, of $F_{d_4}$ is an extreme point. It follows that $q-(d_4-d_3)=F_{d_3}(q)\in\widehat\ell_3$ and $F_{d_4}(F_{d_3}(q))$ are other extreme points. The segment $[q,F_{d_3}(q)]$ is a subset of $\widehat\ell_3$ and only its images under $F_{d_3}$ and $F_{d_4}$ are edges of $\conv(A)$.

The only edges of the convex hull of $A$ left to find are the images of $J_1$ and those of a segment of $\widehat\ell_4$. The slope of $F_{d_0}(J_1)=F_{d_0}([F_{d_2d_0d_2}(p),F_{d_2d_1d_2}(p)])$ is more flat than the one of $\widehat\ell_4$ and $I_3\cap F_{d_0}(J_1)=\{F_{d_0d_2d_1d_2}(p)\}$. Thus $J_2=F_{d_0}(J_1)\subset\partial\conv(A)$ and $F_{d_0d_2d_0d_2}(p)\in\ext(A)$. Note that this last point is also on $\widehat\ell_4$ since $F_{d_4d_2d_0d_2}(p)=F_{d_0d_2d_0d_2}(p)+(d_4-d_0)$ and this point is on $L_1=F_{d_4}(L_0)$, whose slope is steeper than the one of $\widehat\ell_4$. Consequently, $[F_{d_4d_2d_0d_2}(p),F_{d_0d_2d_0d_2}(p)]\subset\widehat\ell_4$ is an edge of $\conv(A)$.

The resulting system of equations is shown on the right in Figure \ref{fig:RealArbPosNeg}. On the attractor we have highlighted with arrows the edges of $\conv(A)$ which are part of a bounding line.
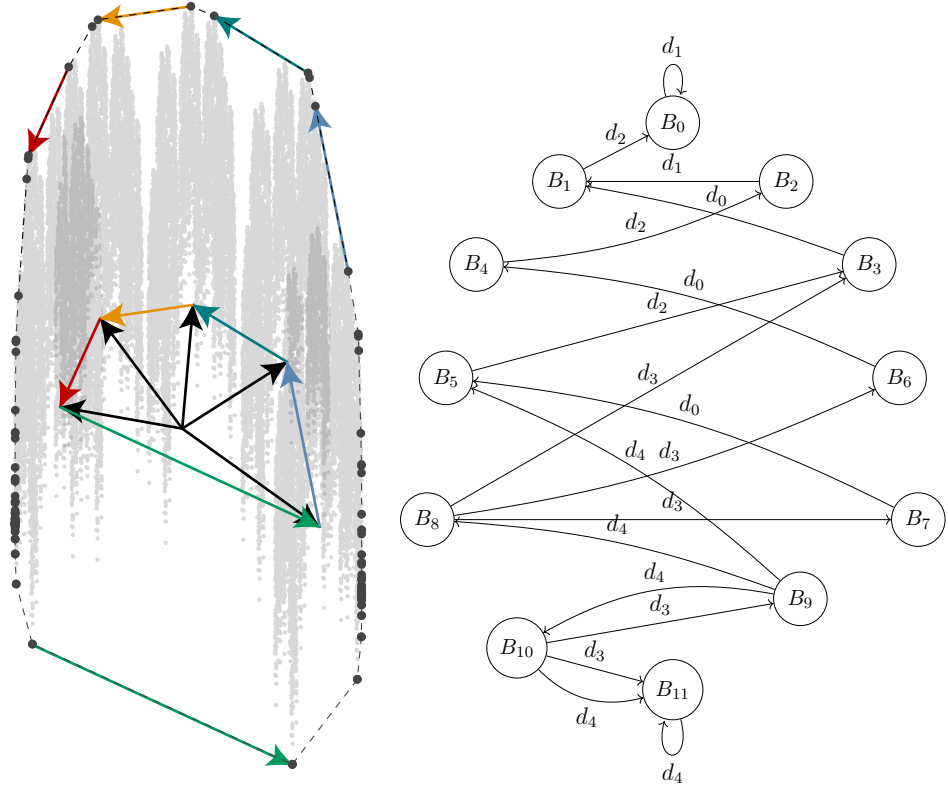
\begin{figure}
    \centering
    \includegraphics[width=0.35\linewidth]{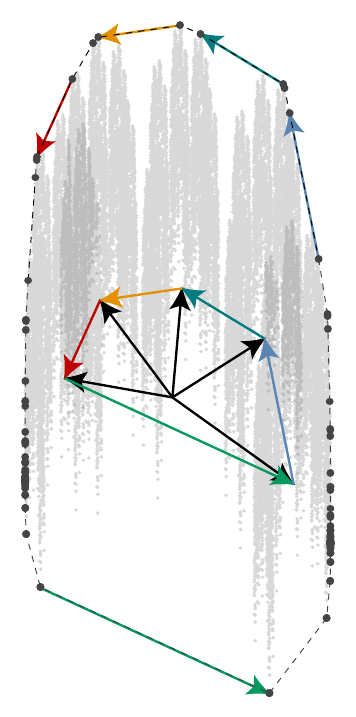}
    % \hspace{1cm}
    \begin{tikzpicture}
    
    \node [draw, circle] (90) at (90:4.5) (s0) {\small $B_{0}$};
    \node [draw, circle] (120) at (120:4) (s1) {\small $B_{1}$};
    \node [draw, circle] (60) at (60:4) (s2) {\small $B_{2}$};
    \node [draw, circle] (30) at (30:4) (s3) {\small $B_{3}$};
    \node [draw, circle] (150) at (150:4) (s4) {\small $B_{4}$};
    \node [draw, circle] (180) at (180:4) (s5) {\small $B_{5}$};
    \node [draw, circle] (0) at (0:4) (s6) {\small $B_{6}$};
    \node [draw, circle] (-30) at (-30:5) (s7) {\small $B_{7}$};
    \node [draw, circle] (210) at (210:5) (s8) {\small $B_{8}$};
    \node [draw, circle] (-60) at (-60:4.5) (s9) {\small $B_{9}$};
    \node [draw, circle] (240) at (240:5.5) (s10) {\small $B_{10}$};
    \node [draw, circle] (270) at (270:5.5) (s11) {\small $B_{11}$};
    
    \path[->] (s0) edge [loop above] node {$d_1$} (s0);
    \path[->] (s1) edge [left,above] node {$d_2$} (s0);
    \path[->] (s2) edge [above] node {$d_1$} (s1);
    \path[->] (s3) edge [bend right=5, above] node {$d_0$} (s1);
    \path[->] (s4) edge [bend right=10, above] node {$d_2$} (s2);
    \path[->] (s5) edge [above] node {$d_2$} (s3);
    \path[->] (s6) edge [bend right=10, above] node {$d_0$} (s4);
    \path[->] (s7) edge [bend right=10, above] node {$d_0$} (s5);

    \path[->] (s8) edge [above] node {$d_3$} (s3);
    \path[->] (s8) edge [bend right=8, above] node {$d_3$} (s6);
    \path[->] (s8) edge [above] node {$d_3$} (s7);
    \path[->] (s9) edge [bend right=8, above] node {$d_4$} (s8);
    \path[->] (s9) edge [bend right=10, above] node {$d_4$} (s5);
    \path[->] (s9) edge [bend right=20, above] node {$d_4$} (s10);
    \path[->] (s10) edge [above] node {$d_3$} (s9);
    \path[->] (s10) edge [above] node {$d_3$} (s11);
    \path[->] (s10) edge [bend right, below] node {$d_4$} (s11);
    \path[->] (s11) edge [loop below] node {$d_4$} (s11);
    
\end{tikzpicture}
    \caption{On the left, the elements of $\ext(A)$ of the IFS from Section \ref{subsubsect:realarbposneg}. On the right, the DGIFS describing them.}
    \label{fig:RealArbPosNeg}
\end{figure}

\section{The Jordan block case}
\label{sect:jordan}
The last family of homogeneous IFS that we will consider has a contracting matrix of the form $T=\begin{pmatrix}\eta&1\\0&\eta\end{pmatrix}$ for non-zero real $-1<\eta<1$. As in the previous section, the choice of translations influences which pairs of integers $(k,j)$ index the relevant bounding lines $\widehat{\ell}_{k,j}$ and, consequently, the structure of the DGIFS describing the set of extreme points. We proceed as before by restricting first to collinear translation vectors to build intuition and, then, deal with the arbitrary translations.

As far as we know, the convex hull and extreme points for this particular family of IFS has only been studied in \cite{HS16} when $n=2$ and $\eta>0$. Just like in Section \ref{sect:real}, we will see that the itineraries of the corner points for these IFS can be either periodic or pre-periodic.

\subsection{Collinear IFS}
\label{subsect:jordancollinear}
Fix $n\geq2$ and let $\vec{w}$ be any vector in $\mathbb{R}^2$ except thos along the axes. Consider the IFS $\Psi_n=\{F_k\}_{k=0}^{n-1}$ generated by the maps \[F_k(\vec{x})=T\vec{x}+(n-2k-1)\vec{w}\quad 0\leq k\leq n-1\]
and set $A$ to be the attractor of such an IFS.

\begin{theorem}
\label{thm:mainjordancollinear}
    The set of extreme points of the attractor associated to the IFS $\Psi_n=\{F_k\}_{k=0}^{n-1}$ is the unique vector solution to the following DGIFS
    \begin{itemize}
        \item if $\eta<0$,
        \[\begin{cases}
            B_0=F_0(B_2)&\\
            B_1=\bigcup_{k=0}^{n-2}F_k(B_2) \cup F_{n-1}(B_3)&\\
            B_2=F_{n-1}(B_0)&\\
            B_3=F_0(B_1) \cup \bigcup_{k=1}^{n-1}F_k(B_0)&
        \end{cases};\]
        \item if $\eta>0$,
        \[\begin{cases}
            B_0=\bigcup_{k=0}^{n-2}F_k(B_2) \cup F_{n-1}(B_0)&\\
            B_1=F_{n-1}(B_1)&\\
            B_2=F_0(B_2)&\\
            B_3=F_0(B_3) \cup \bigcup_{k=1}^{n-1}F_k(B_1)&
        \end{cases}.\]
    \end{itemize}
\end{theorem}

\begin{proof}    
Just like in the real case, the attractor $A$ satisfies the following:
\begin{enumerate}[label=(\roman*)]
    \item the attractor is symmetric about $0$;
    \item the attractor has translational symmetry along the line through the origin spanned by $\vec{w}$: $F_k(A)$, for $0<k\leq n-1$, is a copy of $F_0(A)$ translated by $(2-2k)\vec{w}$;
\end{enumerate}
As for the IFS in Section \ref{sect:real} there are only two core lines $\ell_{1,0}$ and $\ell_{0,1}$ both parallel to $\vec{w}$, but with different directions. Let $\alpha$ be the angle that $\vec{w}$ makes with the horizontal axis; let $\theta_0=\alpha+1/4\mod1$ and $\theta_1=\theta_0+1/2\mod 1$ indicate the normal directions of each core line.
The symmetries of the attractor mentioned above imply that any edge of $\conv(A)$ at angle $\pm t\in(\theta_0,\theta_1)$ must go through extreme points belonging to either $F_0(A)$ or $F_{n-1}(A)$. Without loss of generality we can assume that $0<\alpha<1/4$ so that $1/4<\theta_0<1/2$.

Consider first the case of $\eta>0$. The image of the bounding line $\widehat\ell_{1,0}$ under any $F_k$ is less steep than $\widehat\ell_0$ due to $T$ being a positive upper triangular matrix and $\eta<1$. In fact, denoting by $t_k$ the outward normal direction of $T^k\widehat\ell_{1,0}$ we obtain \[1/4<\cdots<t_k<\cdots<t_2<t_1<\theta<1/2.\] Since $A$ is in the closed left half plane of $\widehat\ell_{1,0}$ and $T$ preserves orientation, then $A$ would be in the closed left half plane of $F_j(\widehat\ell_{1,0})$, only if $j=0$ because of the translational symmetry of the attractor. Repeating the argument, $A$ is in the closed left half plane of each $F_\omega(\widehat\ell_{1,0})$ for each finite sequence $\omega$ of the digit $0$. Taking the limit, the line $F_\omega(\widehat\ell_{1,0})$, where $\omega$ is the infinite sequence of repeating $0$, is horizontal and tangent to $A$. Therefore, the fixed point $p$ of $F_0$ must be an extreme point of $A$. Furthermore, this horizontal line must intersect the other bounding line $\widehat\ell_{0,1}$ at exactly $p$. We conclude that $F_k(p)$ are in $\widehat\ell_{0,1}$ for each $0\leq k<n$ and, by symmetry, that $F_k(q)\in\widehat\ell_{1,0}$ where $q=F_{n-1}(q)$. Let $B_k$ indicate a subset of $\ext(A)$ we have proved that 
\[
\begin{cases}
    B_0=\bigcup_{k=0}^{n-2}F_k(B_2) \cup F_{n-1}(B_0)&\\
    B_1=F_{n-1}(B_1)&\\
    B_2=F_0(B_2)&\\
    B_3=F_0(B_3) \cup \bigcup_{k=1}^{n-1}F_k(B_1)&
\end{cases}.
\]

Let now $\eta<0$ and denote by $t_k$ the outward normal direction of $T^k\widehat\ell_{1,0}$: simple computations show that \[0<t_1<t_3<\cdots<t_{2k+1}<\cdots<1/4<\theta_0<1/2<t_2<t_4<\cdots<t_{2k}<\cdots<3/4,\] Given that the linear part of $F_j$ is $T$, to make sure to have the attractor in the left half plane of $F_j(\widehat\ell_0)$, we want to alternate between $F_0$ and $F_{n-1}$. Observe that the bounding line $\widehat\ell_{0,1}$ must intersect $F_\omega(\widehat\ell_{1,0})$, for the repeating sequence $\omega$ of the digits $0$ and $n-1$, at the fixed point $p$ of $F_0\circ F_{n-1}$. Thus, $p$ and $F_k(p)$ for any $0\leq k<n$ are extreme points of $A$ on the bounding line $\widehat\ell_{0,1}$. By symmetry we have that $q$ and $F_k(q)$ for any $0\leq k<n$ are in $\ext(A)$ and on the line $\widehat\ell_{1,0}$, where $q=F_{n-1}(F_0(q))$.

The DGIFS can then be recovered from the above argument: let $B_0$ be the fixed point of $F_0\circ F_{n-1}$ and let $B_2$ be the fixed point of $F_{n-1}\circ F_0$, that is $B_0=F_0(B_2)$ and $B_2=F_{n-1}(B_0)$; the remaining extreme points are the images of these and, because of the symmetry across the $y$-axis, we need to alternate between $F_0$ and $F_{n-1}$, obtaining 
\[\begin{cases}
    B_0=F_0(B_2)&\\
    B_1=\bigcup_{k=0}^{n-2}F_k(B_2) \cup F_{n-1}(B_3)&\\
    B_2=F_{n-1}(B_0)&\\
    B_3=F_0(B_1) \cup \bigcup_{k=1}^{n-1}F_k(B_0)&
\end{cases}.\]
\end{proof}

\begin{corollary}
    In the case when $\eta<0$, the DGIFS has only one connected component but it is not strongly connected. In the other case, the DGIFS has two connected components but neither of them is strongly connected.
\end{corollary}
\begin{remark}
    Once again the system of equations of each DGIFS in Theorem \ref{thm:mainjordancollinear} can not be decoupled. It follows, that even in this setting we have $B_k$ for $0\leq k\leq 3$ is not the attractor of an IFS.
\end{remark}

\subsection{Examples}
\label{subsect:jordancollinearexamples}
For the following examples we fix $w=1+\I$. We tried to find values of $\eta$ for which each component $F_k(A)$ of the attractor $A$ was visible enough.

\subsubsection{Positive $\eta$}
Let $n=4$ and set $\eta=17/40$ then Theorem \ref{thm:mainjordancollinear} tells us that the set of extreme points of $A$ corresponds to the language with symbols from $\{0,1,2,3\}$ defined by the sofic system shown in Figure \ref{fig:JordanPos}.
% \[\begin{cases}
%     B_0=F_0(B_2)\cup F_1(B_2)\cup F_2(B_2)\cup F_3(B_0)&\\
%     B_1=F_3(B_1)&\\
%     B_2=F_0(B_2)&\\
%     B_3=F_0(B_3)\cup F_1(B_1)\cup F_2(B_1)\cup F_3(B_1)&
% \end{cases}\]
\begin{figure}
    \centering
    \includegraphics[width=0.4\linewidth]{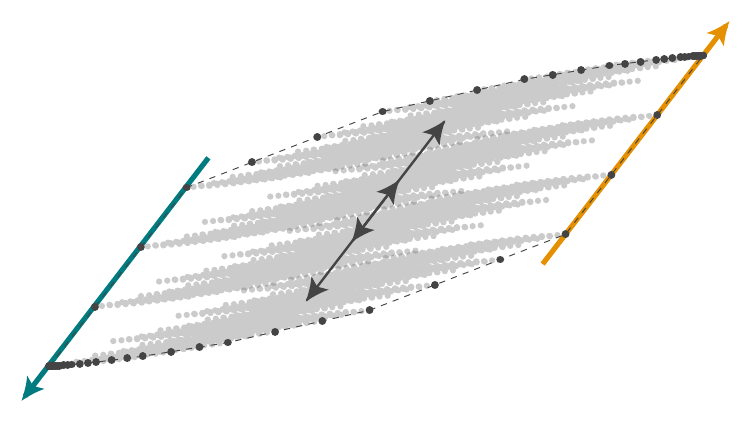}
    % \hspace{1cm}
    \begin{tikzpicture}
    
    \node [draw, circle] (0) at (0:2) (s0) {\small $B_0$};
    \node [draw, circle] (270) at (270:2) (s1) {\small $B_1$};
    \node [draw, circle] (90) at (90:2) (s2) {\small $B_2$};
    \node [draw, circle] (180) at (180:2) (s3) {\small $B_3$};
        
    \path[->] (s0) edge [bend right=60, above] node {$0$} (s2);
    \path[->] (s0) edge [bend right=30, above] node {$1$} (s2);
    \path[->] (s0) edge [ above] node {$2$} (s2);
    \path[->] (s0) edge [loop below] node {$3$} (s0);
    
    \path[->] (s1) edge [loop right] node {$3$} (s1);
    \path[->] (s2) edge [loop left] node {$0$} (s2);
    
    \path[->] (s3) edge [loop above] node {$0$} (s3);
    \path[->] (s3) edge [bend right=60, below] node {$1$} (s1);
    \path[->] (s3) edge [bend right=30, below] node {$2$} (s1);
    \path[->] (s3) edge [ below] node {$3$} (s1);
        
\end{tikzpicture}
    \caption{On the left the extreme points, the convex hull, and bounding set of the attractor of the IFS $\Psi_4$ when $\eta>0$. On the right the sofic system used to identify points in $\ext(A)$.}
    \label{fig:JordanPos}
\end{figure}
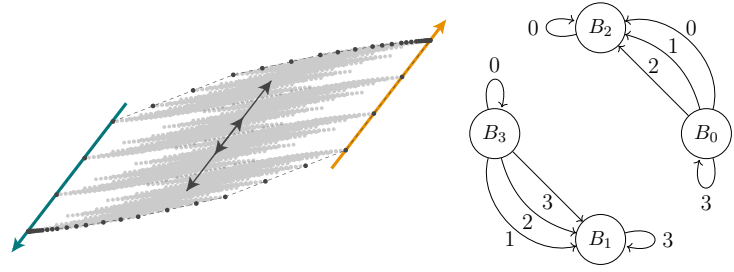

\subsubsection{Negative $\eta$}
Let $n=3$ and set $\eta=-3/4$, then by Theorem \ref{thm:mainjordancollinear} a point in $\ext(A)$ has an itinerary defined by an infinite walk on the graph in Figure \ref{fig:JordanNeg}.
% \[\begin{cases}
%     B_0=F_0(B_2)&\\
%     B_1=F_0(B_2)\cup F_1(B_2)\cup F_2(B_3)&\\
%     B_2=F_2(B_0)&\\
%     B_3=F_0(B_1)\cup F_1(B_0)\cup F_2(B_0)&
%\end{cases}.\]
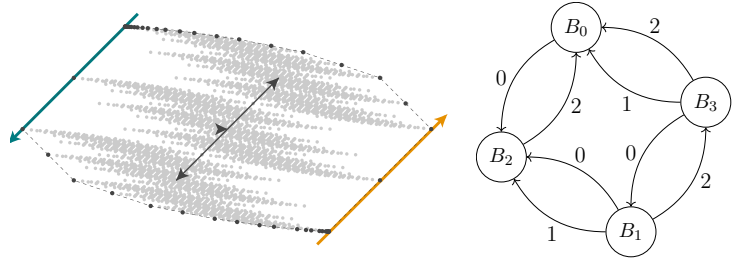
\begin{figure}
    \centering
    \includegraphics[width=0.4\linewidth]{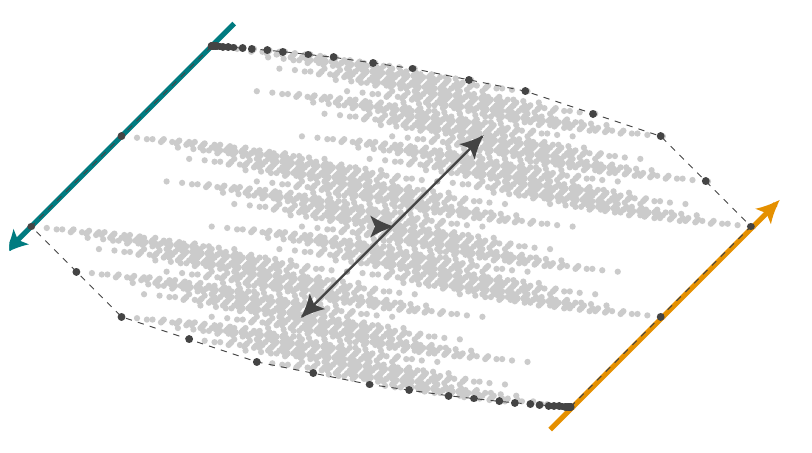}
    % \hspace{1cm}
    \begin{tikzpicture}
    
    \node [draw, circle] (90) at (90+15:2) (s0) {\small $B_0$};
    \node [draw, circle] (270) at (270+15:2) (s1) {\small $B_1$};
    \node [draw, circle] (180) at (180+15:2) (s2) {\small $B_2$};
    \node [draw, circle] (0) at (0+15:2) (s3) {\small $B_3$};
        
    \path[->] (s0) edge [bend right, left] node {$0$} (s2);
    \path[->] (s2) edge [bend right, right] node {$2$} (s0);
    \path[->] (s1) edge [bend right,above] node {$0$} (s2);
    \path[->] (s1) edge [bend left, below] node {$1$} (s2);
    \path[->] (s1) edge [bend right=30, right] node {$2$} (s3);
    \path[->] (s3) edge [bend right, left] node {$0$} (s1);
    \path[->] (s3) edge [bend left, below] node {$1$} (s0);
    \path[->] (s3) edge [bend right=30, above] node {$2$} (s0);
    
\end{tikzpicture}
    \caption{On the left the extreme points, the convex hull, and bounding set of the attractor of the IFS $\Psi_3$ when $\eta<0$. On the right the DGIFS for $\ext(A)$.}
    \label{fig:JordanNeg}
\end{figure}

\subsection{Arbitrary translations}
\label{subsect:jordanarbitrary}
Let $\mathcal{D}$ be a discrete subset of $\mathbb{R}^2$ of size $n$. Denote by $\Psi_n=\{F_d\}_{d\in\mathcal{D}}$ the IFS of affine transformations \[F_d(\vec{x})=T\vec{x}+\vec{v}_d=\begin{pmatrix}\eta&1\\0&\eta\end{pmatrix}\vec{x}+\vec{v}_d\] where $-1<\eta<1$ is non-zero. From Proposition \ref{prop:extremeconvexcore} the bounding lines intersect $\conv(A)$ along a specific edge. It is impossible to obtain the exact system of equations that describe the set of extreme points of $A$ without knowing the shape of the of the bounding set, which is a scaled version of $\conv(\mathcal{D})$. However, Proposition \ref{prop:extremeconvexcore} guarantees that such system exists and what we will show is that this system is necessarily finite. As we saw in the real case, to deal with arbitrary translations we have to pay particular attention to the different slopes of each bounding line.
\begin{theorem}
    \label{thm:mainjordanarbitrary}
    Fix integers $n\geq2$ and consider the IFS $\Psi_n$. The extreme points of the limit set $A$ are described by a DGIFS with finitely many equations and functions $F_d$ where $d\in\overline{\conv(\mathcal{D})}\setminus\conv(\mathcal{D})$.
\end{theorem}
\begin{proof}
    Let $\eta>0$. Consider the bounding line $\ell$ which has the smallest outward normal direction $\theta$ in the arc $(1/4,1/2)$. Given that $\eta<1$ and $T$ is a positive upper triangular matrix, we can deduce that  \[1/4<\cdots<t_k<\cdots<t_2<t_1<\theta<1/2\] where $t_k$ indicates the normal direction of $T^k\ell$. We can assume that $\conv(\mathcal{D})$ is a polygon and not a strip, otherwise refer to Theorem \ref{thm:mainjordancollinear}. Therefore, we can find bounding lines whose normal direction are positive and $\leq1/4$. In fact, we want the bounding line $\ell'$ with the maximal such normal direction. Then, $T^k\ell\cap\ell'\neq\emptyset$ for each $k\geq0$ and, in particular, the attractor is in the closed left half plane  of $F^k_d(\ell')$ for $d\in\mathcal{D}(\ell')$ with the largest argument. It follows that segments of $F_d^k(\ell)$ are edges of $\conv(A)$ and one extreme point of $A$ is the fixed point of $F_d$. 

    We started with the assumption that $\theta$ was the smallest outward normal direction greater than $1/4$. However, the above argument can be applied to any other bounding line at angle between $1/4$ and $3/4$. This means that edges of $\conv(A)$ in the upper half plane are eventually covered by images of these bounding lines under $F_d$. A symmetric argument, deals with the bounding lines at angles in $(3/4,1/4)$.

    Let now $\eta<0$. The argument is similar to the previous case except that now $t_{2k+1}$ tends to $1/4$ and $t_{2k}$ to $3/4$, so we have the fixed points of $F_{dd'}$ and $F_{d'd}$ as extreme points for some pair $d,d'\in\overline{\conv(\mathcal{D})}\setminus\conv(\mathcal{D})$.

    For both cases, $\eta>0$ and $\eta<0$, to find the equations for the DGIFS is just a matter of comparing the slope of the images of bounding lines until we end in a cycle of either period $1$ or period $2$, respectively.
    
\end{proof}

\subsection{Examples}
\label{subsect:jordanarbitraryexamples}
For the examples we reuse the setting from the real arbitrary case: let $v=\exp(2\pi\I(\sqrt{5}-1)/2)$ and we consider the set of translations \[\mathcal{D}=\{d_0=v^5,d_1=v^2,d_2=-0.2+v^7,d_3=v^4,d_4=0.5+v^3\}\] ordered by argument. Recall the edges of $\conv(\mathcal{D})$ are 
\[
\begin{split}
    \mathcal{D}(\ell_0)=\{d_0,d_1\}&\quad\mathcal{D}(\ell_1)=\{d_1,d_2\}\quad
    \mathcal{D}(\ell_2)=\{d_2,d_3\}\\ &\mathcal{D}(\ell_3)=\{d_3,d_4\}\quad\mathcal{D}(\ell_4)=\{d_4,d_0\}.    
\end{split}
\]

\subsubsection{Positive $\eta$}
\label{subsubsect:jordanarbitrarypositive} Let $\eta=0.25$. Following the argument in the proof of Theorem \ref{thm:mainjordanarbitrary}, we look for the bounding lines with normal direction between $1/4$ and $3/4$:  these are in order $\widehat\ell_{1}$, $\widehat\ell_{2}$, $\widehat\ell_{3}$. The bounding line $\widehat\ell_{0}$ is the one with largest normal direction $\leq1/4$. This means that an extreme point of $A$ is the fixed point of $F_{d_1}$ and that, eventually, the images of $\widehat\ell_{k}$ that contain edges of $\conv(A)$ are those under $F_{d_1}$. Similarly, the fixed point of $F_{d_4}$ is in $\ext(A)$. Let $p=F_{d_1}(p)$ and $q=F_{d_4}(q)$, then 
\begin{align*}
p-(d_1-d_0)=F_{d_0}(p),&\quad& F_{d_0}(p)-(d_0-d_4)=F_{d_4}(p),\\
q+(d_3-d_4)=F_{d_3}(q),&\quad& F_{d_3}(q)-(d_3-d_2)=F_{d_2}(q)
\end{align*}
are necessarily other extreme points because of the shape of $\conv(\mathcal{D})$. In fact, as the equations show we conclude
\begin{align*}
[F_{d_0}(p),p]\subset\widehat\ell_1,&\quad&[F_{d_4}(p),F_{d_0}(p)]\subset\widehat\ell_4,\\
\,[q,F_{d_3}(q)]\subset\widehat\ell_3,&\quad&[F_{d_3}(q),F_{d_2}(q)]\subset\widehat\ell_2
\end{align*}
are edges of $\conv(A)$. What is left to do now is iterate these edges and pay attention to their slope. The resulting system of equations is given in Figure \ref{fig:JordanArbPos}.

\begin{figure}[t]
    \centering
    \includegraphics[width=0.4\linewidth]{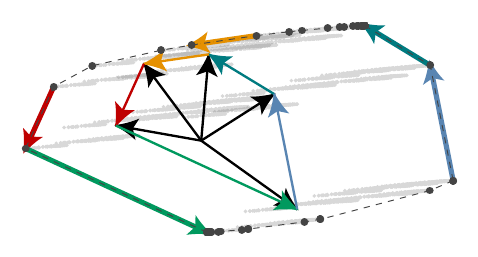}
    % \hspace{1cm}
    \begin{tikzpicture}
    
    \node [draw, circle] (30) at (30:4) (s0) {\small $B_{0}$};
    \node [draw, circle] (-30) at (-30:5) (s1) {\small $B_{1}$};
    \node [draw, circle] (265) at (265:4) (s2) {\small $B_{2}$};
    \node [draw, circle] (250) at (250:4) (s3) {\small $B_{3}$};
    \node [draw, circle] (210) at (210:4) (s4) {\small $B_{4}$};
    \node [draw, circle] (180) at (180:3) (s5) {\small $B_{5}$};
    \node [draw, circle] (150) at (150:4) (s6) {\small $B_{6}$};
    \node [draw, circle] (110) at (110:3) (s7) {\small $B_{7}$};
    \node [draw, circle] (90) at (90:4) (s8) {\small $B_{8}$};
    \node [draw, circle] (50) at (50:4) (s9) {\small $B_{9}$};
    
    \path[->] (s0) edge [loop left, left] node {$d_1$} (s0);
    \path[->] (s1) edge [bend right=0, right] node {$d_1$} (s0);
    \path[->] (s1) edge [bend right=30, right] node {$d_0$} (s0);
    \path[->] (s1) edge [bend right=60, right] node {$d_4$} (s0);
    
    \path[->] (s2) edge [right] node {$d_4$} (s1);
    \path[->] (s2) edge [loop above, above] node {$d_4$} (s3);
    
    \path[->] (s3) edge [loop above, above] node {$d_4$} (s3);
    
    \path[->] (s4) edge [below] node {$d_3$} (s3);
    
    \path[->] (s6) edge [left] node {$d_2$} (s4);
    
    \path[->] (s7) edge [left] node {$d_2$} (s5);
    
    \path[->] (s8) edge [above] node {$d_2$} (s6);
    \path[->] (s8) edge [bend right, above] node {$d_1$} (s6);
    
    \path[->] (s9) edge [loop right, right] node {$d_1$} (s9);
    \path[->] (s9) edge [above] node {$d_1$} (s8);
    \path[->] (s9) edge [above] node {$d_1$} (s7);
\end{tikzpicture}
    \caption{On the left, the elements of $\ext(A)$ of the IFS $\Psi_5$ with $\eta>0$. On the right, the DGIFS describing them.}
    \label{fig:JordanArbPos}
\end{figure}
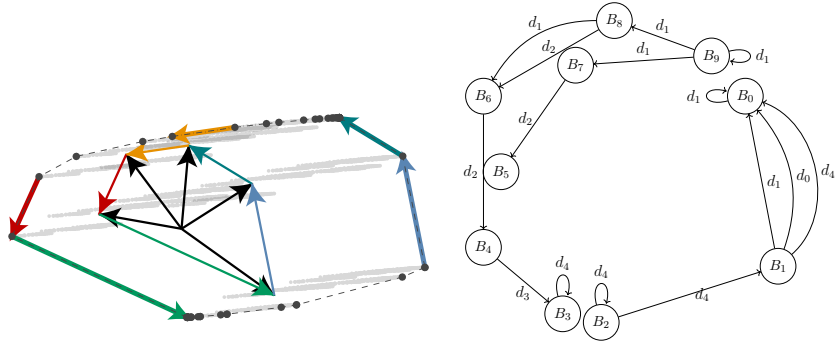

\subsubsection{Negative $\eta$}
\label{subsubsect:jordanarbitrarynegative}
Let $\eta=-0.25$. Here we have to look for a pair of periodic points of period $2$: given the slope of $\ell_1$ and $\ell_4$, the extreme point are $p=F_{d_1d_4}(p)$ and $q=F_{d_4d_1}(q)$. It follows that 
\[
\begin{split}
p+(d_2-d_1)=F_{d_2}(p),&\quad F_{d_2}(p)+(d_3-d_2)=F_{d_3}(p),\quad F_{d_3}(p)+(d_4-d_3)=F_{d_4}(p),\\
q+(d_0-d_4)=F_{d_0}(q),&\quad F_{d_0}(q)+(d_1-d_0)=F_{d_1}(q)
\end{split}
\]
are extreme points of $A$ and $[p,F_{d_2}(p)]\subset\ell_1$, $[F_{d_2}(p),F_{d_3}(p)]\subset\ell_2$, $[F_{d_3}(p),F_{d_4}(p)]\subset\ell_3$, $[q,F_{d_0}(q)]\subset\ell_4$, $[F_{d_0}(q),F_{d_1}(q)]\subset\ell_0$ are edges of $\conv(A)$. Iterating these we obtain the DGIFS in Figure \ref{fig:JordanArbNeg}. 
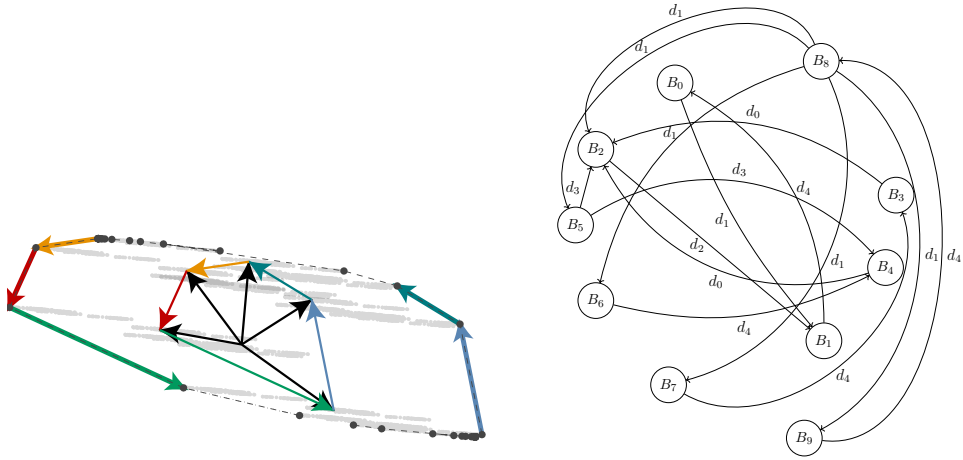
\begin{figure}[t]
    \centering
    \includegraphics[width=0.45\linewidth]{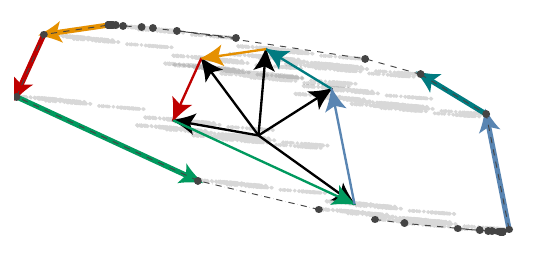}
    % \hspace{1cm}
    \begin{tikzpicture}
    
    \node [draw, circle] (110) at (110:4) (s0) {\small $B_{0}$};
    \node [draw, circle] (310) at (310:4) (s1) {\small $B_{1}$};
    \node [draw, circle] (150) at (150:4) (s2) {\small $B_{2}$};
    \node [draw, circle] (10) at (10:4.55) (s3) {\small $B_{3}$};
    \node [draw, circle] (345) at (345:4.35) (s4) {\small $B_{4}$};
    \node [draw, circle] (180) at (180:4) (s5) {\small $B_{5}$};
    \node [draw, circle] (210) at (210:4) (s6) {\small $B_{6}$};
    \node [draw, circle] (250) at (250:4.5) (s7) {\small $B_{7}$};
    \node [draw, circle] (60) at (60:5) (s8) {\small $B_{8}$};
    \node [draw, circle] (290) at (290:6) (s9) {\small $B_{9}$};
    
    \path[->] (s0) edge [bend right=10,left] node {$d_1$} (s1);
    \path[->] (s1) edge [bend right,right] node {$d_4$} (s0);
    
    \path[->] (s2) edge [left] node {$d_2$} (s1);
    
    \path[->] (s3) edge [bend right, above] node {$d_0$} (s2);
    
    \path[->] (s4) edge [bend left=40, below] node {$d_0$} (s2);
    
    \path[->] (s5) edge [left] node {$d_3$} (s2);
    \path[->] (s5) edge [bend left=40, above] node {$d_3$} (s4);
    
    \path[->] (s6) edge [bend right=20, below] node {$d_4$} (s4);
    
    \path[->] (s7) edge [bend right=70, below] node {$d_4$} (s3);
    
    \path[->] (s8) edge [bend right=90, above] node {$d_1$} (s2);
    \path[->] (s8) edge [bend right=80, above] node {$d_1$} (s5);
    \path[->] (s8) edge [bend right=30, above] node {$d_1$} (s6);
    \path[->] (s8) edge [bend left=50, right] node {$d_1$} (s7);
    \path[->] (s8) edge [bend left=60, right] node {$d_1$} (s9);
    
    \path[->] (s9) edge [bend right=95, right] node {$d_4$} (s8);
\end{tikzpicture}
    \caption{On the left, the elements of $\ext(A)$ of the IFS $\Psi_5$ with $\eta<0$. On the right, the DGIFS used to plot them.}
    \label{fig:JordanArbNeg}
\end{figure}

\section{Beyond the Homogeneous Setting}
\label{sect:future}
We conclude by considering some non-homogeneous IFS. These examples suggest that the results of the previous sections extend beyond the homogeneous setting, and the arguments required are surprisingly similar in spirit. We leave the verification that the given DGIFS describes the set of extreme points as a challenge to the reader, and hope these examples instigate further investigation into the non-homogeneous case. 

\subsection{Fibonacci IFS}
This is a one-parameter family of IFS studied in \cite{ST15}. Let $\lambda\in\mathbb{D}\setminus\{0\}$ then the IFS is defined by the similarities \[f_0(z)=\lambda z\quad\text{and}\quad f_1(z)=\lambda^2z+1.\]
For our example we fix $\lvert\lambda\rvert=2^{-1/2}$ and $\arg(\lambda)=2\pi/4$, see Figure \ref{fig:Fibonacci}. The attractor $A$ is a dendrite and the convex hull is a pentagon. The set of extreme points is given by \[\begin{cases}
    B_0=f_0(B_3)\\
    B_1=f_1(B_3)\\
    B_2=f_1(B_0)\cup f_0(B_1)\\
    B_3=f_0(B_2)
\end{cases}\]
\begin{figure}[!htb]
    \centering
    \includegraphics[width=0.4\linewidth]{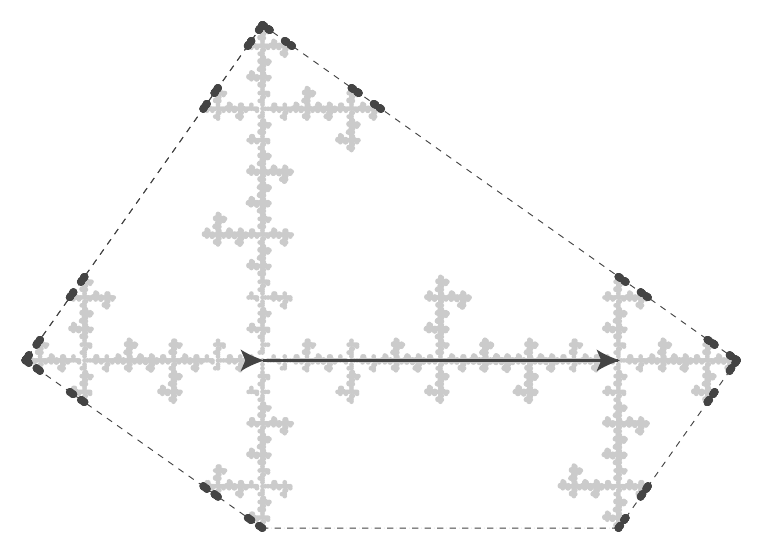}
    \caption{The attractor of a Fibonacci IFS with the points of $\ext(A)$ marked.}
    \label{fig:Fibonacci}
\end{figure}
We conjecture, based on experimental evidence and the result of \cite{DT05}, that a theorem similar to Theorem \ref{thm:maincplxpolygon} applies for this family of IFS and for the $N$-bonacci IFS $\{f_j(z)\}_{j=0}^{N-1}$ generalization where
\[f_0(z)=\lambda z,\quad f_k(z)=\lambda^{k+1}z+\sum_{j=0}^{k-1}\lambda^j\quad 1\leq k\leq N-1.\]

One thing to notice that differs from the homogeneous case is that the number of equations in the DGIFS is smaller than the number of sides of $\conv(A)$. The missing side however is nonetheless a bounding line since it is parallel to the translation vector of $f_1(z)$, which suggests that a slightly modified version of Proposition \ref{prop:extremeconvexcore} might hold for the non-homogeneous case.

\subsection{Symmetric Binary Tree IFS}
Here we consider another one-parameter family of IFS studied in \cite{BH89,MF99}. Let $\lambda\in\mathbb{D}\setminus\{0\}$ then the IFS is defined by the similarities \[f_0(z)=\lambda z+1\quad\text{and}\quad f_1(z)=\overline{\lambda}z-1.\]
If we fix $\lvert\lambda\rvert=2^{-0.5}$ and $\arg(\lambda)=2\pi/8$, then we obtain the famous L{\'e}vy curve \cite{L38}, see the top row of Figure \ref{fig:LevyKochCurve}. If we fix $\lvert\lambda\rvert=3^{-0.5}$ and $\arg(\lambda)=2\pi\cdot5/12$, then we obtain the famous Koch curve \cite{K04}, see the bottom row of Figure \ref{fig:LevyKochCurve}. It can be proved that when $\lambda$ is on the boundary of the connectedness locus of $A$, then the DGIFS describing $\ext(A)$ has finitely many equations: if $(2\pi)^{-1}\arg(\lambda)\in(1/(2k+2),1/(2k)]$ for $k\geq2$, then it has $2k$ equations; while if $(2\pi)^{-1}\arg(\lambda)\in(1/4,1/2)$ it has three.
\begin{figure}[!htb]
    \centering
    \includegraphics[width=0.4\linewidth]{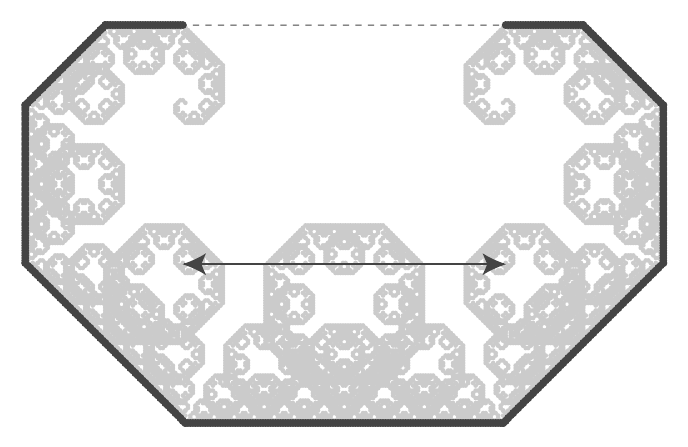}
    % \hspace{1cm}
    \begin{tikzpicture}
    
    \node [draw, circle] (90) at (90:2) (s0) {\small $B_0$};
    \node [draw, circle] (135) at (135:2) (s1) {\small $B_1$};
    \node [draw, circle] (45) at (45:2) (s2) {\small $B_2$};
    \node [draw, circle] (180) at (180:2) (s3) {\small $B_3$};
    \node [draw, circle] (0) at (0:2) (s4) {\small $B_4$};
    \node [draw, circle] (225) at (225:2) (s5) {\small $B_5$};
    \node [draw, circle] (270) at (270:2) (s6) {\small $B_6$};
    \node [draw, circle] (315) at (315:2) (s7) {\small $B_7$};
        
    \path[->] (s0) edge [midway, above] node {$1$} (s1);
    \path[->] (s0) edge [midway, above] node {$0$} (s2);
    \path[->] (s1) edge [midway, left] node {$1$} (s3);
    \path[->] (s2) edge [midway, right] node {$0$} (s4);
    \path[->] (s3) edge [midway, left] node {$1$} (s5);
    \path[->] (s4) edge [midway, right] node {$0$} (s7);
    \path[->] (s5) edge [bend right, below] node {$1$} (s6);
    \path[->] (s6) edge [bend right, above] node {$0$} (s5);
    \path[->] (s6) edge [bend left, above] node {$1$} (s7);
    \path[->] (s7) edge [bend left, below] node {$0$} (s6);
    
\end{tikzpicture}
    \includegraphics[width=0.45\linewidth]{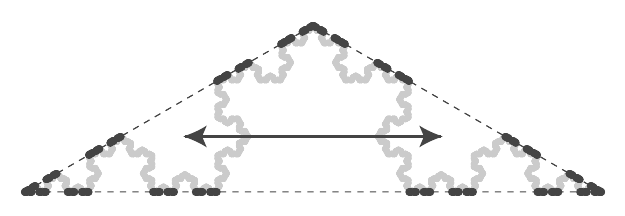}
    % \hspace{1cm}
    \begin{tikzpicture}
    
    \node [draw, circle] (225) at (225:2) (s0) {\small $B_0$};
    \node [draw, circle] (270) at (270:2) (s1) {\small $B_1$};
    \node [draw, circle] (315) at (315:2) (s2) {\small $B_2$};
        
    \path[->] (s0) edge [bend right, below] node {$1$} (s1);
    \path[->] (s1) edge [bend right, above] node {$0$} (s0);
    \path[->] (s1) edge [bend left, above] node {$1$} (s2);
    \path[->] (s2) edge [bend left, below] node {$0$} (s1);
    
\end{tikzpicture}
    \caption{The L{\'e}vy and the Koch curves with their respective extreme points. Their itinearies corresponds to infinite walks on the directed graphs shown on the right.}
    \label{fig:LevyKochCurve}
\end{figure}
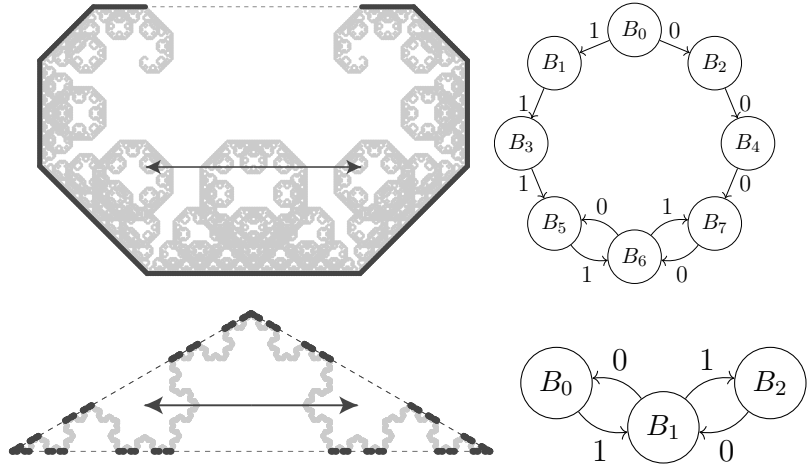

When $\lambda$ is in the interior of the connectedness locus of $A$, however, it gets more complicated for reasons we have yet to understand. It would be also interesting to extend the above discussion to the symmetric $N$-ary tree introduced in \cite{E13}.

\subsection{Kiesswetter curve}
The Kiesswetter curve \cite{K66} is an example of a continuous nowhere differentiable function on $[0,1]$. Edgar \cite{E89} showed that such curve is also the attractor of a non-homogeneous IFS $\{F_j(\vec{x})\}_{j=0}^3$ where
\begin{align*}
F_0(\vec{x})=\begin{pmatrix}
    1/4 &0\\0&-1/2
\end{pmatrix}\vec{x},\quad F_1(\vec{x})=\begin{pmatrix}
    1/4 &0\\0&1/2
\end{pmatrix}\vec{x}+\begin{pmatrix}
    1/4\\-1/2
\end{pmatrix},
\\ F_2(\vec{x})=\begin{pmatrix}
    1/4 &0\\0&1/2
\end{pmatrix}\vec{x}+\begin{pmatrix}
    1/2\\0
\end{pmatrix},\quad  F_3(\vec{x})=\begin{pmatrix}
    1/4 &0\\0&1/2
\end{pmatrix}\vec{x}+\begin{pmatrix}
    3/4\\1/2
\end{pmatrix}.    
\end{align*}
The attractor $A$ lives in the rectangle $[0,1]\times[-1,1]$. Note that the functions $F_1,F_2,$ and $F_3$ share the same linear factor and their translation vectors are collinear. The convex hull is not a finite polygon, but $\ext(A)$ can be described by a DGIFS with finitely many equations, see Figure \ref{fig:Kiesswetter}.
% \[\begin{cases}
%     B_0=F_1(B_0)\\
%     B_1=F_1(B_0)\cup F_2(B_0)\cup F_3(B_1)\\
%     B_2=F_0(B_3)\\
%     B_3=F_1(B_0)\cup F_0(B_2)
% \end{cases}.\]
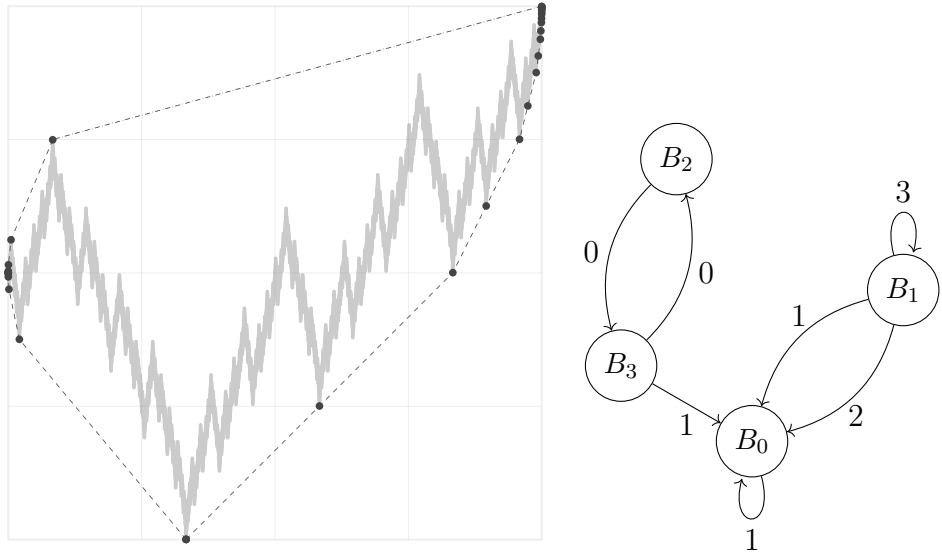
\begin{figure}[!htb]
    \centering
    \includegraphics[width=0.5\linewidth]{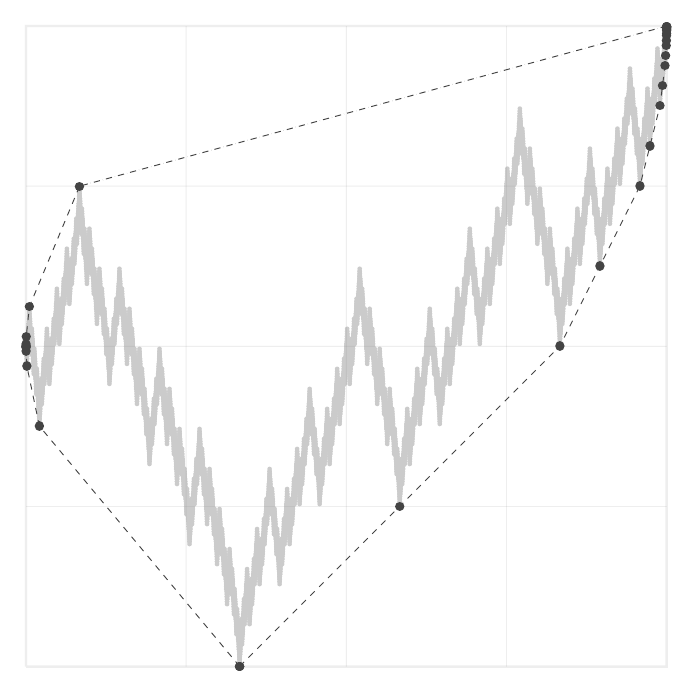}
    % \hspace{1cm}
    \begin{tikzpicture}
    
    \node [draw, circle] (270) at (270:2) (s0) {\small $B_0$};
    \node [draw, circle] (0) at (0:2) (s1) {\small $B_1$};
    \node [draw, circle] (120) at (120:2) (s2) {\small $B_2$};
    \node [draw, circle] (210) at (210:2) (s3) {\small $B_3$};
        
    \path (s0) edge [loop below] node {$1$} (s0);
    \path (s1) edge [loop above] node {$3$} (s1);
    \path[->] (s1) edge [bend right, above] node {$1$} (s0);
    \path[->] (s1) edge [bend left, below] node {$2$} (s0);
    \path[->] (s2) edge [bend right, left] node {$0$} (s3);
    \path[->] (s3) edge [bend right, right] node {$0$} (s2);
    \path[->] (s3) edge [midway, below] node {$1$} (s0);
    
\end{tikzpicture}
    \caption{The Kiesswetter curve and its extreme points inside the rectangle $[0,1]\times[-1,1]$. The vertical lines of the grid are at $x=k/4$ for $0\leq k\leq 4$ while the horizontal lines are at $y=j/2$ for $-2\leq j\leq2$. Each infinite path on the directed graph on the right represents an extreme point.}
    \label{fig:Kiesswetter}
\end{figure}
In \cite{LiMiao14} Li and Miao showed how it is possible to generalize Kiesswetter's function and even provided a representation via non-homogeneous affine IFS. Can the set of extreme points for this functions be described with a DGIFS as above? 

\subsection{Takagi curves}
Another example of a simple continuous nowhere differentiable function was given by Takagi \cite{T03} in 1903. It is defined on the unit interval as \[T(x)=\sum_{n\geq0}\frac{1}{2^n}\phi(2^nx)\] where $\phi(x)$ is the distance of $x$ to the nearest integer. There exists various ways to extend its definition, but the one we consider here is \[T_\eta(x)=\sum_{n\geq0}\eta^n\phi(2^nx)\] for $\eta\in(0.25,1)$. The graph $T_\eta$ can also be realized (see for example \cite{ABK25}) as the attractor of the affine IFS $\{F_0,F_1\}$ where \[F_0(\vec{x})=\begin{pmatrix}1/2&0\\1/2&\eta\end{pmatrix}\vec{x},\quad F_1(\vec{x})=\begin{pmatrix}1/2&0\\-1/2&\eta\end{pmatrix}\vec{x}+\begin{pmatrix}1/2\\1/2\end{pmatrix}.\]
Given that the global extrema for $T_\eta$ have a simple description when $0.5\leq \eta<1$, the sofic system for the extreme points is straightforward, see Figure \ref{fig:TakagiCurves}. We have found that the DGIFS when $0.25<\eta<0.5$ is not constant for each value of $\eta$ and not easily determined.

We remark that the generalized Takagi curve $T_\eta$ is actually defined for non zero $-1<\eta<1$, however the cases $\eta\leq0.25$ turn out to be more difficult than expected and, thus, are left for further investigation.

Also worth mentioning is that the attractor of the IFS $\{F_0,F_1\}$ is not geometrically similar to that of the IFS $\{G_0,G_1\}$ where the contracting matrix of $G_j$ is the Jordan form of that in $F_j$.
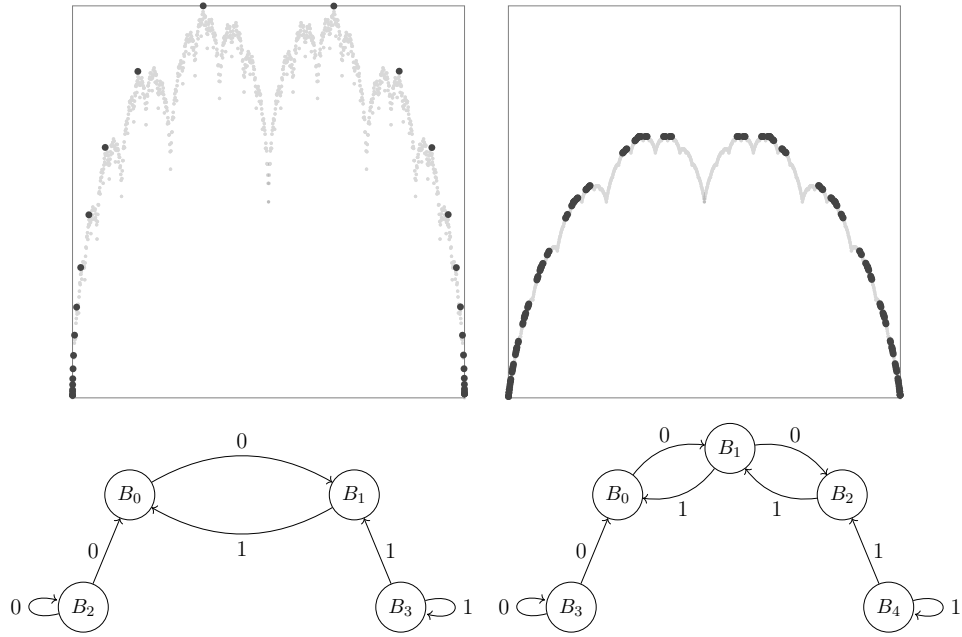
\begin{figure}
    \centering
    \includegraphics[scale=0.7]{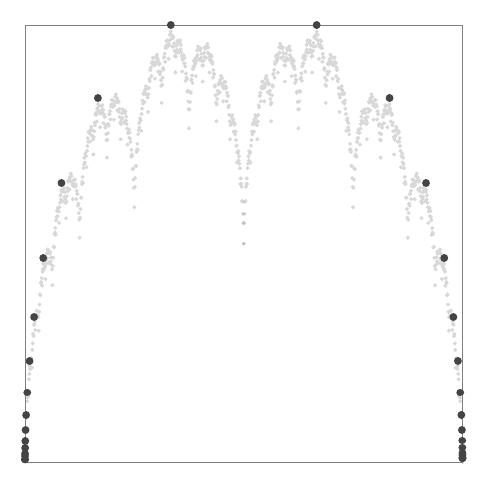}%
    \includegraphics[scale=0.7]{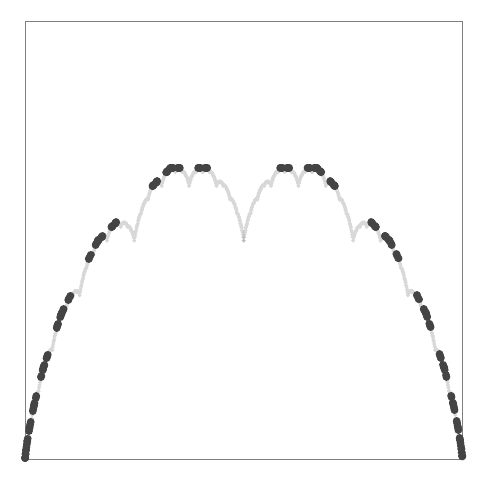}
    
    \begin{tikzpicture}
    
    \node [draw, circle] (135) at (135:3) (s0) {\small $B_0$};
    \node [draw, circle] (45) at (45:3) (s1) {\small $B_1$};
    \node [draw, circle] (180) at (180:3) (s2) {\small $B_2$};
    \node [draw, circle] (0) at (0:3) (s3) {\small $B_3$};
        
    \path[->] (s0) edge [bend left,midway, above] node {$0$} (s1);
    \path[->] (s1) edge [bend left,midway, below] node {$1$} (s0);
    \path[->] (s2) edge [midway, left] node {$0$} (s0);
    \path[->] (s2) edge [loop left, left] node {$0$} (s2);
    \path[->] (s3) edge [midway, right] node {$1$} (s1);
    \path[->] (s3) edge [loop right, right] node {$1$} (s3);
    
\end{tikzpicture}
    \begin{tikzpicture}
    
    \node [draw, circle] (135) at (135:3) (s0) {\small $B_0$};
    \node [draw, circle] (90) at (90:3) (s1) {\small $B_1$};
    \node [draw, circle] (45) at (45:3) (s2) {\small $B_2$};
    \node [draw, circle] (180) at (180:3) (s3) {\small $B_3$};
    \node [draw, circle] (0) at (0:3) (s4) {\small $B_4$};
    
    \path[->] (s0) edge [bend left,midway, above] node {$0$} (s1);
    \path[->] (s1) edge [bend left,midway, below] node {$1$} (s0);
    \path[->] (s1) edge [bend left,midway, above] node {$0$} (s2);
    \path[->] (s2) edge [bend left,midway, below] node {$1$} (s1);
    \path[->] (s3) edge [left] node {$0$} (s0);
    \path[->] (s3) edge [loop left, left] node {$0$} (s3);
    \path[->] (s4) edge [right] node {$1$} (s2);
    \path[->] (s4) edge [loop right, right] node {$1$} (s4);
    
\end{tikzpicture}
    \caption{The top shows the attractor of $T_{\eta}$ in the unit square with its extreme points for $\eta=2/3$, on the left, and $\eta=1/2$, on the right. The bottom row shows the respective DGIFS for $\ext(A)$.}
    \label{fig:TakagiCurves}
\end{figure}

\subsection{Generalized PU-sets}
Here we consider the so-called Przytycki--Urba\'nski sets, introduced in \cite{PU89}, and their generalizations. The PU-sets are the limit sets of a one-parameter family of homogeneous self-affine IFS $\{F_0,F_1\}$ where
\[ F_0(\vec{x})=\begin{pmatrix}
    \mu&0\\0&1/2
\end{pmatrix}\vec{x},\quad F_1(\vec{x})=\begin{pmatrix}
    \mu&0\\0&1/2
\end{pmatrix}\vec{x}+\begin{pmatrix}
    1-\mu\\1/2
\end{pmatrix}\]
with $\mu\in(1/2,1)$. The attractor $A$ lives in the unit square and its extreme points can be found by applying Theorem \ref{thm:mainrealcollinear}:
\[\begin{cases}
    B_0=F_0(B_0)\\
    B_1=F_0(B_0)\cup F_1(B_1)\\
    B_2=F_0(B_2)\cup F_1(B_3)\\
    B_3=F_1(B_3)
\end{cases}.\]
An example is shown in Figure \ref{fig:PUsets} where $\mu=2/3$.

The generalized PU-sets, studied in \cite{BU90,PW94,N01}, are the attractors of a four-parameters family of self-affine IFS $\{F_0,F_1\}$ where
\[ F_0(\vec{x})=\begin{pmatrix}
    \mu_1&0\\0&\eta_1
\end{pmatrix}\vec{x},\quad F_1(\vec{x})=\begin{pmatrix}
    \mu_2&0\\0&\eta_2
\end{pmatrix}\vec{x}+\begin{pmatrix}
    1-\mu_2\\1-\eta_2
\end{pmatrix}\] where $0<\eta_j<\mu_j<1$ for $j=0,1$ and also $\mu_1+\mu_2>1$, $\eta_1+\eta_2\leq1$. The attractor $A$ for this IFS also lives in the unit square.
\begin{figure}[!htb]
    \centering
    \includegraphics[width=0.385\linewidth]{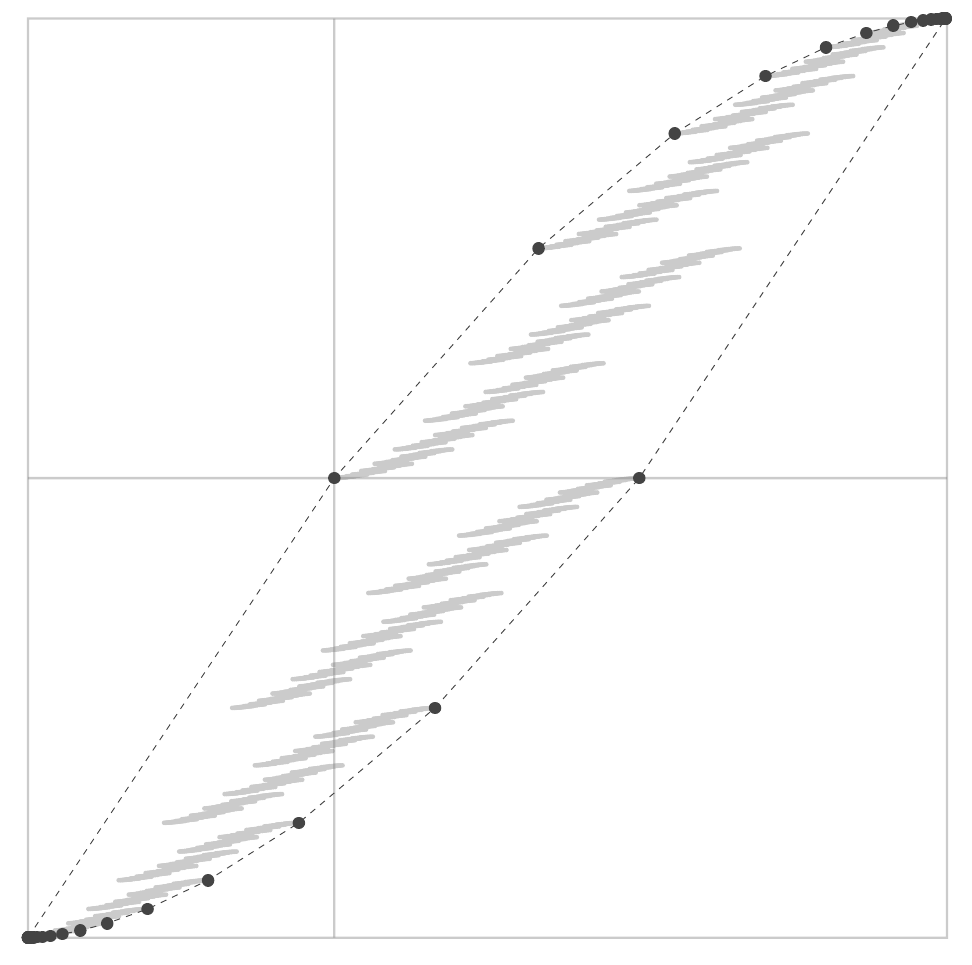}
    % \hspace{1cm}
    \includegraphics[width=0.385\linewidth]{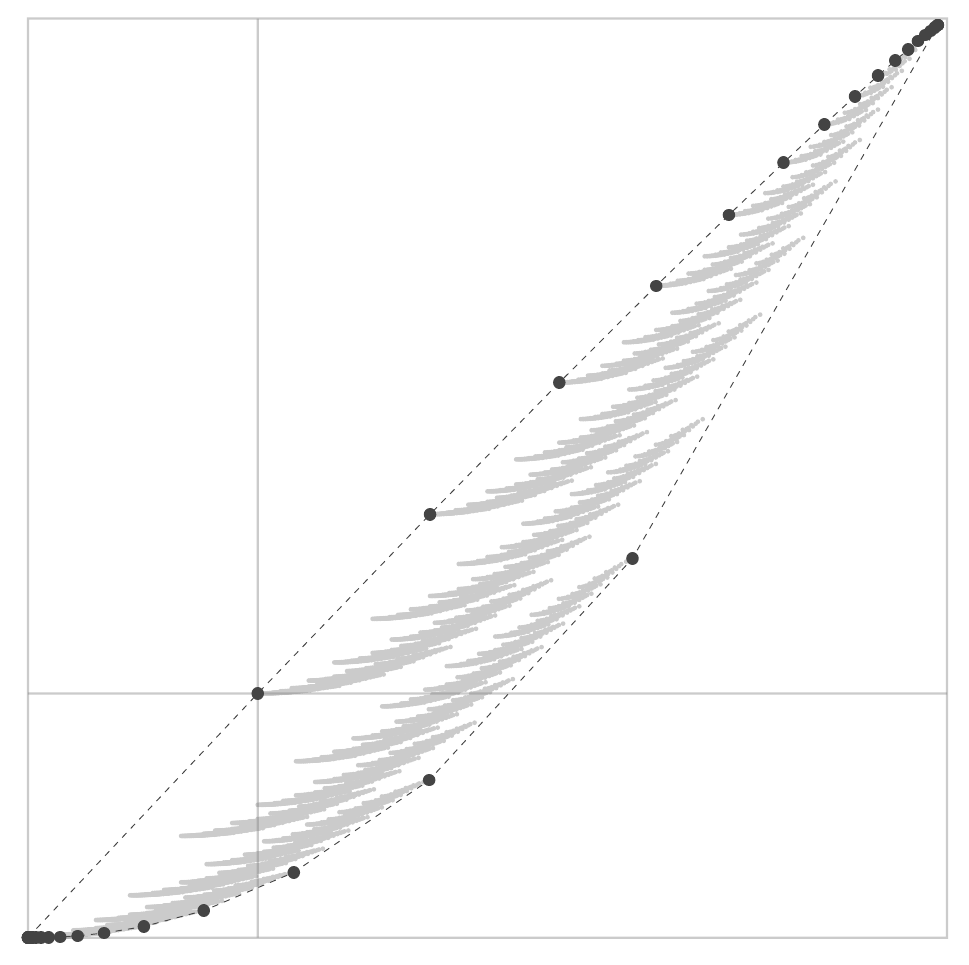}
    \caption{A PU-set and a generalized PU-set with their extreme points highlighted. Both attractors live in the unit square. The vertical line in the left picture is at $x=1-\mu$ and the horizontal line is at $y=1/2$. For the right picture the lines are at $x=1-\mu_2$ and $y=1-\eta_2$.}
    \label{fig:PUsets}
\end{figure}
Despite being non-homogeneous, the set of extreme points of $A$ is described by the same DGIFS used for the PU-sets. The attractor on the right of Figure \ref{fig:PUsets} is obtained by setting values $\mu_1=2/3$, $\mu_2=3/4$, $\eta_1=\mu_1-1/4$, and $\eta_2=\mu_2-1/64$. 

\newpage
\bibliographystyle{plainurl}
\bibliography{references}

@article{ABK25,
    title={Slices of the {T}akagi function},
    author={Anttila, Roope and B{\'a}r{\'a}ny, Bal{\'a}zs and K{\"a}enm{\"a}ki, Antti},
    journal ={Ergodic Theory Dynam. Systems},
    fjournal ={Ergodic Theory and Dynamical Systems},
    volume={44},
    number={9},
    pages={2361--2398},
    year={2024},
	zblnumber={1568.28009},
    MRNUMBER ={4802758},
    doi={10.1017/etds.2023.117},
    url={https://doi.org/10.1017/etds.2023.117}
}

@article{B89a,
    title={Self-similar sets {I} {T}opological {M}arkov chains and mixed self-similar sets},
    author={Bandt, Cristoph},
    journal={Math. Nachr.},
    fjournal={Mathematische Nachrichten},
    volume={142},
    number={},
    pages={107--123},
    year={1989},
	zblnumber={0707.28004},
    MRNUMBER ={1017373},
    doi={10.1002/mana.19891420107},
    url={https://doi.org/10.1002/mana.19891420107}
}

@article{B89b,
    title={Self-similar sets {III} {C}onstructions with sofic systems},
    author={Bandt, Cristoph},
    journal={Monatsh. Math.},
    fjournal={Monatshefte f{\"u}r Mathematik},
    volume={108},
    number={2},
    pages={89--102},
    year={1989},
    publisher={Springer-Verlag},
	zblnumber={0712.58039},
    MRNUMBER ={1026611},
    doi={10.1007/BF01308664},
    url={https://doi.org/10.1007/BF01308664}
}

@article{BG94,
    title={Classification of Self-Affine Lattice Tilings},
    author={Bandt, Cristoph and Gelbrich, G{\"o}tz},
    journal={J. London Math. Soc. (2)},
    fjournal={Journal of the London Mathematical Society. Second Series},
    volume={50},
    number={3},
    pages={581--593},
    year={1994},
	zblnumber={0820.52012},
    MRNUMBER ={1299459},
    doi={10.1112/jlms/50.3.581},
    url={https://doi.org/10.1112/jlms/50.3.581}
}

@article{BH89,
    title={A {M}andelbrot set whose boundary is piecewise smooth},
    author={Barnsley, Michael F. and Hardin, Doug P.},
    journal={Trans. Amer. Math. Soc.},
    fjournal={Transactions American Mathematical Society},
    volume={315},
    number={2},
    pages={641--659},
    year={1989},
    MRNUMBER ={1011232},
    doi={10.1090/S0002-9947-1989-1011232-6},
    url={https://doi.org/10.1090/S0002-9947-1989-1011232-6}
}

@phdthesis{B84,
    author ={Bedford, Timothy J.},
    title ={Crinkly curves, {M}arkov partitions, and dimension},
    school ={Warwick University},
    year ={1984}
}

@article{BU90,
    title={The box and {H}ausdorff dimension of self-affine sets},
    author={Bedford, Tim and Urba{\'n}ski, Mariusz},
    journal ={Ergodic Theory Dynam. Systems},
    fjournal ={Ergodic Theory and Dynamical Systems},
    volume={10},
    number={4},
    pages={627--644},
    year={1990},
	zblnumber={0725.28006},
    MRNUMBER ={1091418},
    doi={10.1017/S0143385700005812},
    url={https://doi.org/10.1017/S0143385700005812}
}

@incollection{B92,
    author={Berger, Marc A.},
    title={Random affine iterated function systems: mixing and encoding},
    booktitle={Diffusion processes and related problems in analysis,{V}ol.{II} ({C}harlotte,{NC}, 1990)},
    series={Progr. Probab.},
    volume={27},
    pages={315--346},
    publisher={Birkh{\"a}user Boston, Boston, MA},
    year={1992},
    MRNUMBER ={1187998}
}

@article{CW19,
    title={Extreme points in limit sets},
    author={Calegari, Danny and Walker, Alden},
    journal={Proc. Amer. Math. Soc.},
    fjournal={Proceedings American Mathematical Society},
    volume={147},
    number={09},
    pages={3829--3837},
    year={2019},
	zblnumber={1428.37043},
    MRNUMBER ={3993775},
    doi={10.1090/proc/14656},
    url={https://doi.org/10.1090/proc/14656}
}

@article{DT05,
    title={On convex hulls of self-similar sets},
    author={Davydkin, Ivan and Tetenov, Andrei},
    journal={Vestn. Novosib. Gos. Univ., Ser. Mat. Mekh. Inform.},
    fjournal={Vestnik Novosibirskogo Gosudarstvennogo Universiteta. Seriya Matematika, Mekhanika, Informatika},
    volume={5},
    number={2},
    pages={21--27},
    year={2005},
    language={Russian}
}

@article{D82,
    title={Recurrent sets},
    author={Dekking, F.M.},
    journal={Adv. in Math.},
    fjournal={Advances in Mathematics},
    volume={44},
    number={1},
    pages={78--104},
    year={1982},
    MRNUMBER ={654549},
    doi={10.1016/0001-8708(82)90066-4},
    url={https://doi.org/10.1016/0001-8708(82)90066-4}
}

@book{Edgar1990,
    title={Measure, topology, and fractal geometry},
    author={Edgar, Gerald A.},
    series={Undergraduate Texts in Mathematics},
    publisher={Springer-Verlag, New York},
    pages={xiv+230},
    MRNUMBER ={1065392},
    year={1990},
    doi={10.1007/978-1-4757-4134-6},
    url={https://doi.org/10.1007/978-1-4757-4134-6}
}

@article{E89,
    title={Kiesswetter's fractal has {H}ausdorff dimension {$3/2$}},
    author={Edgar, Gerald A.},
    journal={Real Anal. Exchange},
    fjournal={Real Analysis Exchange},
    volume={14},
    number={1},
    pages={215--223},
    MRNUMBER ={988367},
    year={1988/89}
}

@article{E13,
    title ={Generalized Self-Contacting Symmetric Fractal Trees},
    author ={Espigul{\'e}, Bernat},
    journal ={Symmetry Culture and Science},
    volume ={21},
    number ={1--4},
    pages ={333--351},
    year ={2013}
}

@article{EJS24,
    title ={Collinear Fractals and {B}andt's Conjecture},
    author ={Espigul{\'e}, Bernat and Juher, David and Salda{\~n}a, Joan},
    journal ={Fractal and Fractional},
    volume ={8},
    number ={12},
    pages ={},
    year ={2024},
    doi={10.3390/fractalfract8120725},
    url={https://doi.org/10.3390/fractalfract8120725}
}

@article{G87,
    title={Complex bases and fractal similarity},
    author={Gilbert, William J.},
    journal={Ann. Sci. Math. Qu\'{e}bec},
    fjournal={Annales des Sciences Math\'{e}matiques du Qu\'{e}bec},
    volume={11},
    number={1},
    pages={65--77},
    year={1987},
    MRNUMBER={912163},
    zblnumber={0633.10008}
}

@inproceedings{G08,
  title={Simultaneous and hybrid beta-encodings},
  author={G{\"u}nt{\"u}rk, Sinan C.},
  booktitle={Annual Conference on Information Sciences and Systems},
  year={2008},
  url={https://api.semanticscholar.org/CorpusID:8905747}
}

@article{HS16,
    title={Two-dimensional self-affine sets with interior points, and the set of uniqueness},
    author={Hare, Kevin G. and Sidorov, Nikita},
    journal={Nonlinearity},
    fjournal={Nonlinearity},
    volume={29},
    number={1},
    pages={1--26},
    year={2016},
	zblnumber={1334.28019},
    MRNUMBER ={3460748},
    doi={10.1088/0951-7715/29/1/1},
    url={https://doi.org/10.1088/0951-7715/29/1/1}
}

@article{HS17,
    title={On a family of self-affine sets: Topology, uniqueness, simultaneous expansions},
    author={Hare, Kevin G. and Sidorov, Nikita},
    journal ={Ergodic Theory Dynam. Systems},
    fjournal ={Ergodic Theory and Dynamical Systems},
    volume={37},
    number={1},
    pages={193--227},
    year={2017},
	zblnumber={1378.37021},
    MRNUMBER ={3590500},
    doi={10.1017/etds.2015.41},
    url={https://doi.org/10.1017/etds.2015.41}
}

@article{H81,
  title={Fractals and self similarity},
  author={Hutchinson, John E},
  journal={Indiana Univ. Math. J.},
  fjournal={Indiana University Mathematics Journal},
  volume={30},
  number={5},
  pages={713--747},
  year={1981},
  zblnumber={0598.28011},
  MRNUMBER ={625600},
  doi={10.1512/iumj.1981.30.30055},
  url={https://doi.org/10.1512/iumj.1981.30.30055}
}

@article{K66,
    title={Ein enfaches{B}eispiel f{\"u}r eine Funktion, welche{\"u}berall stetig und nicht differenzierbar ist},
    author={Kiesswetter, Karl},
    journal={Math.-Phys. Semesterber.},
    fjournal={Mathematisch-physikalische Semesterberichte},
    volume={13},
    pages={216--221},
    year={1966},
    MRNUMBER ={202936},
    language={German}
}

@article{K04,
    title={Sur une courbe continue sans tangente obtenue par une construction g{\'e}om{\'e}trique {\'e}l{\'e}mentaire},
    author={Koch, Helge von},
    journal={Ark. Mat. Astron. Fys.},
    fjournal={Arkiv f{\"o}r Matematik, Astronomi och Fysik},
    volume={1},
    pages={681-702},
    year={1904}
}

@article{L38,
    title={Les courbes planes ou gauches et les surfaces compos{\'e}es de parties semblabes au tout},
    author={L{\'e}vy, Paul},
    journal={J.{\'E}c. Polytech., Math.},
    fjournal={Journal de l'{\'E}cole Polytechnique},
    volume={81},
    pages={227--247,249--291},
    year={1938}
}

@article{LiMiao14,
    title ={Generalized {K}iesswetter’s Functions},
    author ={Li, Delong and Miao, Jie},
    journal ={Real Anal. Exchange},
    fjournal ={Real Analysis Exchange},
    volume ={40},
    number ={1},
    pages ={141--156},
    year ={2014/15},
    MRNUMBER ={3365395},
    url ={http://projecteuclid.org/euclid.rae/1435759200}
}

@book{Mandelbrot83,
    title={The Fractal Geometry of Nature},
    author={Mandelbrot, Benoit B.},
    publisher={W.H. Freeman and Co.},
    year={1983} 
}

@article{MF99,
    author={Mandelbrot, Benoit B. and Frame, Michael},
    title={The Canopy and Shortest Path in a Self-Contacting Fractal Tree},
    journal={Math. Intelligencer},
    fjournal={The Mathematical Intelligencer},
    volume={21},
    pages={18--27},
    year={1999},
	zblnumber={1052.28500},
    MRNUMBER ={1684366},
    doi={10.1007/BF03024842},
    url={https://doi.org/10.1007/BF03024842}
}

@article{MW88,
    title ={Hausdorff dimension in graph directed constructions},
    author ={Mauldin, R.D. and Williams, S.C.},
    journal ={Trans. Amer. Math. Soc.},
    fjournal ={Transactions of the American Mathematical Society},
    volume ={309},
    number ={2},
    pages ={811--829},
    year ={1988},
    MRNUMBER ={961615},
    doi ={10.1090/S0002-9947-1988-0961615-4},
    url ={https://doi.org/10.1090/S0002-9947-1988-0961615-4}
}

@article{M84,
    title={The {H}ausdorff dimension of general {S}ierpi{\'n}ski},
    author={McMullen, Curt},
    journal={Nagoya Math. J.},
    fjournal={Nagoya Mathematical Journal},
    volume={96},
    pages={1--9},
    year={1984},
	zblnumber={0539.28003},
    MRNUMBER ={771063},
    doi={10.1017/S0027763000021085},
    url={https://doi.org/10.1017/S0027763000021085}
}

@article{N01,
    title={Properties of Some Overlapping Self-Similar and Some Self-Affine Measures},
    author={Neunh{\"a}userer, J{\"o}rg},
    journal={Acta Math. Hungar.},
    fjournal={Acta Mathematica Hungarica},
    volume={92}, 
    number={1-2}, 
    pages={143--161},
    year={2001},
	zblnumber={1002.28011},
    MRNUMBER ={1924256},
    doi={10.1023/A:1013716430425},
    url={https://doi.org/10.1023/A:1013716430425}
}

@article{PW94,
    title={The dimensions of some self-affine limit sets in the plane and hyperbolic sets},
    author={Pollicott, Mark and Weiss, Howard},
    journal={J. Statist. Phys.},
    fjournal={Journal of Statistical Physics},
    volume={77},
    number={3-4},
    pages={841--866},
    year={1994},
	zblnumber={0840.58027},
    MRNUMBER ={1301464},
    doi={10.1007/BF02179463},
    url={https://doi.org/10.1007/BF02179463}
}

@article{PU89,
    title={On the {H}ausdorff dimension of some fractal sets},
    author={Przytycki, Feliks and Urba{\'n}ski, Mariusz},
    journal={Studia Math.},
    fjournal={Polska Akademia Nauk. Instytut Matematyczny. Studia Mathematica},
    volume={93},
    number={2},
    pages={155--186},
    year={1989},
	zblnumber={0691.58029},
    MRNUMBER ={1002918},
    doi={10.4064/sm-93-2-155-186},
    url={https://doi.org/10.4064/sm-93-2-155-186}
}

@mastersthesis{R25,
    title={Almost Regular Closedness of the Connectedness Locus for Pairs of Affine Maps on {$\mathbb{R}^2$}},
    author={Rosler, Omer},
    year={2025},
    school={Bar-Ilan University},
    url={https://arxiv.org/abs/2411.06953}
}

@article{ST15,
    title ={On the {F}ibonacci–{M}andelbrot set},
    author ={Sirvent, V{\'i}ctor F. and Thuswaldner, J{\"o}rg M.},
    journal ={Indag. Math. (N.S.)},
    fjournal={Koninklijke Nederlandse Akademie van Wetenschappen. Indagationes Mathematicae. New Series},
    volume ={26},
    number ={1},
    pages ={174-190},
    year ={2015},
	zblnumber={1317.37053},
    MRNUMBER ={3281699},
    doi ={10.1016/j.indag.2014.09.004},
    url ={https://doi.org/10.1016/j.indag.2014.09.004}
}

@article{S94,
    author={Solomyak, Boris},
    title={Conjugates of beta-numbers and the zero-free domain for a class of analytic functions},
    journal={Proc. London Math. Soc. (3)},
    fjournal={Proceedings of the London Mathematical Society. Third Series},
    volume={68},
    year={1994},
    issue ={3},
    pages ={477--498},
	zblnumber={0820.30007},
    MRNUMBER ={1262305},
    doi ={10.1112/plms/s3-68.3.477},
    url ={https://doi.org/10.1112/plms/s3-68.3.477}
}

@article{T03,
    title={A simple example of the continuous function without derivative},
    author={Takagi, Teiji},
    journal={Proc. Phys.-Math. Soc. Japan},
    fjournal={Proceedings of the Physico-Mathematical Society of Japan},
    volume={1},
    series={II},
    pages={176--177},
    year={1903},
    doi={10.1007/978-4-431-54995-6_3},
    url={https://doi.org/10.1007/978-4-431-54995-6_3}
}

@article{TR18,
    title={Convex Hulls of {S}ierpi{\'n}ski relatives},
    author={Taylor, Tara D. and Rowley, S},
    journal={Fractals},
    fjournal={Fractals. Complex Geometry, Patterns, and Scaling in Nature and Society},
    volume={26},
    number={6},
    pages={1850098, 15},
    year={2018},
	zblnumber={1433.28030},
    MRNUMBER ={3900103},
    doi={10.1142/S0218348X18500986},
    url={https://doi.org/10.1142/S0218348X18500986}
}

@unpublished{TD02,
    title={Hausdorff dimension of the set of extreme points of a self-similar set},
    author={Tetenov, Andrei and Davydkin, Ivan},
    note={arXiV:math/0202215},
    url={https://arxiv.org/abs/math/0202215},
    year={2002}
}

@article{T17,
    title={Generalized {$\beta$}-transformations  and  the  entropy  of  unimodal  maps},
    author={Thompson, Daniel J.},
    journal={Comment. Math. Helv.},
    fjournal={Commentarii Mathematici Helvetici},
    volume={92},
    number={4},
    pages={777--800},
    year={2017},
	zblnumber={1380.37080},
    MRNUMBER ={3718487},
    doi={10.4171/CMH/424},
    url={https://doi.org/10.4171/CMH/424}
}

@article{V18,
    author={Vass, J{\'o}zsef},
    title={On the Exact Convex Hull of {IFS} fractals},
    journal={Fractals},
    fjournal={Fractals. Complex Geometry, Patterns, and Scaling in Nature and Society},
    volume={26},
    issue={01},
    pages={1850002, 21},
    year={2018},
	zblnumber={1432.28012},
    MRNUMBER ={3766077},
    doi={10.1142/S0218348X18500020},
    url={https://doi.org/10.1142/S0218348X18500020}
}

\end{document}